\documentclass[11pt,reqno]{amsart}

\usepackage{amssymb,latexsym}
\usepackage{amscd,amsthm,bm,mathrsfs}
\usepackage[all]{xy}
\usepackage{tikz}
\usetikzlibrary{matrix,arrows}
\usepackage{tikz-cd}
\usepackage{url}
\usepackage[T1]{fontenc}
\usepackage{textcomp}
\usepackage{palatino,helvet}
\usepackage[hang,flushmargin]{footmisc} 
\usepackage[numbers]{natbib}

\newtheorem{theorem}{Theorem}[section]

\newtheorem{lemma}[theorem]{Lemma}
\newtheorem{proposition}[theorem]{Proposition}
\newtheorem{corollary}[theorem]{Corollary}

\theoremstyle{definition}
\newtheorem{definition}[theorem]{Definition}

\newtheorem{remark}[theorem]{Remark}
\newtheorem{notation}[theorem]{Notation}
\newtheorem{example}[theorem]{Example}

\DeclareMathOperator{\Ext}{Ext}
\DeclareMathOperator{\Hom}{Hom}

\DeclareMathOperator{\Ch}{Ch}

\newcommand{\cat}[1]{\mathcal{#1}}           

\newcommand{\class}[1]{\mathcal{#1}}   

\newcommand{\FP}[1]{\mathcal{FP}_{#1}}

\newcommand{\tFP}[1]{$\text{FP}_{#1}$}   

\newcommand{\ch}{\textnormal{Ch}(R)}

\begin{document}

\title[Locally type $\textrm{FP}_n$ and $n$-coherent categories]{Locally type $\text{FP}_n$ and $\bm{n}$-coherent categories}

\author{Daniel Bravo}
\address{
Universidad Austral de Chile. \\
Facultad de Ciencias. \\
Instituto de Ciencias F\'isicas y Matem\'aticas. \\
Valdivia, Regi\'on de los R\'ios. CHILE.
}
\email[Daniel Bravo]{daniel.bravo@uach.cl}
\urladdr{http://icfm.uach.cl/personas/daniel-bravo.php}

\author{James Gillespie}
\address{
Ramapo College of New Jersey. \\
School of Theoretical and Applied Science. \\
505 Ramapo Valley Road. \\
Mahwah, NJ 07430. USA. 
}
\email[Jim Gillespie]{jgillesp@ramapo.edu}
\urladdr{http://pages.ramapo.edu/~jgillesp/}

\author{Marco A. P\'erez}
\address{
Universidad de la Rep\'ublica. \\
Facultad de Ingenier\'ia. \\
Instituto de Matem\'atica y Estad\'istica ``Prof. Ing. Rafael Laguardia''. \\
Julio Herrera y Reissig 565. \\
Montevideo, CP11300. URUGUAY. 
}
\email[Marco A. P\'erez]{mperez@fing.edu.uy}
\urladdr{http://maperez.net/}



\date{\today}

\baselineskip=14pt

\begin{abstract}
We study finiteness conditions in Grothendieck categories by introducing the concepts of objects of type \tFP{n} and studying their closure properties with respect to short exact sequences. This allows us to propose a notion of locally type $\text{FP}_n$ categories as a generalization of locally finitely generated and locally finitely presented categories. We also define and study the injective objects that are Ext-orthogonal to the class of objects of type \tFP{n}, called $\text{FP}_n$-injective objects, which will be the right half of a complete cotorsion pair. 

As a generalization of the category of modules over an $n$-coherent ring, we present the concept of $n$-coherent categories, which also recovers the notions of locally noetherian and locally coherent categories for $n = 0, 1$. Such categories will provide a setting in which the $\text{FP}_n$-injective cotorsion pair is hereditary, and where it is possible to construct (pre)covers by $\text{FP}_n$-injective objects. Moreover, we see how $n$-coherent categories provide a suitable framework for a nice theory of Gorenstein homological algebra with respect to the class of $\text{FP}_n$-injective modules. We define Gorenstein $\text{FP}_n$-injective objects and construct two different model category structures (one abelian and the other one exact) in which these Gorenstein objects are the fibrant objects. 
\end{abstract}


\maketitle


\section*{Introduction}

Throughout this paper, $\mathcal{G}$ will denote a Grothendieck category. Examples of such categories will include: (1) the category $R\mbox{-}\mathsf{Mod}$ of (left) $R$-modules over an associatve ring $R$ with identity; (2) the category $\mathsf{Ch}(R)$ of chain complexes of  $R$-modules; (3) the category $\mathcal{O}_X\mbox{-}\mathsf{Mod}$ of all sheaves of $\mathcal{O}_X$-modules with $(X,\mathcal{O}_X)$ a ringed space; (4) the category $\mathfrak{Qcoh}(X)$ of quasi-coherent sheaves on a scheme $X$; and (5) the category $\mathsf{Fun}(\mathcal{C}^{\rm op},\mathsf{Ab})$ of additive contravariant functors from a skeletally small additive category $\mathcal{C}$ into the category $\mathsf{Ab}$ of abelian groups. 

It is very well known the relation between noetherian rings and finitely generated modules over such rings. One important result asserts that a ring $R$ is noetherian if, and only if, the class of finitely generated modules is closed under taking kernels of epimorphisms. A similar equivalence holds true between coherent rings and finitely presented modules. For each of these types of modules and rings there is a generalizing concept, namely: modules of type $\text{FP}_n$ (also called finitely $n$-presented modules) and $n$-coherent rings. The former was probably first introduced in \cite{Bourbaki}, while the latter is due to D. L. Costa \cite{costa}. As one can expect, there is a nice interplay between modules of type $\text{FP}_n$ and $n$-coherent rings in terms of closure properties. This is described in \cite{bravo-perez} by the first and third authors. Namely, a ring $R$ is $n$-coherent if, and only if, the class $\mathcal{FP}_n$ of modules of type $\text{FP}_n$ is closed under taking kernels of epimorphisms. Another description of $n$-coherent rings can be stated in terms of the existence of a certain hereditary cotorsion pair. Such cotorsion pairs are scarce in the literature, and an interesting point about the theory of $n$-coherent rings is that they govern some conditions for the existence of hereditary cotorsion pairs constructed from $\mathcal{FP}_n$. Specifically, if one considers the class $\class{FP}_n\textnormal{-Inj}$ of $\text{FP}_n$-injective modules defined in \cite{bravo-perez}, one has a complete cotorsion pair $({}^{\perp_1}(\class{FP}_n\textnormal{-Inj}),\class{FP}_n\textnormal{-Inj})$ cogenerated by a set, which is hereditary if, and only if, the ground ring $R$ is $n$-coherent. This is proved as one of the main results in \cite[Theorem 5.5]{bravo-perez}.  

The first general goal of this article is to present and study the concept of $n$-coherent categories as a general framework for the study of finiteness conditions of objects, based mainly in the proposal of the concepts of locally type $\text{FP}_n$ categories and $n$-coherent objects, as generalizations of locally finitely generated and locally finitely presented categories, and of noetherian and coherent objects (see \cite{stenstrom,stovicek-purity}). Our main result is Theorem \ref{them-locally n-coherent} where we give several characterizations of $n$-coherent categories. One of these characterizations is given in terms of the existence of a hereditary small cotorsion theory generated by the class of objects of type $\text{FP}_n$. One important consequence is that any $\text{Ext}^k(F,-)$ can be computed using $\text{FP}_n$-injective coresolutions whenever $F$ is of type $\text{FP}_n$. Theorem \ref{them-locally n-coherent} also generalizes the results in \cite{bravo-perez} about modules of type $\text{FP}_n$, $\text{FP}_n$-injective modules and $n$-coherent rings to the more general context of Grothendieck categories. In particular, we shall be able to apply and interpret our notions of $n$-coherency and objects of type $\text{FP}_n$ in categories widely used in algebraic geometry and representation theory of algebras, such as $\mathfrak{Qcoh}(X)$ and $\mathsf{Fun}(\mathcal{C}^{\rm op},\mathsf{Ab})$


The second general goal is to set the path towards a nice theory of Gorenstein injective homological algebra in Grothendieck categories. For this we present the concept of Gorenstein $\text{FP}_n$-injective objects, which recovers the notion of Gorenstein injective and Ding injective modules in the cases where $n = 0$ and $n = 1$, respectively. From a homological point of view, this class is going to satisfy a series of expected properties. In the context of homotopical algebra, we shall study the stable category associated to the Gorenstein $\text{FP}_n$-injectives and propose two different model structures that describe it. 

One strong point about our definitions and most of our results is that they do not need that our ground Grothendieck category has enough projectives. So the core of the theory presented in this paper can be applied to some of such categories widely used in algebraic geometry, like for instance $\mathcal{O}_X\mbox{-}\mathsf{Mod}$ and $\mathfrak{Qcoh}(X)$.

The present article is organized as follows. We begin with some categorical and homological preliminaries. In Section \ref{Sec-locally finitely n-presented cats} we present the concept of objects of type \tFP{n} in a Grothendieck category and study several closure properties along with some alternative descriptions under some extra assumption in our ground category. We also define locally type \tFP{n} categories as a formal setting for the existence of  objects of type \tFP{n}. In Section \ref{sec:injectivity} we study injectivity relative to  objects of type \tFP{n}. We define the class $\class{FP}_n\textnormal{-Inj}$ of $\text{FP}_n$-injective objects and show that this class is the right half of a complete cotorsion pair $({}^{\perp_1}(\class{FP}_n\textnormal{-Inj}),\class{FP}_n\textnormal{-Inj})$ cogenerated by a set in any locally type \tFP{n} category. Section \ref{Sec-locally n-coherent cats} is devoted to $n$-coherent categories. One of the main results in that section will be to show that the previous cotorsion pair is hereditary if, and only if, the ground category is $n$-coherent, thus generalizing \cite[Theorem 5.5]{bravo-perez}. Another important result that holds in $n$-coherent categories is that $\class{FP}_n\textnormal{-Inj}$ will be a covering class. As an application of this, we obtain a result due to S. Crivei, M. Prest and B. Torrecillas \cite{CPT} about the existence of absolutely pure covers in locally coherent categories. Finally, in Section \ref{sec-Goren-FP_n} we define the Gorenstein $\text{FP}_n$-injective objects and construct two different model structures such that they form the class of fibrant objects. The first structure will be abelian in the sense of Hovey's \cite{hovey} and it will be constructed on $n$-coherent categories. The second one will be exact in the sense of \cite{gillespie-exact-model-structures} and it will be constructed on certain thick subcategories of the ground Grothendieck category without imposing any condition on $\mathcal{G}$.


\section{Preliminaries and notations}\label{Sec-preliminaries}

The categorical setting for this paper is that of Grothendieck categories for which our main reference is Stenstr\"om's \cite{stenstrom}.


\subsection*{{Grothendieck categories}}\label{subsec-Grothendieck cats}

Recall that a Grothendieck category is a cocomplete abelian category $\cat{G}$, with a generating set, and with exact direct limits. We shall often refer to~\cite[Chapter~V]{stenstrom}. To orient the reader, we now summarize some standard facts. First, a Grothendieck category is always complete and every object $B \in \cat{G}$ has an injective envelope $E(B)$. In particular, $\cat{G}$ has enough injectives and these can be used to compute $\Ext^n_{\cat{G}}$. A useful fact is that any Grothendieck category is \emph{well-powered}, meaning the class of subobjects of any given object is in fact a set. See~\cite[Proposition~IV.6.6]{stenstrom}, although he uses the term \emph{locally small} instead of well-powered. Finally, given any regular cardinal $\gamma$, by~\cite[Corollary~1.69]{adamek-rosicky}, the class  of all $\gamma$-presented objects is skeletally small (or essentially small). This means there exists a set of isomorphism representatives for this class.

\subsection*{{$\boldsymbol{\text{free}(\mathcal{S})}$ and $\boldsymbol{\text{add}(\mathcal{S})}$}}\label{subsec-Grothendieck cats gen sets}

Let $\class{S}$ be a set of objects in a Grothendieck category $\cat{G}$. We shall denote by  $\textnormal{Free}(\mathcal{S})$ the class of all set indexed direct sums $\bigoplus S_i$ in which each $S_i \in \class{S}$. We use the notation $\textnormal{free}(\mathcal{S})$ to denote the set of all such \emph{finite} direct sums.
Similarly, we shall denote by  $\textnormal{Add}(\mathcal{S})$ (respectively, $\textnormal{add}(\mathcal{S})$) the class of all direct summands of objects in $\textnormal{Free}(\mathcal{S})$ (respectively, $\textnormal{free}(\mathcal{S})$). These notations are motivated by the special case of modules over a ring: Taking $\class{S} = \{R\}$, we get the classes of free, finitely generated free, projective, and finitely generated projective $R$-modules.


\subsection*{{Approximations}}\label{subsec-Grothendieck cats}

Given a class $\mathcal{X}$ of objects in an abelian category $\mathcal{A}$, a morphism $f \colon X \to A$ is called an \emph{$\mathcal{X}$-precover} of $A \in \mathcal{A}$ if $X \in \mathcal{X}$ and if for every morphism $f' \colon X' \to A$ with $X' \in \mathcal{X}$ there exists a morphism $h \colon X' \to X$ such that $f' = f \circ h$. In some references, $\mathcal{X}$-precovers are called \emph{right approximations}. If in addition, in case $X' = X$ and $f' = f$, the previous equality can only be completed by automorphisms $h$ of $X$, then the $\mathcal{X}$-precover $f$ is called an \emph{$\mathcal{X}$-cover} of $A$. Furthermore, an $\mathcal{X}$-precover $f \colon X \to A$ is \emph{special} if it is epic and ${\rm Ext}^1_{\mathcal{A}}(X',{\rm Ker}(f)) = 0$ for every $X' \in \mathcal{X}$. Dually, one has the notions of (\emph{special}) \emph{$\mathcal{X}$-preenvelopes} and \emph{$\mathcal{X}$-envelopes}.

A class $\mathcal{X}$ of objects in $\mathcal{A}$ is called \emph{precovering} if every object $A \in \mathcal{A}$ has an $\mathcal{X}$-precover. \emph{Special precovering}, \emph{covering}, (\emph{special}) \emph{preenveloping} and \emph{enveloping classes} are defined similarly.


\subsection*{{Cotorsion pairs}}\label{subsec-abelian model cats}

Given a class $\mathcal{C}$ of objects in an abelian category $\mathcal{A}$, the \emph{$i$-th right orthogonal} $\mathcal{C}^{\perp_i}$, with $i \geq 1$, is defined as the class of all objects $X \in \mathcal{A}$ such that ${\rm Ext}^i_{\mathcal{A}}(C,X) = 0$ for every $C \in \mathcal{C}$. The \emph{total right orthogonal} is defined as the intersection $\mathcal{C}^\perp = \bigcap_{i \geq 1} \mathcal{C}^{\perp_i}$. Similarly, we define the \emph{$i$-th left orthogonal} and \emph{total left orthogonal} ${}^{\perp_i}\mathcal{C}$ and ${}^\perp\mathcal{C}$. 

Recall that two classes of objects $\mathcal{X}$ and $\mathcal{Y}$ in an abelian category $\mathcal{A}$ form a \emph{cotorsion pair} $(\mathcal{X,Y})$ in $\mathcal{A}$ if $\class{Y} = \mathcal{X}^{\perp_1}$ and $\class{X} = {}^{\perp_1}\mathcal{Y}$. Any class $\class{S}$ for which $\mathcal{S}^{\perp_1} = \class{Y}$ is said to \emph{cogenerate} the cotorsion pair $(\mathcal{X,Y})$. In particular, we shall say the cotorsion pair is \emph{cogenerated by a set} if there exists a set $\mathcal{S}$ (not just a proper class) such that $\mathcal{S}^{\perp_1} = \class{Y}$.

The cotorsion pair is \emph{hereditary} if $\Ext^i_{\cat{A}}(X,Y) = 0$ for all $X \in \class{X}$, $Y \in \class{Y}$ and $i \geq 1$ (In other words, $\mathcal{Y} = \mathcal{X}^\perp$ and $\mathcal{X} = {}^\perp\mathcal{Y}$). We also say the cotorsion pair is \emph{complete} if it has enough injectives and enough projectives. This means that for each $A \in \cat{A}$ there exist short exact sequences 
\[
0 \to A \to Y \to X \to 0 \mbox{ \ and \ } 0 \to Y' \to X' \to A \to 0
\] 
with $X,X' \in \class{X}$ and $Y,Y' \in \class{Y}$ (In other words, every object has a special $\mathcal{X}$-precover and a special $\mathcal{Y}$-preenvelope). If these short exact sequences can be taken functorially with respect to $A$ then, following~\cite[Definition~2.3]{hovey}, we say the cotorsion pair is \emph{functorially complete}.  In particular, cotorsion pairs in $\mathcal{A}$ cogenerated by a set are functiorally complete, provided that $\mathcal{A}$ is a Grothendieck category with enough projectives (See \cite[Corollary 6.8]{hovey}).

Besides their connection to abelian model structures which we describe next, cotorsion pairs are fundamental in modern homological algebra. There are several good references. In particular we shall refer to \cite{enochs-jenda-book} and \cite{hovey}.


\subsection*{{Abelian model structures}}\label{subsec-abelian model cats}

Let $\cat{A}$ be a bicomplete abelian category. M. Hovey showed in \cite{hovey} that an abelian model structure on $\cat{A}$ is nothing more than two nicely related cotorsion pairs in $\cat{A}$. The main theorem of \cite{hovey} showed that an abelian model structure on $\cat{A}$ is equivalent to a triple $(\class{Q},\class{W},\class{R})$ of classes of objects in $\cat{A}$ for which $\class{W}$ is thick and $(\class{Q} \cap \class{W},\class{R})$ and $(\class{Q},\class{W} \cap \class{R})$ are each complete cotorsion pairs. By \emph{thick} we mean that the $\class{W}$ is closed under direct summands and satisfies the 2 out of 3 property on short exact sequences. In this case, $\class{Q}$ is precisely the class of cofibrant objects of the model structure, $\class{R}$ are precisely the fibrant objects, and $\class{W}$ is the class of trivial objects. We say that $\class{M}$ is \emph{hereditary} if both of these associated cotorsion pairs are hereditary. 

The equivalence between these $(\mathcal{Q,W,R})$ and abelian model structures was later generalized by the second author in \cite{gillespie-exact-model-structures} to the context of exact categories. The notion of (complete and hereditary) cotorsion pairs are analogous in such categories, and the corresponding model structures in this equivalence are called \emph{exact}. For a complete survey of exact categories, we recommend \cite{buhler}.


\subsection*{{Cofibrantly generated and finitely generated model categories}}

We refer to \cite[Definition 2.1.17 and Chapter 7]{hovey} for the definitions of cofibrantly generated and finitely generated model categories as well as theory relating them to triangulated categories. Here we shall just note some basic facts used in this paper and give appropriate references to guide the reader. So let $\class{M}$ be an abelian model category. Its homotopy category is denoted by Ho($\class{M}$). It is known that Ho($\class{M}$) is always a pretriangulated category~\cite[Section 6.5]{hovey} and that it is in fact triangulated\footnote{Here we mean a triangulated category in the sense of \cite[Definition 7.1.1]{hovey}. This notion of triangulated categories is stronger than the classical concept due to Verdier's \cite{Verdier96}.} whenever $\class{M}$ is hereditary \cite[Corollary 1.1.15]{becker}. In this case it follows from \cite[Section 7.4]{hovey} that Ho($\class{M}$) is compactly generated whenever $\class{M}$ is a finitely generated model category.


\subsection*{{Pure exact sequences}} 

Given a short exact sequence 
\[
\mathbb{E} \colon 0 \to A \to B \to C \to 0
\] 
of objects in a Grothendieck category $\mathcal{G}$, recall that $\mathbb{E}$ is said to be \emph{pure} if for every finitely presented object $F \in \mathcal{G}$, the induced sequence ${\rm Hom}_{\mathcal{G}}(F,\mathbb{E})$ of abelian groups is also exact. In case where $\mathcal{G}$ is the category of $R$-modules, this is equivalent to saying that $\mathbb{E}$ remains exact after tensoring with any right $R$-module. We cannot state this equivalence for general Grothendieck categories since they may not even come equipped with a tensor product.

One can consider certain closure properties with respect to pure exact sequences. Namely, a class $\mathcal{X}$ of objects in $\mathcal{G}$ is said to be \emph{closed under pure subobjects} (resp., \emph{under pure quotients}) if whenever we are given a pure exact sequence as $\mathbb{E}$ above with $B \in \mathcal{X}$, then one has $A \in \mathcal{X}$ (resp., $C \in \mathcal{X}$). 

Pure exact sequences are not  the only concept considered in this article with an equivalent interpretation for modules, that does not necessarily hold for arbitrary Grothendieck categories. This will also be the case of  objects of type \tFP{n} studied in the next section.


\subsection*{{Some specific notations}}\label{subsec-notation} 

We specify the use of some symbols throughout this article:
\begin{itemize}
\item In some cases, monomorphisms (respectively, epimorphisms) will be denoted as arrows $\rightarrowtail$ (respectively, $\twoheadrightarrow$).

\item Given two objects $X$ and $Y$ in an abelian category $\mathcal{A}$, by $X \simeq Y$ we shall mean that $X$ and $Y$ are isomorphic. If $F, G \colon \mathcal{A} \longrightarrow \mathcal{D}$ are two functors between abelian categories, by $F \cong G$ we shall mean that there exists a natural isomorphism between $F$ and $G$. 

\item Recall that two short exact sequences 
\[
\mathbb{E} \colon 0 \to Y \xrightarrow{\alpha} Z \xrightarrow{\beta} X \to 0 \mbox{ \ and \ } \mathbb{E}' \colon 0 \to Y \xrightarrow{\alpha'} Z' \xrightarrow{\beta'} X \to 0
\] 
are equivalent if there exists a morphism $h \colon Z \to Z'$ such that $h \circ \alpha = \alpha'$ and $\beta' \circ h = \beta$. This will be denoted as $\mathbb{E} \sim \mathbb{E}'$. We shall use the same notation to denote equivalences between $n$-fold extensions (See Appendix A at the end of this article). 

In some cases, the groups of equivalence classes of $n$-extensions ${\rm Ext}^n_{\mathcal{A}}(X,Y)$ appearing in certain commutative diagrams will be denoted as ${}^n(X,Y)$ due to space limitations.    
\end{itemize}


\section{Objects of type \tFP{n}}\label{Sec-locally finitely n-presented cats}

Throughout this paper $\cat{G}$ denotes a Grothendieck category (with not necessarily enough projective objects). In this section we study the notion of objects of type \tFP{n} in $\cat{G}$. We also define what it means to say $\cat{G}$ is locally  type \tFP{n}.  

Note that for any object $C \in \cat{G}$, and any direct system $\{X_i\}_{i \in I}$, there is a canonical map 
\[
\xi_n \colon \varinjlim \Ext^n_{\cat{G}}(C,X_i) \xrightarrow{} \Ext^n_{\cat{G}}(C,\varinjlim X_i)
\]
for each $n \geq 0$. To say $\Ext^n_{\cat{G}}(C,-)$ \emph{preserves direct limits} means that $\xi_n$ is an isomorphism for each direct system $\{X_i\}_{i \in I}$. Recall that $C$ is called \emph{finitely presented} if $\Ext^0_{\cat{G}}(C,-) = \Hom_{\cat{G}}(C,-)$ preserves direct limits.  The following definition generalizes this.

\begin{definition}\label{def-finitely-n-presented}
Let $n \geq 1$ be a positive integer. We say that an object $F \in \cat{G}$ is \textbf{of type \tFP{\bm{n}}}, if the functors $\Ext^i_{\cat{G}}(F,-)$ preserve direct limits for all $0 \leq i \leq n-1$. 
\end{definition}

Note that any object of type \tFP{n} is finitely presented and that the notion of finitely presented is synonymous with type \tFP{1}. Moreover, any object of type \tFP{n} is finitely generated by~\cite[Def.~V.3.1 and Prop.~V.3.2]{stenstrom}. It will be convenient to think of the finitely generated objects as the objects of type \tFP{0}\footnote{With a particular exception in the category of $\mathcal{O}_X$-modules. Indeed, the notion of finitely generated $\mathcal{O}_X$-modules may be different from that of $\mathcal{O}_X$-modules of type \tFP{0}. For example, in Ueno's \cite[Definition 4.18]{Ueno-2}, an $\mathcal{O}_X$-module $\mathcal{F}$ is called \emph{finitely generated} if for every $x \in X$ there exists an open set $U$ containing $x$ and a positive integer $n > 0$ so that sequence of $\mathcal{O}_U$-modules $\mathcal{O}^{\oplus n}_U \to \mathcal{F}|_U \to 0$ is exact. Other authors refer to such $\mathcal{O}_X$-modules as \emph{locally finitely generated}.}. Thus by an \emph{object of type \tFP{0}} we mean a finitely generated object. Finally, we may let $n = \infty$, and call an object $F$ \emph{of type \tFP{\infty}} if $\Ext^i_{\cat{G}}(F,-)$ preserves direct limits for all $i \geq 0$.  

Now for all $0 \leq n \leq \infty$, we let $\class{FP}_n$ denote the class of all objects of type \tFP{n} in $\mathcal{G}$. For convenience we let $\class{FP}_{-1}$ denote the whole class of objects of $\mathcal G$. We note that $\class{FP}_{\infty} = \bigcap_{n \geq 0} \class{FP}_n$ and that we have a decreasing chain of containments:
\[
\class{FP}_0 \supseteq \class{FP}_1 \supseteq \cdots \supseteq \class{FP}_n \supseteq \class{FP}_{n+1} \supseteq \cdots \supseteq \class{FP}_{\infty}.
\]

\begin{example}\label{Example-fg-projectives}
We give some examples of objects of type \tFP{n}.
\begin{enumerate}
\item Any finitely generated projective object must be of type ${\rm FP}_{\infty}$ by \cite[Example~3.2]{gillespie-models-of-injectives}. 

\item \textbf{Modules over a ring}. For each $n \geq 0$, by \cite[Example~13]{bravo-parra}, one can construct a ring $R$ such that 
\[
\class{FP}_0 \supsetneq \class{FP}_1 \supsetneq \cdots \supsetneq \class{FP}_n = \class{FP}_{\infty}.
\] 
It is important to mention that the class $\mathcal{FP}_n$ in $R\mbox{-}{\rm Mod}$ has an equivalent description. Namely, a module $F$ is of type \tFP{n} if, and only if, there exists an exact sequence
\begin{align}
P_n & \to P_{n-1} \to \cdots \to P_1 \to P_0 \to F \to 0 \label{eqn:resFPn}
\end{align}
of modules where $P_i$ is a finitely generated projective module for every $0 \leq i \leq n$. We shall refer to such sequences \eqref{eqn:resFPn} as \emph{$n$-presentations (by finitely generated projective objects) of $F$}. We have chosen this terminology since modules of type $\text{FP}_n$ are also known as \emph{finitely $n$-presented} (See \cite{bravo-perez}, for example). 

The existence of this equivalent description for $\mathcal{FP}_n$ is due to the fact that modules form a Grothendieck category which has a generating set of finitely generated projective objects. We shall specify this later in Corollary~\ref{resolutions_for_proj}. Without such generators, the concepts of objects of type $\text{FP}_n$ and objects with an $n$-presentation may differ, as shown in Example (5) below.

\item \textbf{Chain complexes}. The previous description of objects of type $\text{FP}_n$ in terms of $n$-presen-tations is also true in the category $\textrm{Ch}(R)$ of chain complexes of modules over $R$, studied in \cite{ZhaoPerez}. Moreover, complexes of type $\text{FP}_n$ are also described as those $X \in \Ch(R)$ such that $X$ is bounded (above and below) and each $X_m$ is a module of type $\text{FP}_n$ (See \cite[Proposition 2.1.4]{ZhaoPerez}).

\item \textbf{Functors on additive categories}. Let $\mathcal{C}$ be a skeletally small additive category and consider the category $\textrm{Fun}(\mathcal{C}^{\rm op},\mathsf{Ab})$ of contravariant additive functors from $\mathcal{C}$ to $\mathsf{Ab}$. As the categories of modules and chain complexes of modules, $\textrm{Fun}(\mathcal{C}^{\rm op},\mathsf{Ab})$ is a Grothendieck category with a generating set of finitely generated projective objects. It is known by Auslander's \cite{Auslander1and2} that an object in $\textrm{Fun}(\mathcal{C}^{\rm op},\mathsf{Ab})$ is finitely generated and projective if, and only if, it is a direct summand of a representable functor $\textrm{Hom}_{\mathcal{C}}(-,X)$ for some object $X \in \mathcal{C}$. Thus, an object $F \in \textrm{Fun}(\mathcal{C}^{\rm op},\mathsf{Ab})$ is of type $\text{FP}_n$ if, and only if, there exists an exact sequence of the form
\[
\mbox{ \ \ \ \ \ \ \ \ \ \ } \textrm{Hom}_{\mathcal{C}}(-,X_n) \to \textrm{Hom}_{\mathcal{C}}(-,X_{n-1}) \to \cdots \to \textrm{Hom}_{\mathcal{C}}(-,X_1) \to \textrm{Hom}_{\mathcal{C}}(-,X_0) \to F \to 0. 
\]

\item \textbf{Quasi-coherent sheaves}. Let $k$ be an infinite field. Consider the quasi-compact and semi-separated scheme $X = \mathbb{P}^1(k)$ along with the category $\mathfrak{Qcoh}(X)$ of quasi-coherent sheaves over $X$. It is a well known fact that $\mathfrak{Qcoh}(X)$ has not enough projectives. (See Hartshorne's \cite[Exercise VI.6.2]{hartshorne}). Moreover, $\mathfrak{Qcoh}(X)$ has no nonzero projective objects (see \cite[Theorem 2.4.12]{Berest}), and so every object having an $n$-presentation must be the zero object, for any $n \geq 0$.

On the other hand, for any $n \geq 0$, one can construct generators (and so nonzero objects) of type $\text{FP}_n$ for the category $\mathfrak{Qcoh}(X)$ from the semi-separating cover of $\mathbb{P}^1(k)$ given by $D_+(x_0)$ and $D_+(x_1)$. (See \cite[Corollary 2.5]{Estrada-Gillespie} for details). Hence, the notions of being of type $\text{FP}_n$ and having an $n$-presentation are not necessarily equivalent.   

In \cite[Proposition 3.7]{AbsPureSheaves}, Enochs, Estrada and Odaba\c{s}{\i} characterized the finitely presented objects in $\mathfrak{Qcoh}(X)$ in the case where $X$ is a semi-separated or a concentrated scheme. Specifically, $\mathscr{F} \in \mathfrak{Qcoh}(X)$ is finitely presented if, and only if, $\mathscr{F}|_{U}$ is finitely presented in $\mathfrak{Qcoh}(U)$ for every open affine subset $U \subseteq X$, or if, and only if, the stalk $\mathscr{F}_x$ is a finitely presented $\mathcal{O}_{X,x}$-module for every $x \in X$. A similar description with more conditions is also true for quasi-coherent sheaves over $X$ of type $\text{FP}_n$, for the case $X$ is quasi-compact and semi-separated. Namely, the following conditions are equivalent for $\mathscr{F} \in \mathfrak{Qcoh}(X)$ and $n \geq 1$:
\begin{itemize}
\item[(a)] $\mathscr{F}$ is of type $\text{FP}_n$ in $\mathfrak{Qcoh}(X)$.

\item[(b)] $\mathscr{E}{xt}^k_{X}(\mathscr{F},-) \colon \mathfrak{Qcoh}(X) \longrightarrow \mathcal{O}_X\mbox{-}\textrm{Mod}$ preserves direct limits for every $0 \leq k \leq n-1$.\footnote{Here, $\mathscr{E}{xt}^k_X(\mathscr{F},-)$ denotes the Ext sheaves, that is, the right derived functors of the hom sheaf $\mathscr{H}{om}(\mathscr{F},-)$. (See \cite[Section III.6]{hartshorne}).} 

\item[(c)] $\mathscr{F}|_U$ is of type $\text{FP}_n$ in $\mathfrak{Qcoh}(U)$ for all quasi-compact (or affine) open subset $U \subseteq X$. 

\item[(d)] $\mathscr{E}{xt}^k_U(\mathscr{F}|_U,-) \colon \mathfrak{Qcoh}(U) \longrightarrow \mathcal{O}_X\mbox{-}\textrm{Mod}$ preservers direct limits for every $0 \leq k \leq n-1$ and every quasi-compact (or affine) open subset $U \subseteq X$. 

\item[(e)] $\mathscr{F}(U)$ is an $\mathcal{O}_X(U)$-module of type $\text{FP}_n$ for every affine open subset $U \subseteq X$. 
\end{itemize}
\end{enumerate}
For a detailed proof of this equivalence, see \cite[Proposition 2.3]{Estrada-Gillespie}.\footnote{The result is stated and proved for quasi-coherent sheaves of type $\text{FP}_\infty$, but the arguments are also valid for objects of type $\text{FP}_n$.}
\end{example}


\subsection*{{Locally type \tFP{\bm{n}} categories}}

Although below we provide ways to construct new objects of type $\text{FP}_n$ from old ones, there is no guarantee that a Grothendieck category possesses any nonzero objects of type ${\rm FP}_n$. So following~\cite{gillespie-models-of-injectives} we propose Definition~\ref{def-locally FP-infinity} below in the spirit of locally finitely generated and locally finitely presented categories. Recall that a Grothendieck category $\cat{G}$ is called \emph{locally finitely generated} if it has a set of finitely generated generators. This is equivalent to saying that each $C \in \cat{G}$ is a direct union of finitely generated subobjects~\cite[pp.~122]{stenstrom}. $\cat{G}$ is called \emph{locally finitely presented} if it has a set of finitely presented generators. This is equivalent to saying that each $C \in \cat{G}$ is a direct limit of finitely presented objects~\cite[Definition 1.9 and Theorem 1.11]{adamek-rosicky}.

\begin{definition} \label{def-locally FP-infinity}
We say that a Grothendieck category $\cat{G}$ is  \textbf{locally type \tFP{\bm{n}}}, if it has a generating set consisting of objects of type \tFP{n}.
\end{definition}

So $n = 0$ gives us the locally finitely generated categories, $n = 1$ the locally finitely presented categories, and $n = \infty$ gives us the locally type ${\rm FP}_{\infty}$ categories of~\cite{gillespie-models-of-injectives}. Note that for $1 \leq n \leq \infty$, any locally type \tFP{n} category is a locally type \tFP{n-1} category. In particular any such category is locally finitely presented and hence locally finitely generated.

\begin{example}\label{Examples-cats}
The categories $R\mbox{-}\textrm{Mod}$, $\textrm{Ch}(R)$, $\mathfrak{Qcoh}(\mathbb{P}^n(A))$ (with $A$ a commutative ring) and $\textrm{Fun}(\mathcal{C}^{\rm op},\mathsf{Ab})$ are locally type $\text{FP}_n$ with the following generating sets formed by objects of type $\text{FP}_\infty$, respectively:
\begin{itemize}
\item The singleton $\{ R \}$.

\item The set of disk complexes $\{ D^m(R) \}_{m \in \mathbb{Z}}$, where $D^m(R)$ is the complex with 
\[
D^m(R)_k = \left\{ \begin{array}{ll} M & \text{if $k = m, m-1$}, \\ 0 & \text{otherwise} \end{array} \right.,
\] 
and such that the only nonzero differential map is given by ${\rm id}_M \colon M \to M$.

\item The set of twisted sheaves $\{ \mathcal{O}_{\mathbb{P}^n(A)}(m) \}_{m \in \mathbb{Z}}$ (see \cite[Corollary 2.5]{Estrada-Gillespie} for more details).

\item The set of representable contravariant functors $\{ {\rm Hom}_{\mathcal{C}}(-,X) \}_{X \in \mathcal{C'}}$ where $\mathcal{C}'$ is a set of representative objects of the skeletally small category $\mathcal{C}$ (see Stenstr\"om's \cite[Corollary IV.7.5]{stenstrom}).
\end{itemize}
\end{example}


\subsection*{{Properties of objects of type \tFP{n}}}

Recalling the notion of a thick subcategory from the preliminaries, we have the following proposition which is proved in \cite[Proposition 3.3]{gillespie-models-of-injectives}.

\begin{proposition}\label{prop-thickness of FP-infinity}
The class of all objects of type ${\rm FP}_{\infty}$ is a thick subcategory.
\end{proposition}

As shown in~\cite[Section~1]{bravo-perez}, for $n < \infty$, the class $\class{FP}_n$ in the category of left $R$-modules over a ring $R$ is almost thick except it need not be closed under taking kernels of epimorphisms between its objects. This is proved using the characterization of modules of type $\text{FP}_n$ mentioned in Example~\ref{Example-fg-projectives} (2). Our goal now is to prove the analogous result in the current context of Grothendieck categories. This is achieved below in Proposition~\ref{coro-properties of FP-n}. In the absence of enough projectives, the key ingredient will be to apply the ``5-lemma'' along with the following technical lemma. The proof of the lemma is quite long and technical and we defer it to an appendix at the end of the present article.

\begin{lemma}\label{lem:mono_condition}
Let $F$ be an  object of type \tFP{n} in a locally finitely presented category $\mathcal{G}$, and let $\{ X_i \mbox{ : } i \in I \}$ be a direct system of objects in $\mathcal{G}$. The canonical map $\xi_n \colon \varinjlim {\rm Ext}^n_{\mathcal{G}}(F, X_i) \to {\rm Ext}^n_{\mathcal{G}}(F,\varinjlim X_i)$ is a monomorphism. 
\end{lemma}

We shall also use the following characterizations of finitely presented objects.

\begin{lemma}[descriptions of finitely presented objects]\label{lem:finitely_presented}
Let $\mathcal{G}$ be a locally finitely generated category. The following conditions are equivalent for every $C \in \mathcal{G}$.
\begin{itemize}
\item[(a)] $C$ is finitely presented.

\item[(b)] $C$ is finitely generated and every epimorphism $B \twoheadrightarrow C$, where $B$ is finitely generated, has finitely generated kernel.

\item[(c)] There exists a short exact sequence 
\[
0 \to K \to F \to C \to 0
\] 
where $K$ is finitely generated and $F$ is finitely presented. 
\end{itemize}
\end{lemma}

\begin{proof}
The equivalence (a) $\Longleftrightarrow$ (b) is due to Stenstr\"om \cite[Proposition V.3.4]{stenstrom}. The implication (a) $\Longrightarrow$ (c) is clear, while (c) $\Longrightarrow$ (b) follows using a standard pullback argument along with the equivalence (a) $\Longleftrightarrow$ (b).
\end{proof}

\begin{proposition}[closure properties of $\mathcal{FP}_n$]\label{coro-properties of FP-n}
Let $\mathcal{G}$ be a locally finitely presented category and 
\[
\mathbb{E} \colon 0 \to A \to B \to C \to 0
\] 
be a short exact sequence in $\mathcal{G}$. The following conditions hold for all $0 \leq n \leq \infty$:
\begin{enumerate}
\item If $A, C \in \class{FP}_n$, then $B \in \class{FP}_n$. That is, $\mathcal{FP}_n$ is closed under extensions. 

\item If $A \in \mathcal{FP}_{n-1}$ and $B \in \mathcal{FP}_n$, then $C \in \mathcal{FP}_n$. In particular, $\mathcal{FP}_n$ is closed under taking cokernels of monomorphisms between its objects.

\item If $B \in \mathcal{FP}_{n-1}$ and $C \in \class{FP}_n$, then $A \in \class{FP}_{n-1}$. 

\item If $\mathbb{E}$ splits and $B \in \class{FP}_n$ then $A,C \in \class{FP}_n$. That is, $\class{FP}_n$ is closed under direct summands.
\end{enumerate}
\end{proposition}

\begin{proof}
The case $n = 0$ is done in \cite[Lemma V.3.1(2)]{stenstrom}, and  the case $n = \infty$ is given by Proposition \ref{prop-thickness of FP-infinity}.

Next, let $1 < n < \infty$, and let $X$ be the direct limit of a direct system $\{ X_i \mbox{ : } i \in I \}$ of objects in $\mathcal{G}$, that is, $X = \varinjlim X_i$. For $A, B, C \in \mathcal{G}$ and $k \geq 0$, we consider the corresponding natural homomorphisms 
\begin{align*}
\xi^A_k \colon & \varinjlim {\rm Ext}^k_{\mathcal{G}}(A,X_i) \to {\rm Ext}^k_{\mathcal{G}}(A,X), \\
\xi^B_k \colon & \varinjlim {\rm Ext}^k_{\mathcal{G}}(B,X_i) \to {\rm Ext}^k_{\mathcal{G}}(B,X), \\
\xi^C_k \colon & \varinjlim {\rm Ext}^k_{\mathcal{G}}(C,X_i) \to {\rm Ext}^k_{\mathcal{G}}(C,X).
\end{align*} 

\begin{enumerate}
\item The case $n = 1$ can be proved using Lemma \ref{lem:finitely_presented} and a standard pullback argument. So we may assume that $A, C \in \mathcal{FP}_n$ with $n > 1$. We want to show that $\xi^B_k$ is an isomorphism for every $0 \leq k \leq n-1$. By the previous comments, we already know that $\xi^B_0$ is an isomorphism. For indices $k > 0$, we have the following commutative diagram with exact rows (recall our notation convention from the end of Section~\ref{Sec-preliminaries}):
\[
\begin{tikzpicture}[description/.style={fill=white,inner sep=2pt}]
\matrix (m) [matrix of math nodes, row sep=3.5em, column sep=1.25em, text height=1.25ex, text depth=0.25ex]
{ 
\varinjlim {}^{k-1}(A,X_i) & \varinjlim {}^k(C,X_i) & \varinjlim {}^k(B,X_i) & \varinjlim {}^k(A,X_i) & \varinjlim {}^{k+1}(C,X_i) \\
{}^{k-1}(A,X) & {}^k(C,X) & {}^k(B,X) & {}^k(A,X) & {}^{k+1}(C,X) \\
};
\path[->]
(m-1-1) edge (m-1-2) (m-1-2) edge (m-1-3) (m-1-3) edge (m-1-4) (m-1-4) edge (m-1-5)
(m-2-1) edge (m-2-2) (m-2-2) edge (m-2-3) (m-2-3) edge (m-2-4) (m-2-4) edge (m-2-5)
(m-1-1) edge node[right] {\footnotesize$\xi_{k-1}^A$} (m-2-1)
(m-1-2) edge node[right] {\footnotesize$\xi_k^C$} (m-2-2) 
(m-1-3) edge node[right] {\footnotesize$\xi_k^B$} (m-2-3) 
(m-1-4) edge node[right] {\footnotesize$\xi_k^A$} (m-2-4) 
(m-1-5) edge node[right] {\footnotesize$\xi_{k+1}^C$} (m-2-5)
;
\end{tikzpicture}
\]
By assumption, $\xi_{k-1}^A$, $\xi_k^A$ and $\xi_k^C$ are all isomorphisms for every $0 < k \leq n-1$. Also by assumption, $\xi_{k+1}^C$ is an isomorphism for every $0 < k \leq n-2$, and a monomorphism for $k = n-1$ by Lemma \ref{lem:mono_condition}. By the 5-Lemma \cite[Exercise~1.3.3]{weibel}, we deduce that $\xi^B_k$ is an isomorphism for every $0 < k \leq n - 1$. Therefore, $B \in \mathcal{FP}_n$. 

\item Suppose $A \in \mathcal{FP}_{n-1}$ and $B \in \mathcal{FP}_n$. The case $n = 1$ follows by Lemma \ref{lem:finitely_presented}. So we may assume $n > 1$.  Certainly $\xi^C_0$ is an isomorphism, so our goal is to show that $\xi^C_k$ is an isomorphism for every $0 < k \leq n-1$. 
Now, for each $k > 0$, we consider the following commutative diagram with exact rows:
\[
\begin{tikzpicture}[description/.style={fill=white,inner sep=2pt}]
\matrix (m) [matrix of math nodes, row sep=3.5em, column sep=1.25em, text height=1.25ex, text depth=0.25ex]
{ 
\varinjlim {}^{k-1}(B,X_i) & \varinjlim {}^{k-1}(A,X_i) & \varinjlim {}^k(C,X_i) & \varinjlim {}^k(B,X_i) & \varinjlim {}^k(A,X_i) \\
{}^{k-1}(B,X) & {}^{k-1}(A,X) & {}^k(C,X) & {}^k(B,X) & {}^k(A,X) \\
};
\path[->]
(m-1-1) edge (m-1-2) (m-1-2) edge (m-1-3) (m-1-3) edge (m-1-4) (m-1-4) edge (m-1-5)
(m-2-1) edge (m-2-2) (m-2-2) edge (m-2-3) (m-2-3) edge (m-2-4) (m-2-4) edge (m-2-5)
(m-1-1) edge node[right] {\footnotesize$\xi_{k-1}^B$} (m-2-1)
(m-1-2) edge node[right] {\footnotesize$\xi_{k-1}^A$} (m-2-2) 
(m-1-3) edge node[right] {\footnotesize$\xi_k^C$} (m-2-3) 
(m-1-4) edge node[right] {\footnotesize$\xi_k^B$} (m-2-4) 
(m-1-5) edge node[right] {\footnotesize$\xi_k^A$} (m-2-5)
;
\end{tikzpicture}
\]
This time, $\xi^A_{k-1} , \xi^B_{k-1}$ and $\xi^B_{k}$ are isomorphisms for every $0 < k \leq n-1$. But also $\xi^A_k$ is an isomorphism for every $0 \leq k \leq n-2$, and a monomorphism for $k = n-1$. The 5-Lemma implies then that $\xi^C_k$ is an isomorphism for every $0 < k \leq n-1$. 

\item This part is analogous to (2).

\item In the case where $\mathbb{E}$ is split exact, we have that $A$ and $C$ are retracts (or equivalently, direct summands) of $B$. We only show that $A \in \mathcal{FP}_n$ if $B \in \mathcal{FP}_n$, as the proof for $C$ is similar. We have that there are morphisms $\alpha \colon A \to B$ and $\alpha' \colon B \to A$ such that $\alpha' \circ \alpha = {\rm id}_A$. This induces the following commutative diagram where the horizontal compositions are identities:
\[
\begin{tikzpicture}[description/.style={fill=white,inner sep=2pt}]
\matrix (m) [matrix of math nodes, row sep=3.5em, column sep=8em, text height=1.25ex, text depth=0.25ex]
{ 
\varinjlim {}^k(A,X_i) & \varinjlim {}^k(B,X_i) & \varinjlim {}^k(A,X_i) \\
{}^k(A,X) & {}^k(B,X) & {}^k(A,X) \\
};
\path[->]
(m-1-1) edge node[above] {\footnotesize$\varinjlim {}^k(\alpha', X_i)$} (m-1-2)
(m-1-2) edge node[above] {\footnotesize$\varinjlim {}^k(\alpha, X_i)$} (m-1-3)
(m-2-1) edge node[below] {\footnotesize$\varinjlim {}^k(\alpha',X)$} (m-2-2)
(m-2-2) edge node[below] {\footnotesize$\varinjlim {}^k(\alpha,X)$} (m-2-3)
(m-2-2) edge (m-2-3)
(m-1-1) edge node[right] {\footnotesize$\xi_k^A$} (m-2-1) 
(m-1-2) edge node[right] {\footnotesize$\xi_k^B$} (m-2-2) 
(m-1-3) edge node[right] {\footnotesize$\xi_k^A$} (m-2-3)
;
\end{tikzpicture}
\]
Thus, we have that $\xi^A_k$ is a retraction of $\xi^B_k$ in the category of maps between abelian groups. In the case where $0 \leq k \leq n-1$, we have that $\xi^A_k$ is an isomorphism, since isomorphisms are closed under retractions. Hence, $A \in \mathcal{FP}_n$. 
\end{enumerate}
\end{proof}

\begin{remark}
In general, it is not true that the class $\mathcal{FP}_n$ is closed under taking kernels of epimorphisms between its objects. In the category of left $R$-modules, for instance, $\mathcal{FP}_n$ satisfies this closure property if, and only if, the ground ring $R$ is (left) $n$-coherent, as proved in \cite[Theorem 2.4]{bravo-perez}. This equivalence will be presented in our categorical setting in Section \ref{Sec-locally n-coherent cats}, where we introduce and study the Grothendieck categories that we call \emph{$n$-coherent}. 
\end{remark}

To complete our study of closure properties of the class $\mathcal{FP}_n$, we show that the   objects  of type \tFP{n} are also closed under finite direct sums:

\begin{proposition}\label{prop-direct sums}
For all $0 \leq n \leq \infty$, the class $\class{FP}_n$ of all objects of type \tFP{n}, is closed under finite direct sums.
\end{proposition}

\begin{proof} 
Let $n > 0$,  $\{F_1, F_2, \cdots , F_m \} \subseteq \class{FP}_n$, and let $0 \leq i < n$. We have a standard isomorphism
\[
\Ext^i_{\cat{G}}\left( \bigoplus^m_{k = 1} F_k, - \right) \cong \prod^m_{k = 1} \Ext^i_{\cat{G}}\left(F_k,- \right).
\]
So the result follows from the fact that direct limits commute with finite products. The case $n = 0$ is similar. 
\end{proof}


\subsection*{{Objects of type \tFP{\bm{n}} and $\bm{n}$-presentations}}

We now wish to give a characterization of  objects of type \tFP{n} in terms of $n$-presentations, similar in spirit to Example~\ref{Example-fg-projectives}. We start with the following useful lemma. It is a simple corollary to  Proposition~\ref{coro-properties of FP-n}(2).

\begin{lemma}\label{prop:n-presented}
Let $\mathcal{G}$ be locally finitely presented and $C \in \mathcal{G}$ an object for which there exists an exact sequence
\[
F_n \xrightarrow{f_n} F_{n-1} \to \cdots \to F_1 \xrightarrow{f_1} F_0 \xrightarrow{f_0} C \to 0
\]
with $F_i$ of type \tFP{n} for $0 \leq i \leq n$. Then $C$ is also of type \tFP{n}.\footnote{The case of $n = \infty$ is also true. In this case we assume the given resolution is of infinite length.}
\end{lemma}

\begin{proof}
Note that ${\rm Im}(f_n)$ is finitely generated by \cite[Lemma V.3.1 (i)]{stenstrom}. Thus ${\rm Im}(f_{n-1})$ is finitely presented by part (3) of Lemma \ref{lem:finitely_presented}. In fact, we may repeatedly apply the more general Proposition \ref{coro-properties of FP-n}(2) to conclude each ${\rm Im}(f_{n-i})$ is of type ${\rm FP}_i$, for each $0 \leq i \leq n \leq \infty$. In particular, $C = {\rm Im}(f_0)$ is of type ${\rm FP}_n$.

For $n = \infty$, we consider the truncated resolutions and use the fact that $\class{FP}_{\infty} = \bigcap_{n \geq 0} \class{FP}_n$. 
\end{proof}

Given a class of objects $\mathcal X$ in $\mathcal G$, we say that an object $C$  has an \emph{$n$-presentation by objects in $\mathcal X$} if there is an exact sequence
\[
X_n \xrightarrow{} X_{n-1} \to \cdots \to X_1 \xrightarrow{} X_0 \xrightarrow{} C \to 0
\]
with each $X_i \in \mathcal X$. For example, the object $C$ in Lemma \ref{prop:n-presented} has an $n$-presentation by objects in the class of objects of type $\text{FP}_n$.

Since a typical Grothendieck category need not have a set of projective generators (as shown in Example \ref{Example-fg-projectives} (5)), the following proposition and corollary are interesting. They provide an appropriate characterization of  objects of type \tFP{n} in terms of $n$-resolutions based on the generators.

\begin{proposition}\label{resolutions_for_FPn}
Assume $\mathcal{G}$ is locally type \tFP{n}, with $\mathcal{S}$ denoting a generating set of objects of type \tFP{n} .  Then $C \in \cat{G}$ is an object of type \tFP{n} if, and only if,   $C$ has an $n$-presentation by objects in $\textnormal{add}(\mathcal{S})$. 
\end{proposition}

\begin{proof}
Due to Propositions~\ref{coro-properties of FP-n}(4) and~\ref{prop-direct sums}, we have $\textnormal{add}(S) \subseteq \mathcal{FP}_n$. Thus the ``if'' part follows immediately from Lemma~\ref{prop:n-presented}. It only remains to prove the ``only if'' part. So consider $C \in \mathcal{FP}_n$. Then we can find an epimorphism $\bigoplus_{j \in J} G_j \twoheadrightarrow C$ with each $G_j \in \mathcal{S}$. Write $C = \sum_{j \in J} G'_j$ where $G'_j := {\rm Im}(G_j \rightarrowtail \bigoplus_{j \in J} G_j \twoheadrightarrow C)$. Since $C$ is finitely generated, there exists a finite subset $J' \subseteq J$ such that $C = \sum_{j \in J'} G'_j$. This means $\bigoplus_{j \in J'} G_j \twoheadrightarrow C$ is still an epimorphism. Moreover, $F_0 := \bigoplus_{j \in J'} G_j \in \textnormal{add}(\mathcal{S})$. We obtain a short exact sequence $$0 \to K_0 \to F_0 \to C \to 0,$$ and again $F_0 \in \textnormal{add}(S) \subseteq \mathcal{FP}_n$. Thus we have that $K_0$ is of type \tFP{n-1} by Proposition \ref{coro-properties of FP-n} (3). Continuing with this reasoning, we can find an exact sequence
\begin{align}\label{eqnKn}
0 & \to K_{n-1} \to F_{n-1} \to \cdots \to F_1 \to F_0 \to C \to 0
\end{align}
with $F_i \in \textnormal{add}(\mathcal{S})$ for every $0 \leq i < n$, and with $K_{n-1}$ finitely generated. Finally, we just take another epimorphism $F_n \twoheadrightarrow K_{n-1}$ with $F_n \in \textnormal{add}(\mathcal{S})$, and ``glue it'' with \eqref{eqnKn} to complete the proof. 

For $n = \infty$, we can continue indefinitely using the thickness property of Proposition~\ref{prop-thickness of FP-infinity}.
\end{proof}

We note that if $\cat{G}$ is locally of type \tFP{m} then it is also automatically locally of type \tFP{n} for any $n \leq m$. So the characterization of objects of type \tFP{n} given in Proposition~\ref{resolutions_for_FPn} will hold for all $n \leq m$ whenever $\cat{G}$ is locally of type \tFP{m}. In particular, taking $m = \infty$ we get the following characterization of  objects of type \tFP{n}.

\begin{corollary}\label{resolutions_for_FP-infty}
Assume $\mathcal{G}$ is locally type ${\rm FP}_{\infty}$ with $\mathcal{S}$ denoting a generating set of objects of type ${\rm FP}_{\infty}$. Then $C \in \cat{G}$ is an object of type \tFP{n} (for any $0 \leq n \leq \infty$) if, and only if, there exists an exact sequence 
\[
F_n \xrightarrow{f_n} F_{n-1} \to \cdots \to F_1 \xrightarrow{f_1} F_0 \xrightarrow{f_0} C \to 0
\]
with $F_i \in \textnormal{add}(\mathcal{S})$ for every $0 \leq i \leq n$. That is, $C$ is an   object of type \tFP{n} if, and only if, $C$ has an $n$-presentation by objects in $\textnormal{add}(\mathcal{S})$.\footnote{The case of $n = \infty$ gets interpreted as an infinite resolution.}
\end{corollary}

This specializes to give the following expected characterization for the case that $\cat{G}$ possesses a generating set of finitely generated projective objects.

\begin{corollary}\label{resolutions_for_proj}
Assume $\mathcal{G}$ possesses a generating set of finitely generated projective objects. Then, $C \in \cat{G}$ is an object of type \tFP{n} (for any $0 \leq n \leq \infty$) if, and only if, there exists an exact sequence 
\[
P_n \to P_{n-1} \to \cdots \to P_1 \to P_0 \to C \to 0
\]
where $P_i$ is finitely generated projective for every $0 \leq i \leq n$. That is, $C$ is an  object of type \tFP{n} if, and only if, $C$ has an $n$-presentation in the sense of Example~\ref{Example-fg-projectives}.
\end{corollary}

\begin{proof}
Any finitely generated projective object is of type ${\rm FP}_{\infty}$ by Example~\ref{Example-fg-projectives}(1). So taking $\mathcal{S}$ to be a set of finitely generated projective generators,  Corollary~\ref{resolutions_for_FP-infty} applies, and  in this case $\textnormal{add}(\mathcal{S})$ is precisely the class of finitely generated projective objects. 
\end{proof}

\begin{remark}\label{rem:locally_type_FPinfty}
If $\mathcal{G}$ is any of the categories $R\mbox{-}\textrm{Mod}$, $\textrm{Ch}(R)$ or $\textrm{Fun}(\mathcal{C}^{\rm op},\mathsf{Ab})$, then $\mathcal{G}$ admits a collection of finitely generated projective generators. Hence, in particular, in the present article we recover several results of \cite{bravo-perez} and \cite{ZhaoPerez}.
\end{remark}


\section{Injectivity relative to objects of type $\textrm{FP}_n$}\label{sec:injectivity}

In this section we study the cotorsion pair cogenerated by all the  objects of type \tFP{n}. One may now wish to review the definitions associated to cotorsion pairs from the preliminaries.  

The following brings \cite[Definition 3.1]{bravo-perez} to the context of Grothendieck categories.

\begin{definition}\label{def:FPn-injective}
We say an object $A \in \cat{G}$ is \textbf{{FP}$_{\bm{n}}$-injective} if $\Ext^1_{\cat{G}}(F,A) = 0$ for all $F \in \class{FP}_n$. We denote the class of all ${\rm FP}_n$-injective objects by $\class{FP}_n\textnormal{-Inj}$. So note that $\class{FP}_n\textnormal{-Inj} = \class{FP}_n^{\perp_1}$.
\end{definition}

The definition includes the cases $n=0$ and $n=\infty$. Assuming $\cat{G}$ is locally finitely generated, the $\textrm{FP}_0$-injectives are the usual injective objects. One can prove this by using the analog of Baer's criterion that holds in Grothendieck categories \cite[Proposition V.2.9]{stenstrom}, along with the fact that any epimorphic image of a finitely generated object is again finitely generated \cite[Proposition V.3.1(i)]{stenstrom}. For the case $n=\infty$, the ${\rm FP}_{\infty}$-injectives are the \emph{absolutely clean} objects studied in \cite{gillespie-models-of-injectives} and \cite{bravo-gillespie-hovey}. The case $n=1$ gives us the \emph{absolutely pure} (FP-injective) objects studied in \cite{stovicek-purity,stenstrom2}.

\begin{example}\label{ex:FPn-injectives}
We present description of $\text{FP}_n$-injective objects for some categories studied in the previous section. 
\begin{enumerate}
\item \textbf{$\bm{\text{FP}_n}$-injective complexes}. The class $\mathcal{FP}_n\mbox{-}\textrm{Inj}$ in the category $\textrm{Ch}(R)$ of complexes is defined and studied in \cite[Definition 2.3.1]{ZhaoPerez}. These complexes are characterized as those $X \in \textrm{Ch}(R)$ such that $X$ is exact and each cycle module $Z_m(X)$ is $\text{FP}_n$-injective in $R\mbox{-}\textrm{Mod}$. (See \cite[Theorem 2.3.3]{ZhaoPerez}).

\item \textbf{$\text{FP}_{\bm{n}}$-injective modules over ringed spaces}. For any ringed space $(X,\mathcal{O}_X)$, an $\mathcal{O}_X$-module $\mathscr{A}$ is $\text{FP}_n$-injective if, and only if, $\mathscr{A}|_{U}$ is an $\text{FP}_n$-injective $\mathcal{O}|_U$-module for every open subset $U \subseteq X$. (See \cite[Proposition 2.7]{Estrada-Gillespie}).\footnote{The statement and proof are formulated for absolutely clean $\mathcal{O}_X$-modules, but the arguments also work for $\text{FP}_n$-injectives.} 

\item \textbf{$\text{FP}$-injective functors}. Concerning the functor category $\textrm{Fun}(\mathcal{C}^{\rm op},\mathsf{Ab})$, there is a characterization of $\text{FP}$-injective functors in the case where $\mathcal{C}$ is an additive category with kernels. Namely, an additive functor $G \colon \mathcal{C}^{\rm op} \longrightarrow \mathsf{Ab}$ is $\text{FP}$-injective if, and only if, $G$ is right exact, that is, $G$ maps kernels in $\mathcal{C}$ into cokernels in $\mathsf{Ab}$. (See \cite[Corollary 2.3.4]{Dean}). 

A similar description for $\text{FP}$-injective functors holds true with a slightly weaker assumption on $\mathcal{C}$, namely, that $\mathcal{C}$ has pseudo-kernels. Recall that given two morphisms $f_2 \colon X_2 \to X_1$ and $f_1 \colon X_1 \to X_0$ in $\mathcal{C}$, $f_2$ is a \emph{pseudo-kernel} of $f_1$ if $f_1 \circ f_2 = 0$ and if for every morphism $g \colon Y \to X_1$ satisfying $f_1 \circ h = 0$, there exists $h \colon Y \to X_2$ (not necessa-rily unique!) such that $g = f_2 \circ h$. The following two conditions are equivalent for every additive functor $G \colon \mathcal{C}^{\rm op} \longrightarrow \mathsf{Ab}$ provided that $\mathcal{C}$ has pseudo-kernels:
\begin{itemize}
\item[(a)] $G$ is $\text{FP}$-injective. 

\item[(b)] $G$ maps pseudo-kernels in $\mathcal{C}$ into pseudo-cokernels in $\mathsf{Ab}$. 
\end{itemize}
The proof follows as in \cite[Corollary 2.3.4]{Dean}. 

For the case $n > 1$, we can also obtain the previous equivalence for any additive category $\mathcal{C}$. (See Appendix C). 
\end{enumerate}
\end{example}

Next we shall fix some notation that will be used throughout this section. To do so, recall that the category of all   objects of type \tFP{n} is skeletally small, meaning, the collection of (isomorphism classes of) objects of type $\text{FP}_n$ is a set, not just a proper class. (Reason: Grothendieck categories are locally presentable so the facts from \cite{adamek-rosicky} and \cite[Appendix]{gillespie-quasi-coherent} apply. In particular, it follows from \cite[Appendix, Fact A.9]{gillespie-quasi-coherent}.)

\begin{notation}\label{notation-representatives}
As commented above, we may choose a set, not just a proper class, of isomorphism representatives for each class $\class{FP}_{n}$. We shall always denote this set by ${\rm FP}_{n}(\cat{G})$. We then let $I_n$ denote the set of all inclusions of subobjects $K \rightarrowtail F$ with $F \in FP_{n}(\cat{G})$ and such that $F/K$ is also of type \tFP{n}. (If $\cat{G}$ is locally finitely presented, then by Proposition~\ref{coro-properties of FP-n},  it is equivalent to require that $K$ be of type \tFP{n-1}.)
\end{notation}

\begin{definition}[$I$-injectives]\label{def-I-injective}
Let $I$ be any set of monomorphisms in $\cat{G}$. We shall say that an object $C \in \cat{G}$ is \emph{$I$-injective} if for every monomorphism $(K \rightarrowtail F) \in I$, each morphism $K \xrightarrow{} C$ extends over $F$.
\end{definition}

For example, Baer's Criterion states that a (left) $R$-module is injective if and only if it is $I$-injective with respect to the set $I$ of all inclusions of (left) ideals into $R$. The following is a sort of generalization of this for the sets $I_n$ from Notation~\ref{notation-representatives}.

\begin{proposition}\label{prop-Baer-like}
Let $\cat{G}$ be a locally of type \tFP{n} category. Let $I_n$ be the set of monomorphisms from Notation~\ref{notation-representatives}. Then $A \in \cat{G}$ is ${\rm FP}_n$-injective if and only if $A$ is $I_n$-injective.
\end{proposition}

\begin{proof}
The ``only if'' part is clear using the $\Ext^i_{\cat{G}}(-,A)$ sequence, because the definition of $I_n$ assumes each $F/K \in \class{FP}_n$.

For the converse,  given  $A$ an $I_n$-injective in $\class{G}$, we aim to show that $\Ext^1_{\class{G}}(F,A)$ for all $F \in \FP{n}$. To do this we consider a short exact sequence 
\begin{align}\label{eqn:pF}
0 & \to A \to X \xrightarrow{p} F \to 0
\end{align}
with $F$ of type \tFP{n} and show that any such sequence is split, using the Yoneda description of $\Ext^1_{\cat{G}}(F,A)$. 

By \cite[Lemma V.3.3]{stenstrom}, we can find a finitely generated subobject $S \subseteq X$, in the short exact sequence above, such that $p(S) = F$. Now since ${\rm FP}_n(\cat{G})$ is a generating set we can find an epimorphism $\bigoplus_{j \in J} F_j \twoheadrightarrow S$ with each $F_j \in FP_n(\cat{G})$. As in the proof of Proposition \ref{resolutions_for_FPn}, we can find a finite subset $J' \subseteq J$ such that $q : \bigoplus_{j \in J'} F_j \twoheadrightarrow S$ is still an epimorphism. Moreover, $\bigoplus_{j \in J'} F_j$ is of type \tFP{n} by Proposition \ref{prop-direct sums}. Without loss of generality, we assume $\bigoplus_{j \in J'} F_j \in FP_n(\cat{G})$. Letting $K$ denote the pullback of $A \xrightarrow{} X \xleftarrow{q} \bigoplus_{j \in J'} F_j$, one constructs a morphism of short exact sequences:
\[
\begin{tikzpicture}[description/.style={fill=white,inner sep=2pt}]
\matrix (m) [matrix of math nodes, row sep=3em, column sep=3em, text height=1.25ex, text depth=0.25ex]
{ 
0 & K & \bigoplus_{j \in J'} F_j & F & 0 \\
0 & A & X & F & 0 \\
};
\path[->]
(m-1-2)-- node[pos=0.5] {\footnotesize$\mbox{\bf pb}$} (m-2-3)
(m-1-1) edge (m-1-2) (m-1-2) edge (m-1-3) (m-1-3) edge node[above] {\footnotesize$p \circ q$} (m-1-4) (m-1-4) edge (m-1-5)
(m-2-1) edge (m-2-2) 
(m-2-2) edge (m-2-3) 
(m-2-3) edge node[below] {\footnotesize$p$} (m-2-4) 
(m-2-4) edge (m-2-5)
(m-1-2) edge (m-2-2) 
(m-1-3) edge node[right] {\footnotesize$q$} (m-2-3) 
;
\path[-,font=\scriptsize]
(m-1-4) edge [double, thick, double distance=2pt] (m-2-4)
;
\end{tikzpicture}
\]
The inclusion map $K \rightarrowtail \bigoplus_{j \in J'} F_j$ is in the set $I_n$ from Notation \ref{notation-representatives}, so the assumption on $A$ means there is a morphism $\bigoplus_{j \in J'} F_j \xrightarrow{} A$ producing a commutative triangle in the upper left corner. This is equivalent, by a fact sometimes called ``the homotopy lemma'' (see \cite[Lemma 7.16]{wisbauer}), to a map $F \xrightarrow{} X$ producing a commutative triangle in the lower right corner.  This is precisely a splitting of the short exact sequence \eqref{eqn:pF}.
\end{proof}

We now prove the main result of this section. Note that the class $\class{FP}_n$ cogenerates a cotorsion pair $({}^{\perp_1}(\class{FP}_n\textnormal{-Inj}),\class{FP}_n\textnormal{-Inj})$. We shall call this the \emph{${FP}_n$-injective cotorsion pair}.

\begin{theorem}[completeness of the $\class{FP}_n$-injective cotorsion pair]\label{theorem-FP_n-injective cotorsion pair}
Let $0 \leq n \leq \infty$ and let $\cat{G}$ be a locally type \tFP{n} category. Then, $({}^{\perp_1}(\mathcal{FP}_n\textnormal{-Inj}),\class{FP}_n\textnormal{-Inj})$ is a functorially complete cotorsion pair. 

In fact, $({}^{\perp_1}(\mathcal{FP}_n\textnormal{-Inj}),\class{FP}_n\textnormal{-Inj})$ is a small cotorsion pair with $I_n$ a set of generating monomorphisms in the sense of \cite[Definition 6.4]{hovey}. Moreover, if $n \geq 2$, then $({}^{\perp_1}(\mathcal{FP}_n\textnormal{-Inj}),\class{FP}_n\textnormal{-Inj})$ is a finite cotorsion pair, meaning, $I_n$ is a set of finite generating monomorphisms in the sense that the domains and codomains of each morphism are not just small, but finite in the sense of \cite[Definition 2.1.4 and Section 7.4]{hovey-model-categories}.
\end{theorem}

\begin{proof}
The proof is based on the work in \cite{hovey} and \cite{saorin-stovicek}. First, referring to \cite[Theorem 6.5]{hovey} we see that the set $I_n$ from Notation \ref{notation-representatives} is indeed a set of \emph{generating monomorphisms} for a \emph{small}, and hence functorially complete, cotorsion pair (in the sense of \cite[Definition 6.4]{hovey}). Proposition \ref{prop-Baer-like} makes it clear that this is indeed the ${\rm FP}_n$-injective cotorsion pair. In the context of Grothendieck categories, \emph{finite} in the sense of \cite[Definition 2.1.4, Section 7.4]{hovey-model-categories} coincides with \emph{finitely presented}. So if $n \geq 2$, then all domains and codomains of maps in $I_n$ are finite by Proposition \ref{coro-properties of FP-n}. 
\end{proof}

The following result extends \cite[Proposition 3.5]{bravo-estrada-iacob} and \cite[Corollary 4.3.2]{ZhaoPerez} by allowing for left approximations by FP$_n$-injective objects in any Grothendieck category.

\begin{corollary}[existence of ${\rm FP}_n$-injective preenvelopes]\label{coro-preenveloping}
Let $0 \leq n \leq \infty$ and let $\cat{G}$ be a locally type \tFP{n} category. Then $\class{FP}_n\textnormal{-Inj}$ is a special preenveloping class. 
\end{corollary}

Finally, we close this section by giving a characterization of objects of type \tFP{n} in terms of the orthogonal complement ${}^{\perp_1}(\mathcal{FP}_n\text{-Inj})$. This will provide in the next section one of the alternative descriptions for $n$-coherent categories with $n \geq 2$.

Fix an injective cogenerator $E \in \mathcal{G}$. We can construct a functor $\Psi \colon \mathcal{G} \longrightarrow \mathcal{G}$ given by 
\[
X \mapsto E^{{\rm Hom}_{\mathcal{G}}(X,E)} := \prod_{h \in {\rm Hom}_{\mathcal{G}}(X,E)} E_h
\]
with $E_h = E$. Let $\varinjlim {\rm Im}(\Psi)$ denote the class of objects of $\mathcal{G}$ which are a direct limit of a direct system in ${\rm Im}(\Psi)$. In \cite[Theorem B.1]{bravo-parra}, it is proved that for $n \geq 2$, an object $M \in \mathcal{G}$ is of type $\text{FP}_n$ if, and only if, $M$ is of type $\text{FP}_{n-1}$ and $\varinjlim {\rm Im}(\Psi) \subseteq {\rm Ker}({\rm Ext}^{n-1}_{\mathcal{G}}(M,-))$. From this equivalence we can prove the following result.

\begin{proposition}\label{FPn-in-terms-of-FPn-Inj}
Let $\class{G}$ be a Grothendieck category, $M$ an object in $\class{G}$ and $n \geq 2$. Then $M \in \FP{n}$, if and only if, $M \in \FP{n-1}$ and $M \in {}^{\perp_{1}}\class{FP}_{n}\textnormal{-Inj}$. That is, $\mathcal{FP}_n = \mathcal{FP}_{n-1} \cap {}^{\perp_1}(\mathcal{FP}_n\textnormal{-Inj})$, for all $n \geq 2$. 
\end{proposition}

\begin{proof}
The ``only if'' part is clear. Now suppose that $M$ is an object of type $\text{FP}_{n-1}$ such that ${\rm Ext}^1_{\mathcal{G}}(M,N) = 0$ for every $N \in \class{FP}_{n}\textnormal{-Inj}$. By \cite[Theorem B.1]{bravo-parra}, we show that $M \in \mathcal{FP}_n$ by proving the containment $\varinjlim {\rm Im}(\Psi) \subseteq {\rm Ker}({\rm Ext}^{n-1}_{\mathcal{G}}(M,-))$. For, let us study the two cases $n = 2$ and $n > 2$. 

For the case $n = 2$, let $N \in \varinjlim {\rm Im}(\Psi)$ and write $N = \varinjlim_{i \in I} N_i$ where $N_i \in {\rm Im}(\Psi)$ for every $i \in I$. Note that each $N_i$ is injective since it is a product of injective objects. Thus, $N$ is $\text{FP}_2$-injective, and so ${\rm Ext}^1_{\mathcal{G}}(M,N) = 0$. Then, the containment $\varinjlim {\rm Im}(\Psi) \subseteq {\rm Ker}({\rm Ext}^{1}_{\mathcal{G}}(M,-))$ follows. 

For the case $n > 2$, consider again $N \in \varinjlim {\rm Im}(\Psi)$ along with a partial injective resolution 
\[
0 \to N \to E^0 \to E^1 \to \cdots \to E^{n-3} \to N' \to 0
\]
By dimension shifting, we have that ${\rm Ext}^{n-1}_{\mathcal{G}}(M,N)  \cong {\rm Ext}^1_{\mathcal{G}}(M,N')$. Using again dimension shifting along with \cite[Theorem B.1]{bravo-parra}, one can note that $N' \in \mathcal{FP}_n\mbox{-}\textrm{Inj}$. This implies that ${\rm Ext}^1_{\mathcal{G}}(M,N') = 0$. Hence, ${\rm Ext}^{n-1}_{\mathcal{G}}(M,N) = 0$. 
\end{proof}

\begin{remark}
The previous proposition holds for the case $n = 1$ when $\mathcal{G}$ is the category of modules over a ring. In fact, this is due to Glaz \cite[Theorem 2.1.10]{Glaz}. Specifically, the equality $\mathcal{FP}_1 = \mathcal{FP}_{0} \cap {}^{\perp_1}(\mathcal{FP}_1\text{-Inj})$ holds in $R\mbox{-}\textrm{Mod}$. 

Although we are not aware if the same equality holds in any Grothendieck category, we can prove that it does in the category $\textrm{Ch}(R)$ of complexes of modules and also in the category $\mathfrak{Qcoh}(X)$ of quasi-coherent sheaves over certain schemes $X$. (See Appendix B for details).  
\end{remark}


\section{$n$-coherent objects and categories}\label{Sec-locally n-coherent cats}

In Section \ref{Sec-locally finitely n-presented cats} we introduced the Grothendieck categories that are natural generalizations of locally finitely generated and locally finitely presented categories. We now take it a step further and introduce the natural generalizations of locally noetherian and locally coherent categories, which we call $n$-coherent. We begin by looking at the objects which generate such categories: the $n$-coherent objects.

\begin{definition} \label{def-finitely-n-coherent}
Let $n$ be given with $0 \leq n \leq \infty$. We say that an object $C \in \cat{G}$ is \emph{$n$-coherent} if  each of the following hold.
\begin{enumerate}
\item $C \in \class{FP}_n$. That is, $C$ is of type \tFP{n}. 
\item For each subobject $S \subseteq C$, we have $S \in \class{FP}_{n-1}$ implies $S \in \class{FP}_n$. That is,  every  subobject of $C$ of type \tFP{n-1} is in fact of type \tFP{n}. 
\end{enumerate}
We shall let $\class{C}_n$ denote the class of all $n$-coherent objects in $\cat{G}$.
For the case $n=\infty$, we consider all objects of type ${\rm FP}_{\infty}$ to be $\infty$-coherent.
\end{definition}

\begin{remark} Recall that $\FP{-1}$ is the whole class of objects of $\class{G}$. Then by \cite[Proposition V.4.1]{stenstrom}, an object is noetherian in the usual sense if, and only if, it is 0-coherent in the sense of Definition \ref{def-finitely-n-coherent}. Moreover, a 1-coherent object coincides exactly with the standard definition of a coherent object.
\end{remark}

Proposition~\ref{coro-properties of FP-n} holds for \emph{all} values of $n$ for which $0 \leq n \leq \infty$. From that proposition we may now easily prove the following result.

\begin{proposition}[closure properties of $\mathcal{C}_n$]\label{them-n-coherent thick}
Assume $\cat{G}$ is locally finitely presented. Then, the class $\class{C}_n$, of all $n$-coherent objects satisfies the following properties. 
\begin{enumerate}
\item $\class{C}_n$ is closed under direct summands. 

\item Suppose we have a short exact sequence 
\begin{align}\label{eqn:seqE}
\mathbb{E} \colon & 0 \xrightarrow{} A \xrightarrow{} B \xrightarrow{} C \xrightarrow{} 0
\end{align}
with $B \in \class{C}_n$. Then $A \in \class{C}_n$ if, and only if, $C \in \class{C}_n$.
\end{enumerate}
Thus $\class{C}_n$ is a thick class if, and only if, $\class{C}_n$ is closed under extensions. 
\end{proposition}

\begin{proof}
First, we note that it is clear from Definition~\ref{def-finitely-n-coherent} and Proposition~\ref{coro-properties of FP-n}(4) that $\class{C}_n$ is closed under direct summands. For the remainder of the proof we fix a short exact sequence as $\mathbb{E}$ \eqref{eqn:seqE} above with $B \in \class{C}_n$. 

First we assume $A \in \class{C}_n$. In fact, to prove $C \in \class{C}_n$ we only need to assume $A \in \class{FP}_{n-1}$.
Indeed we start by noting $C \in \class{FP}_n$, by Proposition~\ref{coro-properties of FP-n}(2).  Now suppose $S \subseteq C$ is of type \tFP{n-1}. (We must show that $S$ is of type \tFP{n}.)  We let $P$ denote the pullback of $B \twoheadrightarrow C \leftarrowtail S$ and we obtain a morphism of short exact sequences:
\[
\begin{tikzpicture}[description/.style={fill=white,inner sep=2pt}]
\matrix (m) [matrix of math nodes, row sep=3em, column sep=3em, text height=1.25ex, text depth=0.25ex]
{ 
0 & A & P & S & 0 \\
0 & A & B & C & 0 \\
};
\path[->]
(m-1-3)-- node[pos=0.5] {\footnotesize$\mbox{\bf pb}$} (m-2-4)
(m-1-1) edge (m-1-2) (m-1-2) edge (m-1-3) (m-1-3) edge (m-1-4) (m-1-4) edge (m-1-5)
(m-2-1) edge (m-2-2) 
(m-2-2) edge (m-2-3) 
(m-2-3) edge (m-2-4) 
(m-2-4) edge (m-2-5)
(m-1-4) edge (m-2-4) 
(m-1-3) edge (m-2-3)
;
\path[-,font=\scriptsize]
(m-1-2) edge [double, thick, double distance=2pt] (m-2-2)
;
\end{tikzpicture}
\]
where $P \subseteq B$ is a subobject, which must be of type \tFP{n-1} by Proposition~\ref{coro-properties of FP-n}(1). Thus $P$ is of type $\text{FP}_n$ since $B$ is $n$-coherent. But we now turn around and again apply Proposition~\ref{coro-properties of FP-n}(2) to conclude $S$ too is of type $\text{FP}_n$.

Last,  suppose $C \in \class{C}_n$. In fact, to show $A \in \class{C}_n$ we only need that $C \in \class{FP}_{n}$.
Indeed in this case we start by noting $A \in \class{FP}_{n-1}$ by Proposition~\ref{coro-properties of FP-n}(3). Thus $A \in \class{FP}_{n}$ since $B$ is $n$-coherent. In fact it is now clear that the $n$-coherence of $A$ is immediately inherited from $B$.
\end{proof}

It is well known that the class of noetherian objects (0-coherent objects) is closed under extensions \cite[Proposition V.4.2]{stenstrom}. It is also true that the usual coherent objects (1-coherent objects) are closed under extensions \cite[Proposition 1.5]{herzog-ziegler-spectrum}. In general, we see no reason why $\class{C}_n$ would be closed under extensions for $2 \leq n < \infty$. But we do have the following.

\begin{lemma}\label{lemma-extensions}
Let $\class{S}$ be a set of $n$-coherent generators for $\cat{G}$. The following are equivalent.  
\begin{itemize}
\item[(a)] $\class{C}_n$ is closed under extensions.

\item[(b)] $\class{C}_n$ is closed under finite direct sums. 

\item[(c)]  $\textnormal{free}(\mathcal{S}) \subseteq \class{C}_n$. That is, any finite direct sum of the generators is $n$-coherent. 
\end{itemize}
In particular, by Proposition \ref{them-n-coherent thick}, $\class{C}_n$ is a thick class if, and only if, any one of the above holds. 
\end{lemma}

\begin{proof}
The implications (a) $\implies$ (b) $\implies$ (c) are clear. To see (c) $\implies$ (a), let $\class{S} = \{C_j\}$ denote a generating set for $\cat{G}$ with each $C_j$ an $n$-coherent object and let 
\[
0 \to A \to B \to C \to 0
\] 
be a short exact sequence with $A , C \in \class{C}_n$. As in the proof of Proposition~\ref{prop-Baer-like} we may construct a pullback diagram
\[
\begin{tikzpicture}[description/.style={fill=white,inner sep=2pt}]
\matrix (m) [matrix of math nodes, row sep=2.5em, column sep=2.5em, text height=2.5ex, text depth=1.5ex]
{ 
{} & K & K \\
0 & P & \displaystyle \bigoplus_{j \in I} C_j & C & 0 \\
0 & A & B & C & 0 \\
};
\path[->]
(m-2-2)-- node[pos=0.5] {\footnotesize$\mbox{\bf pb}$} (m-3-3)
(m-2-1) edge (m-2-2) (m-2-2) edge (m-2-3) (m-2-3) edge (m-2-4) (m-2-4) edge (m-2-5)
(m-3-1) edge (m-3-2) (m-3-2) edge (m-3-3) (m-3-3) edge (m-3-4) (m-3-4) edge (m-3-5)
;
\path[>->]
(m-1-2) edge (m-2-2) (m-1-3) edge (m-2-3)
;
\path[->>]
(m-2-3) edge (m-3-3) (m-2-2) edge (m-3-2)
;
\path[-,font=\scriptsize]
(m-1-2) edge [double, thick, double distance=2pt] (m-1-3)
(m-2-4) edge [double, thick, double distance=2pt] (m-3-4)
;
\end{tikzpicture}
\]
where $\bigoplus_{j \in I} C_j$ is a \emph{finite} direct sum of objects in the generating set $\mathcal{S}$. By assumption $\bigoplus_{j \in I} C_j$ is $n$-coherent.
Thus by Proposition~\ref{coro-properties of FP-n}(3) we conclude both $K$ and $P$ are in $\class{FP}_{n-1}$, whence $K,P \in \class{FP}_n$, and in fact $K,P \in \class{C}_n$. But then we conclude from Proposition~\ref{them-n-coherent thick} that $B \in \class{C}_n$.
\end{proof}

\begin{corollary}[$\mathcal{C}_n$ is closed under quotients in $\mathcal{FP}_n$] \label{mixed-closure-1}
Let 
\[
0 \to A \to B \to C \to 0
\]
be a short exact sequence with $B \in \mathcal{C}_n$ and $A \in \FP{n-1}$. Then, $C \in \mathcal{C}_n$. 

In particular, if an object $F$ of type $\text{FP}_n$ is a quotient of an $n$-coherent object, then $F$ is also $n$-coherent. 
\end{corollary}

\begin{proof}
The first part of the statement follows from the proof of part (2) of Proposition \ref{them-n-coherent thick}. For the second part, let $F \in \mathcal{FP}_n$ such that there is an epimorphism $\varphi \colon B \twoheadrightarrow F$ with $B \in \mathcal{C}_n$. It follows that $F$ is $n$-coherent since ${\rm Ker}(\varphi) \in \mathcal{FP}_{n-1}$ by Proposition \ref{coro-properties of FP-n}.
\end{proof}


\subsection*{{$\boldsymbol{n}$-coherent categories}}

We now state the corresponding generalization of locally noetherian and locally coherent categories.

\begin{definition} \label{def-n-coherent category}
Let $\cat{G}$ be a Grothendieck category. We say that $\mathcal{G}$ is \textbf{$n$-coherent} if it is locally type \tFP{n} and each object of type \tFP{n} is $n$-coherent. 
\end{definition}

Note that every $n$-coherent category has a generating set consisting of $n$-coherent objects. Combining this and other conditions, we have the following characterization of $n$-coherent categories.

\begin{theorem}[characterizations of $n$-coherent categories]\label{them-locally n-coherent}
Let $0 \leq n \leq \infty$ and assume that $\cat{G}$ is a locally type \tFP{n} category. The following are equivalent:  
\begin{itemize}
\item[(a)] $\cat{G}$ is $n$-coherent. That is, every   object of type \tFP{n} is $n$-coherent. 

\item[(b)] The class $\class{FP}_n$ is closed under taking kernels of epimorphisms between its objects.

\item[(c)] $\class{FP}_n$ is thick.

\item[(d)] $\cat{G}$ has a generating set of $n$-coherent objects, and satisfies one of the equivalent conditions of Lemma \ref{lemma-extensions}.

\item[(e)] The objects of type ${\rm FP}_{\infty}$ coincide with the  objects of type \tFP{n}. That is, $\class{FP}_n = \class{FP}_{\infty}$.
\end{itemize}
Moreover, if $n \geq 1$, these are also equivalent to: 
\begin{itemize}
\item[(f)] The class of $\class{FP}_n\textnormal{-Inj}$, of all ${\rm FP}_n$-injectives, is closed under taking cokernels of monomorphisms between its objects. 

\item[(g)] The ${\rm FP}_n$-injective cotorsion pair, $({}^{\perp_1}(\mathcal{FP}_n\textnormal{-Inj}),\class{FP}_n\textnormal{-Inj})$, is hereditary. 

\item[(h)] $\class{FP}_n\textnormal{-Inj}$ coincides with the class $\class{FP}_{\infty}\textnormal{-Inj}$ of absolutely clean objects. 

\item[(i)] $\mathcal{FP}_{n+1}\textnormal{-Inj} \subseteq \mathcal{FP}_n\textnormal{-Inj}$. 
\end{itemize}
In particular, $\cat{G}$ is $0$-coherent if and only if it is locally noetherian, and it is $1$-coherent if and only if it is locally coherent in the usual sense. 
\end{theorem}

Before proving the theorem we note the following remark.

\begin{remark}
A couple of trivial observations may be helpful. 
\begin{enumerate} 
\item For condition (d) in the cases $n = 0$ and $n = 1$, any generating set for $\mathcal{G}$ of noetherian or coherent satisfies the conditions of Lemma \ref{lemma-extensions}.

\item Note that conditions (f) and (g) of Theorem~\ref{them-locally n-coherent} cannot possibly be equivalent to the first five for case $n = 0$. Indeed the canonical injective cotorsion pair is \emph{always} hereditary, so it would imply that all locally finitely generated categories are locally noetherian.

\item At first glance it is natural to desire a characterization of $n$-coherent categories in terms of the closure of $\class{FP}_n\textnormal{-Inj}$ under direct limits. Afterall, these are important characterizations for $n = 0$ and $n=1$. But of course $\class{FP}_n\textnormal{-Inj}$ is always closed under direct limits for $n \geq 2$.
\end{enumerate}
\end{remark}

\begin{proof}
First, note that (a) and (b) are immediately seen to be equivalent by using parts (2) and (3) of Proposition \ref{coro-properties of FP-n}. Also, (b) and (c) are equivalent by that same proposition. 

Now if (a) is true, then $\cat{G}$ has a set of $n$-coherent generators. Moreover, if (a) is true, then $\class{C}_n = \class{FP}_n$ satisfies all of the equivalent conditions of Lemma \ref{lemma-extensions}, again because of Proposition \ref{coro-properties of FP-n}. Thus (a) implies (d).

We now show (d) $\Longrightarrow$ (b). To do so, let $\{C_j\}$ denote a generating set for $\cat{G}$ with each $C_j$ an $n$-coherent object and let 
\[
0 \xrightarrow{} A \xrightarrow{}  B \xrightarrow{}  C \xrightarrow{} 0
\] 
be a short exact sequence with $B , C \in \class{FP}_n$. Construct a diagram as in the proof of Lemma \ref{lemma-extensions}, where $\bigoplus_{j \in I} C_j$ is a \emph{finite} direct sum of objects in the generating set. By the hypothesis $\bigoplus_{j \in I} C_j$ is $n$-coherent. Thus by Proposition \ref{coro-properties of FP-n}(3) we conclude both $K$ and $P$ are in $\class{FP}_{n-1}$, whence $K, P \in \class{FP}_n$. But then we conclude from Proposition \ref{coro-properties of FP-n}(2) that $A \in \class{FP}_n$, proving (b). So far we have shown (a) through (d) are equivalent. 

Assuming (c), then of course any set $\class{S}$ of isomorphism representatives for $\class{FP}_n$ is also thick. Thus the hypotheses of \cite[Lemma 3.6(4)]{gillespie-models-of-injectives}, with $\class{S} = FP_{n}(\cat{G})$ as in Notation \ref{notation-representatives}, are satisfied. One can check that the lemma explicitly proves for us that the cotorsioin pair of Theorem \ref{theorem-FP_n-injective cotorsion pair}, is hereditary. This proves (c) $\Longrightarrow$ (g).

Next, we show (g) $\Longrightarrow$ (e) whenever $n \geq 1$. But the proof will also show (d) $\Longrightarrow$ (e) for the special case $n = 0$. Indeed with either hypothesis, we note that the cotorsion pair $({}^{\perp_1}(\mathcal{FP}_n\textnormal{-Inj}),\class{FP}_n\textnormal{-Inj})$ is hereditary and the class $\class{FP}_n\textnormal{-Inj}$, of all $\textrm{FP}_n$-injective objects is closed under direct limits. (For the latter fact, the cases $n >1$ are immediate, the case $n = 0$ is well-known, and the case $n=1$ can be found in \cite[Proposition B.3]{stovicek-purity}.) So now to prove (e), we let $F$ be an object of type $\text{FP}_n$, and we shall show that the functors $\Ext^k_{\cat{G}}(F,-)$ preserve direct limits for all $k \geq 0$. By \cite[Corollary 1.7]{adamek-rosicky} it is enough to show that they preserve well-ordered direct limits. So let $\{X_{\alpha}\}_{\alpha < \lambda}$ be a well-ordered system, where $\lambda$ is some ordinal. Now we know from Theorem \ref{theorem-FP_n-injective cotorsion pair} that $({}^{\perp_1}(\mathcal{FP}_n\textnormal{-Inj}),\class{FP}_n\textnormal{-Inj})$ is functorially complete. So using that we have enough functorial injectives, we can, for each $X_{\alpha}$, find an $\textrm{FP}_n$-injective coresolution $X_{\alpha} \rightarrowtail A_{X_{\alpha}}$, so that the direct system $\{X_{\alpha}\}$ extends to a direct system $\{A_{X_{\alpha}}\}$. Moreover, the class $\class{FP}_n\textnormal{-Inj}$ is closed under direct limits. Thus by exactness of direct limits we get that $\varinjlim X_{\alpha} \rightarrowtail \varinjlim A_{X_{\alpha}}$ is again an $\textrm{FP}_n$-injective coresolution of $\varinjlim X_{\alpha} $. Since $({}^{\perp_1}(\mathcal{FP}_n\textnormal{-Inj})_n ,\class{FP}_n\textnormal{-Inj})$ is hereditary, $\Ext^k_{\cat{G}}(F,A) = 0$ for all $k \geq 1$ and $\textrm{FP}_n$-injective $A$.  In other words, $\textrm{FP}_n$-injective objects are $\Hom_{\cat{G}}(F,-)$-acyclic, and it follows that we can compute $\Ext^k_{\cat{G}}(F,-)$ via $\textrm{FP}_n$-injective coresolutions; see, for example, \cite[Theorem XX.6.2]{lang}. So now we can compute:
\[
\varinjlim \Ext^k_{\cat{G}}(F,X_{\alpha}) \cong \varinjlim H^k[\Hom(F,A_{X_{\alpha}})] \cong  H^k[\Hom(F, \varinjlim A_{X_{\alpha}})]  \cong \Ext^k_{\cat{G}}(F,\varinjlim X_{\alpha}).
\]
This means that the canonical map $\xi_k \colon \varinjlim \Ext^k_{\cat{G}}(F,X_{\alpha}) \xrightarrow{} \Ext^k_{\cat{G}}(F,\varinjlim X_{\alpha})$ is an isomorphism and completes the proof that $F$ is of type ${\rm FP}_{\infty}$.

Note that (e) implies (c) by Proposition \ref{prop-thickness of FP-infinity}. We now show that show (f) and (g) are equivalent. But (g) implies (f) is trivial and a standard argument shows that (f) implies (g). Indeed given any $X \in {}^{\perp_1}(\mathcal{FP}_n\textnormal{-Inj})$ and $Y \in \class{FP}_n\textnormal{-Inj}$, let 
\[
0 \xrightarrow{} Y \xrightarrow{} I \xrightarrow{}  Y'  \xrightarrow{} 0
\]
be a short exact sequence with $I$ injective. Then $Y' \in \class{FP}_n\textnormal{-Inj}$ by assumption. Thus the exactness of 
\[
\Ext^1_{\cat{G}}(X,Y')  \xrightarrow{}   \Ext^2_{\cat{G}}(X,Y)  \xrightarrow{} \Ext^2_{\cat{G}}(X,I)
\] 
shows that $\Ext^2_{\cat{G}}(X,Y) =0$. Repeating this argument with induction we conclude also that $\Ext^i_{\cat{G}}(X,Y) = 0$ for all indices $i > 1$.

The implication (e) $\Longrightarrow$ (h) is clear for any $n \geq 0$, while (h) $\Longrightarrow$ (i) (also for $n \geq 0$) holds since $\mathcal{FP}_{n+1}\mbox{-}\textrm{Inj} \subseteq \mathcal{FP}_\infty\mbox{-}\textrm{Inj}$ and $\mathcal{FP}_\infty\mbox{-}\textrm{Inj} = \mathcal{FP}_n\mbox{-}\textrm{Inj}$ by (h).

For the rest of the proof, let us assume that $n \geq 1$. We show that (i) $\Longrightarrow$ (b). Condition (i) clearly implies that $\mathcal{FP}_n\mbox{-}\textrm{Inj} = \mathcal{FP}_{n+1}\mbox{-}\textrm{Inj}$. On the other hand, we have by Proposition \ref{FPn-in-terms-of-FPn-Inj} that $\mathcal{FP}_{n+1} = \mathcal{FP}_n \cap {}^{\perp_1}(\mathcal{FP}_{n+1}\mbox{-}\textrm{Inj})$. Using the equiality $\mathcal{FP}_n\mbox{-}\textrm{Inj} = \mathcal{FP}_{n+1}\mbox{-}\textrm{Inj}$, the previous implies $\mathcal{FP}_{n+1} = \mathcal{FP}_n$. This in turn clearly implies (b). 
\end{proof}

\begin{example} \
\begin{enumerate}
\item Recall that a ring $R$ is left \emph{$n$-coherent} if the containment $\mathcal{FP}_n \subseteq \mathcal{FP}_{n+1}$ holds in $R\mbox{-}\textrm{Mod}$ (see Costa's \cite[Definition 2.1]{costa}). We can note that $R\mbox{-}\textrm{Mod}$ is an $n$-coherent category if, and only if, $R$ is a left $n$-coherent ring. This equivalence can be extended to the category $\textrm{Ch}(R)$ of complexes of modules, as proved in \cite[Proposition 2.1.9]{ZhaoPerez}. 

\item The functor category $\textrm{Fun}(\mathcal{C}^{\rm op},\mathsf{Ab})$ is $1$-coherent if, and only if, $\mathcal{C}$ has pseudo-kernels. (See Appendix C). 

\item The category $\mathfrak{Qcoh}(X)$ of quasi-coherent sheaves over $X$ can be made into an $n$-coherent category if $X$ comes equipped with a finite affine cover $\{ U_i \}$ (that is, $U_i$ is isomorphic, as a locally ringed space, to ${\rm Spec}(A_i)$) such that each $A_i$ is a commutative $n$-coherent ring, with $n \geq 0$ fixed. (See Appendix B for details). 
\end{enumerate}
\end{example}

Note that condition (e) of the theorem gives us the next two corollaries.

\begin{corollary}
Any $n$-coherent category $\cat{G}$ is locally type ${\rm FP}_{\infty}$.
\end{corollary}

\begin{corollary}
We have the following containments among classes of Grothendieck categories, where $n\text{-Coh}$ represents the class of $n$-coherent categories:
\[
0\text{-Coh} \subseteq 1\text{-Coh} \subseteq \cdots \subseteq n\text{-Coh} \subseteq (n+1)\text{-Coh} \subseteq \cdots \subseteq \infty\text{-Coh}.
\]
\end{corollary}

Finally, condition (d) of the theorem can be used in conjunction with condition (c) of Lemma \ref{lemma-extensions} to check the coherence of a particular category. In particular, we get the following corollary.

\begin{corollary}\label{cor-n-coherent rings}
Let $R$ be a ring. Then $R$ is (left) $n$-coherent if, and only if, every finitely generated free (left) $R$-module is $n$-coherent. 
\end{corollary}


\subsection*{{$\bm{{\rm FP}}_{\bm{n}}$-injective covers in $\bm{n}$-coherent categories}}

Corollary~\ref{coro-preenveloping} showed that the class $\class{FP}_n\textnormal{-Inj}$, of ${\rm FP}_n$-injective objects, is (special) preenveloping. We now consider the question of when it might also be a (pre)covering class. In \cite{CPT}, several conditions in finitely accessible categories are studied in order to produce preenvelopes and covers relative to a class of objects. Since any locally $n$-coherent category is finitely accessible, we can apply Crivei, Prest and Torrecillas' result to obtain ${\rm FP}_n$-injective covers in locally $n$-coherent categories.

\begin{proposition}[${\rm FP}_n$-injectives and purity]\label{prop:FPnInj_purity}
Let $\mathcal{G}$ be a Grothendieck category and $n \geq 1$. The following two conditions hold:
\begin{enumerate}
\item $\class{FP}_n\textnormal{-Inj}$ is closed under pure subobjects.

\item If $\mathcal{G}$ is  $n$-coherent, then $\class{FP}_n\textnormal{-Inj}$ is closed under pure quotients.
\end{enumerate}
\end{proposition}

\begin{proof} 
Suppose that we are given a pure exact sequence 
\[
\mathbb{P} \colon 0 \to A \to B \to C \to 0
\]
in $\mathcal{G}$, that is, the induced sequence ${\rm Hom}_{\mathcal{G}}(F,\mathbb{P})$ of abelian groups is exact whenever $F$ is a finitely presented object in $\mathcal{G}$. Assume that $B \in \class{FP}_n\textnormal{-Inj}$ and let $F$ be an object of type $\text{FP}_n$. Let us first see (1), that $A \in \class{FP}_n\textnormal{-Inj}$. Since $B \in \class{FP}_n\textnormal{-Inj}$, we have that ${\rm Ext}^1_{\mathcal{G}}(F,B) = 0$. On the other hand, $F$ is in particular finitely presented, and so ${\rm Hom}_{\mathcal{G}}(F,\mathbb{P})$ is exact. Thus, we have an exact sequence
\[
0 \to {\rm Hom}_{\mathcal{G}}(F,A) \to {\rm Hom}_{\mathcal{G}}(F,B) \to {\rm Hom}_{\mathcal{G}}(F,C) \to {\rm Ext}^1_{\mathcal{G}}(F,A) \to 0
\]
where ${\rm Hom}_{\mathcal{G}}(F,B) \to {\rm Hom}_{\mathcal{G}}(F,C)$ is an epimorphism. It follows that ${\rm Ext}^1_{\mathcal{G}}(F,A) = 0$, and hence $A$ is ${\rm FP}_n$-injective. 
For (2), if we suppose that $\mathcal{G}$ is in addition  $n$-coherent, then by Theorem \ref{them-locally n-coherent} (f) we may also conclude $C$ is ${\rm FP}_n$-injective. 
\end{proof}

The precise conditions guaranteeing existence of covers are specified in \cite[Theorem 2.6]{CPT}. Namely, a class of objects $\mathcal{C}$ in a finitely accessible category $\mathcal{G}$ is covering provided it is closed under direct limits and pure quotients. Certainly if $\cat{G}$ is  $n$-coherent then the class $\class{FP}_n\textnormal{-Inj}$ of $\text{FP}_n$-injective objects is closed under pure quotients and direct limits; in fact, it is always closed under direct limits for $n > 1$. So their work gives the following generalization of a statement from~\cite[Corollary 3.5]{CPT}.

\begin{corollary}[completeness of the reversed $\class{FP}_n$-injective cotorsion pair]
Let $\mathcal{G}$ be a  $n$-coherent category. Then, the class $\class{FP}_n\textnormal{-Inj}$ of ${\rm FP}_n$-injective objects is covering. Moreover, if $\class{FP}_n\textnormal{-Inj}$ contains a generating set for  $\class{G}$, then $\class{FP}_n\textnormal{-Inj}$ is the left half of a perfect cotorsion pair $(\class{FP}_n\textnormal{-Inj} , (\class{FP}_n\textnormal{-Inj})^{\perp_1})$.  
\end{corollary}

Recall that a cotorsion pair $(\mathcal{X,Y})$ in a Grothendieck category $\mathcal{G}$ is \emph{perfect} if the class $\mathcal{X}$ is covering and the class $\mathcal{Y}$ is enveloping. A well known result asserts that if $(\mathcal{X,Y})$ is complete and $\mathcal{X}$ is closed under direct limits, then $(\mathcal{X,Y})$ is perfect. For example, see \cite[Corollary 2.3.7]{gobel-trlifaj} or \cite[Section 2.2]{xu}.\footnote{Although the proofs given there are for $R$-modules, they carry over to Grothendieck categories.}

\begin{proof}
It is only left to prove the second statement. In this case, any ${\rm FP}_n$-injective cover must be an epimorphism, since $\class{FP}_n\textnormal{-Inj}$ contains a generating set for $\class{G}$. By Wakamutsu's Lemma\footnote{See \cite[Lemma 2.1.1]{xu} for a proof that works in any abelian category.}, any such cover must have a kernel in $(\class{FP}_n\textnormal{-Inj})^{\perp_1}$. With this fact, we show that ${}^{\perp_1}((\class{FP}_n\textnormal{-Inj})^{\perp_1})$ $\subseteq \class{FP}_n\textnormal{-Inj}$. For let $C \in {}^{\perp_1}((\class{FP}_n\textnormal{-Inj})^{\perp_1})$. We have a special $\class{FP}_n\textnormal{-Inj}$ cover for $C$, that is, a short exact sequence 
\[
0 \to A \to B \to C \to 0
\] 
with $B \in \class{FP}_n\textnormal{-Inj}$ and $A \in (\class{FP}_n\textnormal{-Inj})^{\perp_1}$. Since $C \in {}^{\perp_1}((\class{FP}_n\textnormal{-Inj})^{\perp_1})$, the sequence splits, and so $C$ is a direct summand of $B \in \class{FP}_n\textnormal{-Inj}$. The class $\class{FP}_n\textnormal{-Inj}$ is closed under direct summands, and hence we can conclude that $C \in \class{FP}_n\textnormal{-Inj}$. This proves $(\class{FP}_n\textnormal{-Inj},(\class{FP}_n\textnormal{-Inj})^{\perp_1})$ is a cotorsion pair in $\mathcal{G}$.

We already know every object in $\mathcal{G}$ has a special $\class{FP}_n\textnormal{-Inj}$ cover. Using this and the fact that $\mathcal{G}$ has enough injective objects, we can apply a Salce-like argument to show that every object in $\mathcal{G}$ also has a special $(\class{FP}_n\textnormal{-Inj})^{\perp_1}$-preenvelope. Hence, the cotorsion pair $(\class{FP}_n\textnormal{-Inj},(\class{FP}_n\textnormal{-Inj})^{\perp_1})$ is complete. Since $\class{FP}_n\textnormal{-Inj}$ is closed under direct limits, we have that $(\class{FP}_n\textnormal{-Inj},(\class{FP}_n\textnormal{-Inj})^{\perp_1})$ is perfect. 
\end{proof}

Note that we needed the  $n$-coherent hypothesis on $\cat{G}$ to show that the class $\class{FP}_n\textnormal{-Inj}$ of ${\rm FP}_n$-injective objects is covering. But for the category $R$-Mod, the statement holds for $n > 1$ even if $R$ is not assumed $n$-coherent. Indeed the ${\rm FP}_n$-injective modules are always closed under pure quotients (in addition to direct limits) in this case. As shown in~\cite[Proposition 3.10]{bravo-perez}, this follows by a Pontrjagin duality argument. The problem with the more general setting of Grothendieck categories $\mathcal{G}$ is that we do not have a suitable notion of Pontrjagin dual providing similar properties. This essentially stems from that fact that we lack of a tensor product on $\mathcal{G}$ to compare ${\rm FP}_n$-injective and ``${\rm FP}_n$-flat'' objects. So we are not aware if the  $n$-coherent hypothesis on $\cat{G}$ is absolutely necessary to show that $\class{FP}_n\textnormal{-Inj}$ is a covering class.


\section{The Gorenstein $\textrm{FP}_n$-injective model structures}\label{sec-Goren-FP_n}

Our goal now is to point out how a nice theory of Gorenstein $\text{FP}_n$-injective homological algebra exists in any $n$-coherent category $\mathcal{G}$. We define Gorenstein $\text{FP}_n$-injective objects similarly to the usual Gorenstein injective objects.

\begin{definition} 
We say an object $M \in \cat{G}$ is \textbf{Gorenstein $\textrm{FP}_{\bm{n}}$-injective} if $M = Z_{0}(\mathbb{I})$ for some exact complex $\mathbb{I}$ of injectives for which $\Hom_{\cat{G}}(J,\mathbb{I})$ remains exact for any $\textrm{FP}_n$-injective $J$.  We let $\class{GI}$ denote the class of all Gorenstein $\textrm{FP}_n$-injectives in $\cat{G}$ and set $\class{W} := {}^{\perp_1}\class{GI}$.
\end{definition}

Note that if $\cat{G}$ is  $n$-coherent, then by Theorem~\ref{them-locally n-coherent}, the Gorenstein $\textrm{FP}_n$-injectives coincide with the \emph{Gorenstein AC-injective objects} from \cite{gillespie-models-of-injectives}, inspired from~\cite{bravo-gillespie-hovey}. In particular, when $\cat{G}$ is locally noetherian they coincide with the usual notion of \emph{Gorenstein injective}, and when $\cat{G}$ is locally coherent they are the \emph{Ding injective} objects.


\subsection*{{Properties of Gorenstein $\text{FP}_{\bm{n}}$-injective objects}}

We begin our path towards a theory of Gorenstein $\textrm{FP}_n$-injective homological algebra by proving some characterizations and properties of the class $\mathcal{GI}$. For the rest of the present paper, let $\textrm{Inj}$ denote the class of injective objects in a Grothendieck category $\mathcal{G}$.

\begin{lemma}[characterizations of Gorenstein ${\rm FP}_n$-injectives]\label{lem:characterizationGI}
The following are equivalent for any object $C$ in a Grothendieck category $\cat{G}$.
\begin{itemize}
\item[(a)] $C$ is Gorenstein $\textrm{FP}_n$-injective.

\item[(b)] $C$ satisfies the following two conditions.
\begin{enumerate}
\item $C \in (\class{FP}_n\textnormal{-Inj})^{\perp}$. That is, $\Ext^i_{\mathcal{G}}(X,C) = 0$ for every $X \in \class{FP}_n\textnormal{-Inj}$ and every $i > 0$.

\item There exists an exact sequence 
\[
\mathbb{E} \colon \cdots \to E_1 \to E_0 \to C \to 0
\]
 with each $E_i \in  \textnormal{Inj}$ (that is, $\mathbb{E}$ is an injective resolution of $C$) such that $\Hom_{\mathcal{G}}(X,\mathbb{E})$ remains exact for every $X \in \class{FP}_n\textnormal{-Inj}$.
\end{enumerate}

\item[(c)] There exists a short exact sequence 
\[
0 \to C' \to E \to C \to 0
\] 
with $E \in  \textnormal{Inj}$ and $C' \in \mathcal{GI}$. 
\end{itemize}
\end{lemma}

\begin{proof}
These are known characterizations of the usual Gorenstein injective $R$-modules and the proofs carry over to our generality. In particular, the equivalence of (a) and (b) is a straightforward exercise that we leave to the reader\footnote{The proof will use the fact that we always have enough injectives in a Grothendieck category.}. Now (a) implies (c) is clear, and the converse can be proved by imitating (the dual of) the argument from \cite[Lemma~2.5]{Ding-projective}.
\end{proof}

\begin{proposition}[closure properties of Gorenstein ${\rm FP}_n$-injectives]\label{prop:closure_properties}
Let $\mathcal{G}$ be a Grothendieck category. Then the class $\mathcal{GI}$ of Gorenstein $\text{FP}_n$-injective objects of $\mathcal{G}$ is closed under finite direct sums,  direct summands,  extensions and  cokernels of monomorphisms between its objects.
\end{proposition}

\begin{proof}
We leave it to the reader to verify that closure under extensions and cokernels of monomorphisms can be proved using the same techniques as in~\cite{Ding-projective}. We shall include a direct proof that $\mathcal{GI}$ is closed under direct summands\footnote{Although the argument uses standard techniques, we do not believe it has appeared before in the literature.}.

So suppose we are given $A, B \in \mathcal{G}$ such that $A \oplus B \in \mathcal{GI}$. Note first that $A, B \in (\class{FP}_n\textnormal{-Inj})^\perp$. On the other hand, we have a short exact sequence
\[
\mathbb{E}^{A \oplus B}_0 \colon 0 \to K_0 \xrightarrow{\alpha} E_0 \xrightarrow{\beta} A \oplus B \to 0,
\]
where $E_0 \in  \textnormal{Inj}$ and $K_0 \in \mathcal{GI}$ by Lemma \ref{lem:characterizationGI}. Consider now the canonical projections 
\[
\pi_A = (\begin{array}{cc} {\rm id}_A & 0 \end{array}) \colon A \oplus B \to A \mbox{ \ and \ } \pi_B = (\begin{array}{cc} 0 & {\rm id}_B \end{array}) \colon A \oplus B \to B,
\] 
and form the morphisms $\beta_A := \pi_A \circ \beta$ and $\beta_B := \pi_B \circ \beta$, that is
\[
\beta = \left( \begin{array}{c} \beta_A \\ \beta_B \end{array} \right)
\]
using the matrix notation. We have short exact sequences
\[
\mathbb{E}^A_0 \colon 0 \to K^A_0 \xrightarrow{\alpha_A} E_0 \xrightarrow{\beta_A} A \to 0 \mbox{ \ and \ } \mathbb{E}^B_0 \colon 0 \to K^B_0 \xrightarrow{\alpha_B} E_0 \xrightarrow{\beta_B} B \to 0,
\]
where $K^A_0 := {\rm Ker}(\beta_A)$ and $K^B_0 := {\rm Ker}(\beta_B)$. Taking the direct sum of $\mathbb{E}^A_0$ and $\mathbb{E}^B_0$ gives the following short exact sequence:
\[{\scriptsize
\mathbb{E}^A_0 \oplus \mathbb{E}^B_0 \colon 0 \to K^A_0 \oplus K^B_0 \xrightarrow{\left( \begin{array}{cc} \alpha_A & 0 \\ 0 & \alpha_B \end{array} \right)} E_0 \oplus E_0 \xrightarrow{\left( \begin{array}{cc} \beta_A & 0 \\ 0 & \beta_B \end{array} \right)} A \oplus B \to 0.
}
\]
Moreover, we can get the following commutative diagram with exact rows and columns:
\[
\begin{tikzpicture}[description/.style={fill=white,inner sep=2pt}]
\matrix (m) [matrix of math nodes, row sep=3em, column sep=5em, text height=1.25ex, text depth=0.25ex]
{ 
0 & K_0 & E_0 & A \oplus B & 0 \\
0 & K^A_0 \oplus K^B_0 & E_0 \oplus E_0 & A \oplus B & 0 \\
};
\path[->]
(m-1-2) edge node[right] {\scriptsize$k$} (m-2-2)
(m-1-3) edge node[right] {\scriptsize$\Delta = \left( \begin{array}{c} {\rm id}_{E_0} \\ {\rm id}_{E_0} \end{array} \right)$} (m-2-3)
(m-1-1) edge (m-1-2) (m-1-2) edge node[above] {\scriptsize$\alpha$} (m-1-3) (m-1-3) edge node[above] {\scriptsize$\beta$} (m-1-4) (m-1-4) edge (m-1-5)
(m-2-1) edge (m-2-2) (m-2-2) edge node[below] {\scriptsize$\left( \begin{array}{cc} \alpha_A & 0 \\ 0 & \alpha_B \end{array} \right)$} (m-2-3) (m-2-3) edge node[below] {\scriptsize$\left( \begin{array}{cc} \beta_A & 0 \\ 0 & \beta_B \end{array} \right)$} (m-2-4) (m-2-4) edge (m-2-5)
;
\path[-,font=\scriptsize]
(m-1-4) edge [double, thick, double distance=2pt] (m-2-4)
;
\end{tikzpicture}
\]
where $\Delta$ is the diagonal map and $k$ is induced by the universal property of kernels. Using Snake's Lemma, we have ${\rm CoKer}(k) \simeq {\rm CoKer}(\Delta) \simeq E_0$. Thus, we have a split short exact sequence 
\[
0 \to K_0 \xrightarrow{k} K^A_0 \oplus K^B_0 \to E_0 \to 0
\]
and so $K^A_0 \oplus K^B_0 \simeq K_0 \oplus E_0 \in \mathcal{GI}$ since $\mathcal{GI}$ is closed under finite direct sums. Note also that $K^A_0, K^B_0 \in (\class{FP}_n\textnormal{-Inj})^\perp$.

Apply the previous argument to $K^A_0 \oplus K^B_0 \in \mathcal{GI}$, and repeat this procedure infinitely many times in oder to get injective resolutions of $A$ and $B$ which are $\Hom_{\mathcal{G}}(X,-)$-acyclic for every $X \in \class{FP}_n\textnormal{-Inj}$. The result follows by Lemma \ref{lem:characterizationGI}. 
\end{proof}

With these properties in hand, we are ready to construct approximations, cotorsion pairs and model category structures involving the class $\mathcal{GI}$.


\subsection*{{Abelian Gorenstein $\text{FP}_{\bm{n}}$-injective model structures}}

Any theory of relative homological algebra begins with the existence of a complete cotorsion pair, allowing one to construct relative derived functors. So we would like to have a complete cotorsion pair $(\class{W} ,\class{GI})$. Unfortunately it is not enough for $\cat{G}$ to just be a  $n$-coherent category to obtain this complete cotorsion pair. We need the left class $\class{W}$ to contain a generating set for $\cat{G}$. A standard hypothesis that will accomplish this is to assume that $\cat{G}$ has a generating set of objects of finite projective dimension. Recall that an object $A \in \cat{G}$ has \emph{finite projective dimension} if there exists a nonnegative integer $n$ such that for any object $B$ one has $\Ext^i_{\cat{G}}(A,B) = 0$ for all $i > n$. Following~\cite{gillespie-models-of-injectives}, we make the following definition.

\begin{definition}\label{def-generators}
We shall say a Grothendieck category $\cat{G}$ is \emph{locally finite dimensional} if it possesses a generating set $\{G_i\}$ for which each $G_i$ has finite projective dimension. If furthermore, each $G_i$ is of type ${\rm FP}_{\infty}$ we say $\cat{G}$ it is \emph{locally  finite dimensionally type ${\rm FP}_{\infty}$}. Finally, if $\cat{G}$ is also   $n$-coherent, that is, $\class{FP}_n = \class{FP}_{\infty}$, we say it is \emph{locally finite dimensionally $n$-coherent}.
\end{definition}

Examples of  locally  finite dimensionally type $\textrm{FP}_{\infty}$ categories are given in \cite[Section~5]{gillespie-models-of-injectives}. Certainly $R$-Mod and $\ch$ are locally finite dimensionally $n$-coherent whenever $R$ is a (left) $n$-coherent ring. The functor category $\textrm{Fun}(\mathcal{C}^{\rm op},\mathsf{Ab})$ is locally finite dimensionally $1$-coherent for any additive category $\mathcal{C}$ with pseudo-kernels.

\begin{corollary}[the abelian Gorenstein ${\rm FP}_n$-injective model structure]\label{cor-Gor-module}
Let $\cat{G}$ be a locally finite dimensionally $n$-coherent category. Then there is a cofibrantly generated abelian model structure on $\cat{G}$, the \emph{Gorenstein ${\rm FP}_n$-injective model structure}, in which every object is cofibrant and the fibrant objects are the Gorenstein ${\rm FP}_n$-injectives. In particular, $(\cat{W}, \cat{GI})$ is a complete cotorsion pair, cognerated by a set containing the  $\{G_i\}$, as in Definition \ref{def-generators}, and every object has a special Gorenstein ${\rm FP}_n$-injective preenvelope. 
\end{corollary}

\begin{proof}
Since $\cat{G}$ is a locally  finite dimensionally type $\textrm{FP}_{\infty}$ category we get the model structure from \cite[Section 5.1]{gillespie-models-of-injectives}. (See in particular, Theorem 7.5 and Corollary 7.7.) The point is that the Gorenstein AC-injectives coincide with the Gorenstein ${\rm FP}_n$-injectives, since $\class{FP}_n = \class{FP}_{\infty}$.
\end{proof}

\begin{remark}\label{rem:derived}
The other results from \cite{gillespie-models-of-injectives} have special interpretations for locally $n$-coherent categories. For example, $\mathsf{D}(\mathcal{FP}_n\text{-Inj})$, the derived category of $\text{FP}_n$-injectives, is a compactly generated triangulated category and equivalent to a full subcategory of $\mathsf{K}( \textnormal{Inj})$. (See~\cite[Theorems 4.4 and 4.8]{gillespie-models-of-injectives}.) In the case $\cat{G}$ possesses a nice generating set as above in Definition \ref{def-generators}, this category sits in the middle of a recollement involving three compactly generated categories. (See \cite[Corollary 5.12]{gillespie-models-of-injectives}.)
\end{remark}


\subsection*{Exact Gorenstein $\text{FP}_{\bm{n}}$-injective model structures}

Corollary \ref{cor-Gor-module} is not the only way to obtain model structures from Gorenstein $\text{FP}_n$-injective objects. Our aim in this section is to construct exact model structures on certain subcategories of $\mathcal{G}$ for which $\mathcal{GI}$ is the class of fibrant objects, and without imposing any condition on $\mathcal{G}$. We shall achieve this by applying the theory of Frobenius pairs, presented in \cite{frobenius_pairs}. 

Two classes $\mathcal{Y}$ and $\nu$ of objects in an abelian category $\mathcal{G}$ form a (\emph{right})  \emph{Frobenius pair} $(\nu,\mathcal{Y})$ in $\mathcal{G}$ if the following conditions hold:
\begin{enumerate}
\item $\mathcal{Y}$ is closed under extensions, cokernels of monomorphisms between its objects, and under direct summands in $\mathcal{G}$.

\item $\nu \subseteq \mathcal{Y}$ and $\nu$ is closed under direct summands in $\mathcal{G}$.

\item $\nu$ is a relative generator in $\mathcal{Y}$, that is, for every $Y \in \mathcal{Y}$ there exists a short exact sequence 
\[
0 \to Y' \to V \to Y \to 0
\] 
with $Y' \in \mathcal{Y}$ and $V \in \nu$. 

\item $\nu$ is a $\mathcal{Y}$-projective, meaning that $\Ext^i_{\mathcal{G}}(V,Y) = 0$ for every $V \in \nu$, $Y \in \mathcal{Y}$ and $i \geq 1$.
\end{enumerate}
If in addition $(\nu,\mathcal{Y})$ satisfies the dual of conditions (3) and (4), that is, $\nu$ is a $\mathcal{Y}$-injective relative cogenerator in $\mathcal{Y}$, then the Frobenius pair $(\nu,\mathcal{Y})$ is called \emph{strong}. 

Frobenius pairs comprise several properties that allow us to construct left and right approximations by the classes $\mathcal{Y}$ and $\nu$, which in turn we can use to construct cotorsion pairs and exact model structures. These model structures are referred in \cite{frobenius_pairs} as \emph{injective Auslander-Buchweitz model structures}.

\begin{proposition}[the Gorenstein ${\rm FP}_n$-injective Frobenius pair]
Let $\mathcal{G}$ be a Grothendieck category. Then, $( \textnormal{Inj},\mathcal{GI})$ is a strong Frobenius pair in $\mathcal{G}$. 
\end{proposition}

\begin{proof}
First, note from the definition of Gorenstein $\text{FP}_n$-injective objects that $ \textnormal{Inj}$ is a relative generator and cogenerator in $\mathcal{GI}$. Another consequence from the definition of $\mathcal{GI}$ is that ${\rm Ext}^i_{\mathcal{G}}(J,M)$ $= 0$ for every $J \in \textrm{Inj}$, $M \in \mathcal{GI}$ and $i \geq 1$. On the other hand, it is clear that ${\rm Ext}^i_{\mathcal{G}}(M,J) = 0$ for any $J$, $M$ and $i$ as before. The rest of the proof follows by Proposition \ref{prop:closure_properties}.
\end{proof}

Having a strong Frobenius pair $(\nu,\mathcal{Y})$ in $\mathcal{G}$ implies the existence of certain compatible complete cotorsion pairs. These are not cotorsion pairs in an abelian category, but in an exact category. Namely, the full subcategory $\mathcal{Y}^\vee$ formed by the objects in $\mathcal{G}$ which have a finite coresolution by objects in $\mathcal{Y}$, that is, objects $C \in \mathcal{G}$ such that there exist $m \geq 0$ and an exact sequence 
\[
0 \to C \to Y^0 \to Y^1 \to \cdots \to Y^{m-1} \to Y^m \to 0
\] 
where $Y^k \in \mathcal{Y}$ for every $0 \leq k \leq m$. According to \cite[Dual of Theorem 3.6]{frobenius_pairs}, if $(\nu,\mathcal{Y})$ is a Frobenius pair in $\mathcal{G}$, then $(\nu^\vee,\mathcal{Y})$ is a complete cotorsion pair in the exact category $\mathcal{Y}^\vee$. If in addition $(\nu,\mathcal{Y})$ is strong, then $(\mathcal{Y}^\vee,\nu)$ is also a complete cotorsion pair in $\mathcal{Y}^\vee$ by \cite[Dual of Theorem 3.7]{frobenius_pairs}. Thus, the following result holds.

\begin{proposition}[the exact Gorenstein ${\rm FP}_n$-injective cotorsion pair]\label{prop:ex_cot_pairs}
Let $\mathcal{G}$ be a Grothendieck category. Then, $( \textnormal{Inj}^\vee,\mathcal{GI})$ and $(\mathcal{GI}^\vee, \textnormal{Inj})$ are hereditary and complete cotorsion pairs in $\mathcal{GI}^\vee$. 
\end{proposition}

The compatibility between the pairs $( \textnormal{Inj}^\vee,\mathcal{GI})$ and $(\mathcal{GI}^\vee, \textnormal{Inj})$ will be a consequence of Proposition \ref{compapprox} below, which represents a summary of Auslander-Buchweitz approximation theory in the context of Gorenstein ${\rm FP}_n$-injective objects. The reader can see a revisit to AB theory (at least for the part needed for Frobenius pairs) in \cite[Section 2.2]{frobenius_pairs}. 

Define the \emph{Gorenstein ${\rm FP}_n$-injective dimension} of an object $C \in \mathcal{G}$, denoted ${\rm Gid}(C)$ as the smallest nonnegative integer $m \geq 0$ such that there is an exact sequence 
\[
0 \to C \to E^0 \to E^1 \to \cdots \to E^{m-1} \to E^m \to 0
\] 
with $E^k \in \mathcal{GI}$ for every $0 \leq k \leq m$. If such $m$ does not exist, we simply set ${\rm Gid}(C) = \infty$.

\begin{proposition}[compatibility conditions and approximations]\label{compapprox}
Let $\mathcal{G}$ be a Grothendieck category. Then, the following conditions hold true:
\begin{enumerate}
\item $\mathcal{GI}^\vee$ is the smallest thick subcategory of $\mathcal{G}$ containing $\mathcal{GI}$. 

\item $ \textnormal{Inj} = \{ X \in \mathcal{GI} \mbox{ {\rm :} } {\rm pd}_{\mathcal{GI}}(X) = 0 \} = \mathcal{GI} \cap  \textnormal{Inj}^\vee = \mathcal{GI} \cap {}^{\perp}\mathcal{GI}$.

\item $\mathcal{GI} \cap  \textnormal{Inj}^\wedge = \{ X \in \mathcal{GI} \mbox{ {\rm :} } {\rm pd}_{\mathcal{GI}}(X) < \infty \}$.

\item $ \textnormal{Inj}^\vee = {}^\perp\mathcal{GI} \cap \mathcal{GI}^\vee$.

\item $\mathcal{GI}^\vee \cap  \textnormal{Inj}^\perp = \mathcal{GI} = \mathcal{GI}^\vee \cap ( \textnormal{Inj}^\vee)^\perp$. 

\item For every $C \in \mathcal{G}$ with ${\rm Gid}(C) = m < \infty$, there exist short exact sequences
\[
0 \to C \to X \to W \to 0 \mbox{ \ and \ } 0 \to Y \to H \to C \to 0
\]
with $X, Y \in \mathcal{GI}$, ${\rm id}(W) = m-1$ and ${\rm id}(H) \leq m$.
\end{enumerate}
\end{proposition}

From this result, the class $\mathcal{GI}^\vee$ of objects of $\mathcal{G}$ with finite Gorenstein $\text{FP}_n$-injective dimension is a thick subcategory of $\mathcal{G}$, and so it is exact with the usual exact structure of subcategories of an abelian category that are closed under extensions. One can also note easily that the exact category $\mathcal{GI}^\vee$ is \emph{weakly idempotent complete} (see \cite[Definition 2.2]{gillespie-exact-model-structures}). Thus, using the generalization of Hovey's correspondence in the context of exact categories, proved by the second author in \cite{gillespie-exact-model-structures}, we have the following model category structure on $\mathcal{GI}^\vee$.

\begin{theorem}[the exact Gorenstein ${\rm FP}_n$-injective model structure]
Let $\mathcal{G}$ be a Grothendieck category. Then, there exists a unique injective and hereditary exact model structure on $\mathcal{GI}^\vee$ such that $\mathcal{GI}$ is the class of fibrant objects and $ \textnormal{Inj}^\vee$ is the class of trivial objects. We denote this model structure by 
\[
\mathcal{M}^{\rm fp}_n(\mathcal{GI}^\vee) := (\mathcal{GI}^\vee, \textnormal{Inj}^\vee,\mathcal{GI}).
\]
\end{theorem}

From \cite{gillespie-exact-model-structures}, we know also how the homotopy relations are defined for $\mathcal{M}^{\rm fp}_n(\mathcal{GI}^\vee)$. Specifically, we have the following description for the homotopy category of the model structure $\mathcal{M}^{\rm fp}_n(\mathcal{GI}^\vee)$, denoted ${\rm Ho}(\mathcal{GI}^\vee)$.

\begin{theorem}
Let $\mathcal{G}$ be a Grothendieck category. Then, there exists a natural isomorphism 
\[
{\rm Hom}_{{\rm Ho}(\mathcal{GI}^\vee)}(X,Y) \cong {\rm Hom}_{\mathcal{GI}^\vee}(X,RY) / \sim
\]
for every $X, Y \in \mathcal{GI}^\vee$, where:
\begin{itemize}
\item $RY$ is the fibrant replacement of $Y$.

\item For every pair of morphisms $f, g \colon X \to RY$, $f \sim g$ if, and only if, $g - f$ factors through an injective object of $\mathcal{G}$. 
\end{itemize}
Moreover, ${\rm Ho}(\mathcal{GI}^\vee)$ is triangle equivalent to the stable category $\mathcal{GI} / \sim$. 
\end{theorem}

\begin{remark}
As pointed out in \cite[Remark 4.11]{frobenius_pairs}, the meaning of ``triangulated category'' in the previous statement is the classical one (that is, is the sense of Verdier's \cite{Verdier96}), different from the approach to triangulated categories given in \cite[31, Chapter 7]{hovey-model-categories}, and mentioned in Remark \ref{rem:derived}. 
\end{remark}


\appendix

\section{Proof of the monomorphism property}


We devote this first appendix exclusively to prove Lemma \ref{lem:mono_condition}. This is a well known result for the category of modules over a ring $R$. For instance, one proof is due to R. Strebel \cite[Lemma 2.4]{strebel}. Strebel's arguments can be generalized to the category of complexes of modules, and actually to any Grothendieck category with a generating set of finitely generated projective objects. We know by Corollary \ref{resolutions_for_proj} that in any such category, finitely $n$-presented objects are objects with a truncated finitely generated projective resolution of length $n$ (an $n$-presentation), and then a dimension shifting argument can show the validity of Lemma \ref{lem:mono_condition} in this case. The general case, however, is more difficult to show, and requires the reader to have some knowledge on Yoneda $n$-fold extensions and their properties: especially, to keep in mind several descriptions of $n$-fold extensions when they are equivalent to zero. We recommend the reader to check \cite[Chapter VII]{mitchell} and \cite[Chapter 4]{Sieg} to recall the basic terminology and results that will be used in the sequel. We shall continue using the font $\mathbb{E}$ for short exact sequences and longer Yoneda $n$-fold extensions, and shall write $[\mathbb{E}]$ to denote the class of $\mathbb{E}$ under the usual equivalence relation defining the extension groups.

From now on, we fix $F \in \mathcal{G}$ an  object of type \tFP{n} and $\{ X_i \mbox{ : } i \in I \}$ a direct system of objects in $\mathcal{G}$ over a directed set $I$, whose direct limit we denote $X = \varinjlim_{I} X_i$. The canonical morphisms in this direct system will be denoted by $\phi_i \colon X_i \rightarrow X$, and the compatible morphisms by $f_{ij} \colon X_i \to X_j$ with $i \leq j$, that is: (1) $f_{ii} = {\rm id}_{X_i}$, (2) $f_{ik} = f_{jk} \circ f_{ij}$ for $i \leq j \leq k$, and (3) $\phi_j \circ f_{ij} = \phi_i$ for $i \leq j$.


\subsection*{Case $\bm{n = 0}$} 

If $F$ is finitely generated, then $\xi_0 \colon \varinjlim_I {\rm Hom}_{\mathcal{G}}(F,X_i) \to {\rm Hom}_{\mathcal{G}}(F,X)$ is a monomorphism. This case is easy and follows as in the first part of \cite[Proof of Proposition V.3.4]{stenstrom}.


\subsection*{Case $\bm{n = 1}$}

If $F$ is finitely presented, then $\xi_1 \colon \varinjlim_I {\rm Ext}^1_{\mathcal{G}}(F,X_i) \to {\rm Ext}^1_{\mathcal{G}}(F,X)$ is a monomorphism. This case is due to Stenstrom \cite[Proposition 2.1]{stenstrom2}. Although Stenstr\"om presents the result in the category of right $R$-modules, his proof works in any Grothendieck category (even without enough projectives). We have decided to include a more detailed version of this proof for further reference in the case $n > 1$. 

Consider a class $[\mathbb{E}_i] \in {\rm Ext}^1_{\mathcal{G}}(F,X_i)$ represented as 
\[
\mathbb{E}_i \colon 0 \to X_i \to Y_i \to F \to 0.
\] 
The morphisms $f_{ij}$ induce group homomorphisms $\overline{f}_{ij} \colon {\rm Ext}^1_{\mathcal{G}}(F,X_i) \to {\rm Ext}^1_{\mathcal{G}}(F,X_j)$ by the pushout of $X_j \xleftarrow[f_{ij}]{} X_i \to Y_i$ as follows:
\[
\begin{tikzpicture}[description/.style={fill=white,inner sep=2pt}]
\matrix (m) [matrix of math nodes, row sep=2.5em, column sep=2.5em, text height=1.25ex, text depth=0.25ex]
{ 
\mathbb{E}_i \colon 0 & X_i & Y_i & F & 0 \\
\mathbb{E}_j \colon 0 & X_j & Y_j & F & 0 \\
};
\path[->]
(m-1-2)-- node[pos=0.5] {\footnotesize$\mbox{\bf po}$} (m-2-3)
(m-1-2) [yshift=30pt] edge node[left] {\footnotesize$f_{ij}$} (m-2-2) 
(m-1-3) edge (m-2-3)
(m-1-1) edge (m-1-2) (m-1-4) edge (m-1-5)
(m-2-1) edge (m-2-2) (m-2-4) edge (m-2-5)
(m-1-2) edge (m-1-3) 
(m-2-2) edge (m-2-3)
(m-1-3) edge (m-1-4)
(m-2-3) edge (m-2-4)
;
\path[-,font=\scriptsize]
(m-1-4) edge [double, thick, double distance=2pt] (m-2-4)
;
\end{tikzpicture}
\]
Thus $\overline{f}_{ij}([\mathbb{E}_i]) := [\mathbb{E}_j]$. Similarly, the canonical morphisms $\phi_i$ define (via pushouts again) group homomorphisms $\overline{\phi}_i \colon {\rm Ext}^1_{\mathcal{G}}(F,X_i) \to {\rm Ext}^1_{\mathcal{G}}(F,X)$. We also have the following commutative diagram:

\[
\begin{tikzpicture}[description/.style={fill=white,inner sep=2pt}]
\matrix (m) [matrix of math nodes, row sep=2.5em, column sep=0.5em, text height=1.25ex, text depth=0.25ex]
{ 
{\rm Ext}^1_{\mathcal{G}}(F,X_i) & {} & {\rm Ext}^1_{\mathcal{G}}(F,X_j) \\
{} & {\rm Ext}^1_{\mathcal{G}}(F,X) & {} \\
};
\path[->]
(m-1-1) edge node[above] {\footnotesize$\overline{f}_{ij}$} (m-1-3) edge node[left] {\footnotesize$\overline{\phi}_i\mbox{ \ }$} (m-2-2)
(m-1-3) edge node[right] {\footnotesize$\mbox{ \ \ \ }\overline{\phi}_j$} (m-2-2)
;
\end{tikzpicture}
\]
We can note from the previous comments that $\{ {\rm Ext}^1_{\mathcal{G}}(F,X_i) \mbox{ : } i \in I \}$ is a direct system. The universal property of direct limits induces a canonical homomorphism $\xi_1 \colon \varinjlim_I {\rm Ext}^1_{\mathcal{G}}(F,X_i) \to {\rm Ext}^1_{\mathcal{G}}(F,X)$ as shown in the following diagram:
\begin{equation}\label{diag1} 
\parbox{3in}{
\begin{tikzpicture}[description/.style={fill=white,inner sep=2pt}]
\matrix (m) [matrix of math nodes, row sep=3em, ampersand replacement=\&, column sep=0.5em, text height=1.25ex, text depth=0.25ex]
{ 
{\rm Ext}^1_{\mathcal{G}}(F,X_i) \& {} \& {\rm Ext}^1_{\mathcal{G}}(F,X_j) \\
{} \& \displaystyle\operatorname*{\varinjlim}_I {\rm Ext}^1_{\mathcal{G}}(F,X_i) \\
{} \& {\rm Ext}^1_{\mathcal{G}}(F,X) \& {} \\
};
\path[->]
(m-1-1) edge node[above] {\footnotesize$\overline{f}_{ij}$} (m-1-3) edge [bend right=30] node[below,sloped] {\footnotesize$\overline{\phi}_i$} (m-3-2)
(m-1-3) edge [bend left=30] node[below,sloped] {\footnotesize$\overline{\phi}_j$} (m-3-2)
;
\path[right hook->]
(m-1-1) edge node[below,sloped] {\footnotesize$\lambda_i$} (m-2-2)
;
\path[left hook->]
(m-1-3) edge node[below,sloped] {\footnotesize$\lambda_j$} (m-2-2)
;
\path[dotted,->]
(m-2-2) edge node[description] {\footnotesize$\exists! \mbox{ } \xi_1$} (m-3-2)
;
\end{tikzpicture}
}
\end{equation} 
We now show that $\xi_1$ is a monomorphism. Consider an element $([\mathbb{E}_i])_{i \in I}$ in the direct limit group $\varinjlim_I {\rm Ext}^1_{\mathcal{G}}(F,X_i)$ that is mapped to $[0]$ via $\xi_1$. Let $\mathbb{E}$ be a representative of $\xi_1(([\mathbb{E}_i])_{i \in I})$. Using properties of direct limits, \cite[Lemma 2.6.14]{weibel},  we can find an index $i_0 \in I$ and a short exact sequence
\[
\mathbb{E}_{i_0} \colon 0 \to X_{i_0} \to Y_{i_0} \to F \to 0,
\] 
such that $\phi_{i_0} \mathbb{E}_{i_0} \sim \mathbb{E}$, that is, $\lambda_{i_0}([\mathbb{E}_{i_0}]) = ([\mathbb{E}_i])_{i \in I}$. Here, $\phi_{i_0} \mathbb{E}_{i_0}$ denotes the exact sequence obtained after taking the pushout of the morphisms $\varinjlim_I X_i \xleftarrow{\phi_{i_0}} X_{i_0} \to Y_{i_0}$. 

Now  consider the set $I_{i_0} = \{ j \in I : i_{0} \leq j \}$, and so for each $j \in I_{i_0}$ we have $\mathbb{E}_j$ obtained as the pushout of $X_j \xleftarrow{f_{i_0, j}} X_{i_0} \to Y_{i_0}$. Hence from the properties of the pushout, this last short exact sequence $\mathbb{E}_j$ also maps to $\mathbb{E}$. In terms of short exact sequences, we have the following diagram:
\[
\begin{tikzpicture}
\matrix (m) [matrix of math nodes, row sep=2em, column sep=0.75em, text height=1.5ex, text depth=0.25ex]
{
\mathbb{E}_{i_0} \colon 0 & {} & X_{i_0} & {} & Y_{i_0} & {} & F & {} & 0 \\
{} & \mathbb{E}_j \colon 0 & {} & X_j & {} & Y_j & {} & F & {} & 0 \\
\mathbb{E} \colon 0 & {} & \displaystyle\operatorname*{\varinjlim}_I X_i & {} & Y & {} & F & {} & 0 \\
};
\path[->]
(m-1-3) edge node[above,sloped] {\footnotesize$f_{i_0,j}$} (m-2-4)
(m-1-5) edge (m-2-6)
(m-2-6) edge (m-3-5)
(m-1-5) edge (m-3-5) (m-1-7) edge (m-3-7)
(m-1-3) edge (m-1-5)
(m-2-4) edge [-,line width=6pt,draw=white] (m-2-6)
(m-2-4) edge (m-2-6)
(m-3-3) edge (m-3-5)
(m-1-1) edge (m-1-3) (m-1-7) edge (m-1-9) (m-2-8) edge (m-2-10)
(m-3-1) edge (m-3-3) (m-3-7) edge (m-3-9)
(m-1-3) edge node[pos=0.25,left] {\footnotesize$\phi_{i_0}$} (m-3-3)
(m-2-4) edge node[right] {\footnotesize$\phi_j$} (m-3-3)
;
\path[->]
(m-2-2) edge [-,line width=6pt,draw=white] (m-2-4)
(m-2-2) edge (m-2-4)
(m-1-5) edge (m-1-7)
(m-2-6) edge [-,line width=6pt,draw=white] (m-2-8)
(m-2-6) edge (m-2-8)
(m-3-5) edge (m-3-7)
;
\path[-,font=\scriptsize]
(m-1-7) edge [double, thick, double distance=2pt] (m-2-8)
(m-3-7) edge [double, thick, double distance=2pt] (m-2-8)
;
\end{tikzpicture}
\]
It is important to note that the set $I_{i_0}$ is \emph{cofinal} in $I$. This means that for each $i \in I$, there is $j \in I_{i_0}$ such that $i \leq j$. Indeed, just consider the element $j$ given in the directed set $I$ with the property that $i \leq j$ and $i_0 \leq j$. Then by \cite[Exercise V. 5.22(i)]{rotman} we have $\varinjlim_{I} \mathbb{E}_i \simeq \varinjlim_{I_{i_0}} \mathbb{E}_j$. The properties of direct limits give us a new short exact sequence 
\[
\overline{\mathbb{E}} \colon 0 \to \varinjlim_{I_{i_o}} X_j \to \varinjlim_{I_{i_o}} Y_j \to F \to 0,
\] 
along with a unique homomorphism $\overline{\mathbb{E}} \to \mathbb{E}$, where the ends of these short exact sequences are isomorphic, and hence so the middle arrow $\varinjlim_{I_0} Y_{j} \to Y$ is an isomorphism. Thus, we have $[\mathbb{E}] = [\overline{\mathbb{E}}]$ in ${\rm Ext}^1_{\mathcal{G}}(F,\varinjlim_{I} X_i)$. On the other hand, $\mathbb{E} \sim 0$, and thus we have that $\overline{\mathbb{E}}$ splits, which gives us a morphism $h \colon F \to \varinjlim_{I_{i_0}} Y_j$. Since $F$ is finitely presented, $h$ can be factored through some $Y_l$ with $l \in I_{i_0}$, that is,
\[
(F \xrightarrow{h} \varinjlim_{I_{i_0}} Y_j) = (F \xrightarrow{h_l} Y_l \to \varinjlim_{I_{i_0}} Y_j)
\] 
Hence we have the following commutative diagram:  
\[
\begin{tikzpicture}[description/.style={fill=white,inner sep=2pt}]
\matrix (m) [matrix of math nodes, row sep=2.5em, column sep=2.5em, text height=1.25ex, text depth=0.25ex]
{ 
\mathbb{E}_l \colon 0 & X_l & Y_l & F & 0 \\
\overline{\mathbb{E}} \colon 0 & \displaystyle\operatorname*{\varinjlim}_{I_{i_0}} X_j & \displaystyle\operatorname*{\varinjlim}_{I_{i_0}} Y_j & F & 0 \\
};
\path[->]
(m-1-2)-- node[pos=0.5] {\footnotesize$\mbox{\bf po}$} (m-2-3)
(m-1-3) edge (m-2-3)
(m-2-4) edge [bend right=30] node[above,sloped] {\footnotesize$h$} (m-2-3)
(m-1-4) edge [bend right=30] node[above,sloped] {\footnotesize$h_l$} (m-1-3)
(m-1-2) edge node[left] {\footnotesize$\phi_l$} (m-2-2) 
(m-1-2) edge (m-1-3) 
(m-2-2) edge (m-2-3)
(m-1-3) edge (m-1-4)
(m-2-3) edge (m-2-4)
(m-1-1) edge (m-1-2) (m-1-4) edge (m-1-5)
(m-2-1) edge (m-2-2) (m-2-4) edge (m-2-5)
;
\path[-,font=\scriptsize]
(m-1-4) edge [double, thick, double distance=2pt] (m-2-4)
;
\end{tikzpicture}
\] 
We have $\mathbb{E}_l \sim 0$, and so $\overline{f}_{i_0,l}([\mathbb{E}_{i_0}]) = [\mathbb{E}_l] = [0]$. Using the diagram \eqref{diag1}, we have that $\lambda_{i_0}([\mathbb{E}_{i_0}]) = \lambda_l \circ \overline{f}_{i_0,l}([\mathbb{E}_{i_0}]) = [0]$, and so $([\mathbb{E}_i])_{i \in I} = [0]$ since $\lambda_{i_0}([\mathbb{E}_{i_0}]) = ([\mathbb{E}_i])_{i \in I}$. Therefore, every preimage of $[\mathbb{E}] = [0]$ under $\xi_1$ is $[0]$, that is, $\xi_1$ is a monomorphism.


\subsection*{Case $\bm{n > 1}$} 

This case will require strongly the assumption that $\mathcal{G}$ is locally finitely presented. Let us fix $n > 1$ and assume that the result holds for every integer $1 \leq m < n$. First, note that the group homomorphisms  $\overline{f}_{ij} \colon {\rm Ext}^n_{\mathcal{G}}(F,X_i) \to {\rm Ext}^n_{\mathcal{G}}(F,X_j)$, $\overline{\phi}_i \colon {\rm Ext}^n_{\mathcal{G}}(F,X_i) \to {\rm Ext}^n_{\mathcal{G}}(F,X)$ and $\xi_n \colon \varinjlim_I {\rm Ext}^n_{\mathcal{G}}(F,X_i) \to {\rm Ext}^n_{\mathcal{G}}(F,X)$ are constructed as in the case $n = 1$. 

Let $([\mathbb{E}_i])_{i \in I} \in \varinjlim_I {\rm Ext}^n_{\mathcal{G}}(F,X_i)$ such that $\xi_n(([\mathbb{E}_i])_{i \in I}) = [0]$, and let 
\[
\mathbb{E} = 0 \to \varinjlim_I X_i \to A^n \xrightarrow{f_n} A^{n-1} \xrightarrow{f_{n-1}} \cdots \to A^1 \xrightarrow{f_1} F \to 0
\] 
be a representative of $\xi_n(([\mathbb{E}_i])_{i \in I})$. Consider the following splicers of $\mathbb{E}$:
\begin{align*}
\mathbb{E}' \colon & 0 \to \varinjlim_I X_i \to A^n \to K \to 0, \\
\mathbb{E}'' \colon & 0 \to K \to A^{n-1} \to \cdots \to A^1 \to F \to 0,
\end{align*}
where $K = {\rm Ker}(f_{n-1})$, that is, $\mathbb{E}$ is obtained by ``gluing'' $\mathbb{E}'$ and $\mathbb{E}''$ at $K$. We denote this gluing operation as $\mathbb{E}' \mathbb{E}''$. Then, $\mathbb{E}' \mathbb{E}'' = \mathbb{E} \sim 0$. By \cite[Lemma VII.4.1]{mitchell}, there exists an $(n-1)$-fold exact sequence 
\[
\mathbb{H} \colon 0 \to W \to B^{n-1} \to \cdots \to B^1 \to F \to 0
\] 
and a morphism $\psi \colon W \to K$ such that $\mathbb{E}'' \sim \psi \mathbb{H}$ and $\mathbb{E}' \psi \sim 0$, that is, $\mathbb{E}' \psi$ splits. Here, $\mathbb{E}' \psi$ denotes the short exact sequence obtained after taking the pullback of $A^n \to K \leftarrow W$. Let us write
\begin{align*}
\mathbb{E}' \psi \colon & 0 \to \varinjlim_{I} X_i \to B^n \to W \to 0.
\end{align*} 
Thus, we can replace $\mathbb{E} = \mathbb{E'} \mathbb{E}'' \sim (\mathbb{E}' \psi) \mathbb{H}$, which amounts to say that we can choose $\mathbb{E}$ as a sequence with a $1$-fold splicer on the left which is split, as indicated in the following diagram: 

\[
\begin{tikzpicture}[description/.style={fill=white,inner sep=2pt}]
\matrix (m) [matrix of math nodes, row sep=2em, column sep=1em, text height=1.25ex, text depth=0.25ex]
{ 
0 \sim \mathbb{E}' \psi \colon 0 & \displaystyle\operatorname*{\varinjlim}_I X_i & B^n & W & 0 \\
{} & \mathbb{E} \colon 0 & \displaystyle\operatorname*{\varinjlim}_I X_i & B^n & B^{n-1} & \cdots & B^1 & F & 0 \\
{} & {} & {} & \mathbb{H} \colon 0 & W & B^{n-1} & \cdots & B^1 & F & 0 \\
};
\path[->]
(m-1-1) edge (m-1-2) (m-1-2) edge (m-1-3) (m-1-3) edge (m-1-4) (m-1-4) edge (m-1-5)
(m-2-2) edge (m-2-3) (m-2-3) edge (m-2-4) (m-2-4) edge (m-2-5) (m-2-5) edge (m-2-6) (m-2-6) edge (m-2-7) (m-2-7) edge (m-2-8) (m-2-8) edge (m-2-9)
(m-3-4) edge (m-3-5) (m-3-5) edge (m-3-6) (m-3-6) edge (m-3-7) (m-3-7) edge (m-3-8) (m-3-8) edge (m-3-9) (m-3-9) edge (m-3-10)
;
\path[>->]
(m-1-4) edge (m-2-5)
;
\path[->>]
(m-2-4) edge (m-3-5)
;
\path[-,font=\scriptsize]
(m-1-2) edge [double, thick, double distance=2pt] (m-2-3)
(m-1-3) edge [double, thick, double distance=2pt] (m-2-4)
(m-2-5) edge [double, thick, double distance=2pt] (m-3-6)
(m-2-7) edge [double, thick, double distance=2pt] (m-3-8)
(m-2-8) edge [double, thick, double distance=2pt] (m-3-9)
;
\end{tikzpicture}
\]
As we did in the case $n = 1$, we can assert the existence of some $i_0 \in I$ and an $n$-fold exact sequence 
\[
\mathbb{E}_{i_0} \colon 0 \to  X_{i_0} \to C^n \to B^{n-1} \to \cdots \to B^1 \to F \to 0,
\]
that is, $[\mathbb{E}_{i_0}] \in {\rm Ext}^n_{\mathcal{G}}(F,X_{i_0})$, such that, $\phi_{i_0} \mathbb{E}_{i_0} \sim \mathbb{E}$. We have the following diagram:
\begin{equation}\label{diagextra} 
\parbox{5in}{
\begin{tikzpicture}[description/.style={fill=white,inner sep=2pt}]
\matrix (m) [matrix of math nodes, ampersand replacement=\&, row sep=1.5em, column sep=1.5em, text height=1.25ex, text depth=0.25ex]
{ 
\mathbb{E}_{i_0} \colon 0 \& X_{i_0} \& C^n \& {} \& B^{n-1} \& \cdots \& B^1 \& F \& 0 \\
{} \& {} \& {} \& W \\
\mathbb{E} \colon 0 \& \displaystyle\operatorname*{\varinjlim}_I X_i \& B^n \& {} \& B^{n-1} \& \cdots \& B^1 \& F \& 0 \\
{} \& {} \& {} \& W \\
};
\path[->]
(m-1-2)-- node[pos=0.5] {\footnotesize$\mbox{\bf po}$} (m-3-3)
(m-1-1) edge (m-1-2) (m-1-2) edge (m-1-3) (m-1-3) edge (m-1-5) (m-1-5) edge (m-1-6) (m-1-6) edge (m-1-7) (m-1-7) edge (m-1-8) (m-1-8) edge (m-1-9)
(m-3-1) edge (m-3-2) (m-3-2) edge (m-3-3) (m-3-3) edge (m-3-5) (m-3-5) edge (m-3-6) (m-3-6) edge (m-3-7) (m-3-7) edge (m-3-8) (m-3-8) edge (m-3-9)
(m-1-3) edge (m-3-3)
(m-1-2) edge node[left] {\footnotesize$\phi_{i_0}$} (m-3-2) 
;
\path[>->]
(m-2-4) edge (m-1-5)
(m-4-4) edge (m-3-5)
;
\path[->>]
(m-1-3) edge (m-2-4)
(m-3-3) edge (m-4-4)
;
\path[-,font=\scriptsize]
(m-2-4) edge [-,line width=6pt,draw=white] (m-4-4)
(m-2-4) edge [double, thick, double distance=2pt] (m-4-4)
(m-1-5) edge [double, thick, double distance=2pt] (m-3-5)
(m-1-7) edge [double, thick, double distance=2pt] (m-3-7)
(m-1-8) edge [double, thick, double distance=2pt] (m-3-8)
;
\end{tikzpicture}
}
\end{equation}

We now use the assumption that $\mathcal{G}$ is locally finitely presented. This allows us to write $W \simeq \varinjlim_{T} W_t$ for some directed set $T$. For this new direct limit, we fix the notation $\sigma_{tt'} \colon W_t \to W_{t'}$ for the compatible morphisms with $t \leq t'$, and $\psi_t \colon W_t \to \varinjlim_{T} W_t$ for the canonical morphisms. Then, we have the $(n-1)$-fold exact sequence
\[
\mathbb{H} \colon 0 \to \varinjlim_{T} W_t \to B^{n-1} \to \cdots \to B^1 \to F \to 0
\]
and $[\mathbb{H}] \in {\rm Ext}^{n-1}_{\mathcal{G}}(F,\varinjlim_{T} W_t)$, which is a splicer of $\mathbb{E}_{i_0}$. By the induction hypothesis, we know that ${\rm Ext}^{n-1}_{\mathcal{G}}(F,\varinjlim_{T} W_t) \cong \varinjlim_{T} {\rm Ext}^{n-1}_{\mathcal{G}}(F,W_t)$, and so there exists $t_0 \in T$ and $[\mathbb{H}_{t_0}] \in {\rm Ext}^{n-1}_{\mathcal{G}}(F,W_{t_0})$ such that $\psi_{t_0} \mathbb{H}_{t_0} \sim \mathbb{H}$. Let us write the previous equality as the following commutative diagram:
\begin{equation}\label{diag2} 
\parbox{5.25in}{
\begin{tikzpicture}[description/.style={fill=white,inner sep=2pt}]
\matrix (m) [matrix of math nodes, ampersand replacement=\&, row sep=2.5em, column sep=2em, text height=1.25ex, text depth=0.25ex]
{ 
\mathbb{H}_{t_0} \colon 0 \& W_{t_0} \& D^{n-1}_{t_0} \& B^{n-2} \& \cdots \& B^1 \& F \& 0 \\
\mathbb{H} \colon 0 \& \varinjlim_{T} W_t \& B^{n-1} \& B^{n-2} \& \cdots \& B^1 \& F \& 0 \\
};
\path[->]
(m-1-2)-- node[pos=0.5] {\footnotesize$\mbox{\bf po}$} (m-2-3)
(m-1-1) edge (m-1-2) (m-1-2) edge (m-1-3) (m-1-3) edge (m-1-4) (m-1-4) edge (m-1-5) (m-1-5) edge (m-1-6) (m-1-6) edge (m-1-7) (m-1-7) edge (m-1-8)
(m-2-1) edge (m-2-2) (m-2-2) edge (m-2-3) (m-2-3) edge (m-2-4) (m-2-4) edge (m-2-5) (m-2-5) edge (m-2-6) (m-2-6) edge (m-2-7) (m-2-7) edge (m-2-8)
(m-1-3) edge (m-2-3)
(m-1-2) edge node[left] {\footnotesize$\psi_{t_0}$} (m-2-2)
;
\path[-,font=\scriptsize]
(m-1-4) edge [double, thick, double distance=2pt] (m-2-4)
(m-1-6) edge [double, thick, double distance=2pt] (m-2-6)
(m-1-7) edge [double, thick, double distance=2pt] (m-2-7)
;
\end{tikzpicture}
}
\end{equation}
After combining \eqref{diagextra} and \eqref{diag2}, and taking the pullback of $C^n \to \varinjlim_T W_t \leftarrow W_{t_0}$ we obtain the following commutative diagram:
\[
\footnotesize
\begin{tikzpicture}[description/.style={fill=white,inner sep=2pt}]
\matrix (m) [matrix of math nodes, row sep=2em, column sep=1.5em, text height=1.25ex, text depth=0.25ex]
{ 
\overline{\mathbb{E}}_{i_0} \colon 0 & X_{i_0} & D^n_{i_0} & {} & D^{n-1}_{t_0} & {} & B^{n-2} & \cdots & B^1 & F & 0 \\
{} & {} & {} & W_{t_0} & {} & \bullet & {} & {} & {} & {} \\
\mathbb{E}_{i_0} \colon 0 & X_{i_0} & C^n & {} & B^{n-1} & {} & B^{n-2} & \cdots & B^1 & F & 0 \\
{} & {} & {} & \displaystyle\operatorname*{\varinjlim}_T W_t & {} & \bullet & {} & {} & {} \\
\mathbb{E} \colon 0 & \displaystyle\operatorname*{\varinjlim}_I X_i & B^n & {} & B^{n-1} & {} & B^{n-2} & \cdots & B^1 & F & 0 \\
{} & {} & {} & \displaystyle\operatorname*{\varinjlim}_T W_t & {} & {} & {} & {} \\
};
\path[->]
(m-3-2)-- node[pos=0.5] {\footnotesize$\mbox{\bf po}$} (m-5-3)
(m-1-3)-- node[pos=0.45] {\footnotesize$\mbox{\bf pb}$} (m-4-4)
(m-3-1) edge (m-3-2) (m-3-2) edge (m-3-3) (m-3-3) edge (m-3-5) (m-3-5) edge (m-3-7) (m-3-7) edge (m-3-8) (m-3-8) edge (m-3-9) (m-3-9) edge (m-3-10)
(m-5-1) edge (m-5-2) (m-5-2) edge (m-5-3) (m-5-3) edge (m-5-5) (m-5-5) edge (m-5-7) (m-5-7) edge (m-5-8) (m-5-8) edge (m-5-9) (m-5-9) edge (m-5-10)
(m-3-3) edge (m-5-3)
(m-1-5) edge (m-1-7) (m-1-7) edge (m-1-8) (m-1-8) edge (m-1-9) (m-1-9) edge (m-1-10) (m-1-10) edge (m-1-11)
(m-3-10) edge (m-3-11)
(m-5-10) edge (m-5-11)
(m-1-5) edge (m-3-5)
;
\path[dashed,->]
(m-1-1) edge (m-1-2) (m-1-2) edge (m-1-3) (m-1-3) edge (m-1-5)
(m-1-3) edge (m-3-3)
;
\path[dashed,->>]
(m-1-3) edge (m-2-4)
;
\path[>->]
(m-2-4) edge (m-1-5)
(m-2-6) edge (m-1-7)
(m-4-4) edge (m-3-5)
(m-6-4) edge (m-5-5)
(m-4-6) edge (m-3-7)
;
\path[->>]
(m-3-5) edge (m-4-6)
(m-1-5) edge (m-2-6)
(m-3-3) edge (m-4-4)
(m-5-3) edge (m-6-4)
;
\path[-,font=\scriptsize]
(m-4-4) edge [-,line width=6pt,draw=white] (m-6-4)
(m-4-4) edge [double, thick, double distance=2pt] (m-6-4)
(m-3-5) edge [double, thick, double distance=2pt] (m-5-5)
(m-3-7) edge [double, thick, double distance=2pt] (m-5-7)
(m-3-9) edge [double, thick, double distance=2pt] (m-5-9)
(m-3-10) edge [double, thick, double distance=2pt] (m-5-10)
(m-2-6) edge [-,line width=6pt,draw=white] (m-4-6)
(m-2-6) edge [double, thick, double distance=2pt] (m-4-6)
(m-2-4) edge [-,line width=6pt,draw=white] (m-4-4)
(m-1-7) edge [double, thick, double distance=2pt] (m-3-7)
(m-1-9) edge [double, thick, double distance=2pt] (m-3-9)
(m-1-10) edge [double, thick, double distance=2pt] (m-3-10)
;
\path[dashed,font=\scriptsize]
(m-1-2) edge [double, thick, double distance=2pt] (m-3-2)
;
\path[->]
(m-3-2) edge node[left] {\footnotesize$\phi_{i_0}$} (m-5-2) 
(m-2-4) edge node[pos=0.25,right] {\footnotesize$\psi_{t_0}$} (m-4-4)
;
\end{tikzpicture}
\]
We note from the previous diagram that $\mathbb{E}_{i_0} \sim \overline{\mathbb{E}}_{i_0}$ and $[\mathbb{E}] = \xi_n([\overline{\mathbb{E}}_{i_0}])$. The rest of the proof focuses on showing that $\overline{\mathbb{E}}_{i_0} \sim 0$. 

From $\mathbb{E}$ consider the 1-fold splicer 
\[
0 \to \varinjlim_I X_i \to B^n \to \varinjlim_T W_t \to 0,
\] 
which was obtained as the pushout of $\varinjlim_{I} X_i \xleftarrow{\phi_{i_0}} X_{i_0} \to C^n$, where $X_{i_0} \to C^n$ is the morphism appearing in the $1$-fold splicer 
\[
\mathbb{E}^1_{i_0} \colon 0 \to X_{i_0} \to C^n \to \varinjlim_{T} W_t \to 0
\] 
of $\mathbb{E}_{i_0}$. Take the pullback of $B^n \to W \leftarrow W_{t_0}$ to get the following commutative diagram:
\[
\begin{tikzpicture}[description/.style={fill=white,inner sep=2pt}]
\matrix (m) [matrix of math nodes, row sep=2em, column sep=1em, text height=1.25ex, text depth=0.25ex]
{ 
{} & 0 & {} & \varinjlim_I X_i & {} & Z & {} & W_{t_0} & {} & 0 \\
\mathbb{E}^1_{i_0} \colon 0 & {} & X_{i_0} & {} & C^n & {} & \varinjlim_T W_t & {} & 0 \\
{} & 0 & {} & \varinjlim_I X_i & {} & B^n & {} & \varinjlim_T W_t & {} & 0 \\
};
\path[->]
(m-1-2) edge (m-1-4) (m-1-4) edge (m-1-6) (m-1-6) edge (m-1-8) (m-1-8) edge (m-1-10)
(m-2-1) edge (m-2-3) (m-2-3) edge (m-2-5) (m-2-5) edge (m-2-7) (m-2-7) edge (m-2-9)
(m-3-2) edge (m-3-4) (m-3-4) edge (m-3-6) (m-3-6) edge (m-3-8) (m-3-8) edge (m-3-10)
(m-2-5) edge (m-3-6)
(m-2-3) edge node[left] {\footnotesize$\phi_{i_0}$} (m-3-4) 
;
\path[-,font=\scriptsize]
(m-1-4) edge [-,line width=6pt,draw=white] (m-3-4)
(m-1-4) edge [double, thick, double distance=2pt] (m-3-4)
(m-2-7) edge [double, thick, double distance=2pt] (m-3-8)
;
\path[->]
(m-1-6) edge [-,line width=6pt,draw=white] (m-3-6)
(m-1-6) edge (m-3-6)
;
\path[->]
(m-1-8) edge [-,line width=6pt,draw=white] (m-3-8)
(m-1-8) edge node[pos=0.25,right] {\footnotesize$\psi_{t_0}$} (m-3-8)
;
\end{tikzpicture}
\]
Since $(\phi_{i_0} \mathbb{E}^1_{i_0}) \psi_{t_0} \sim \phi_{i_0} (\mathbb{E}^1_{i_0} \psi_{t_0})$, taking the pullback of $C^n \to \lim_T W_t \leftarrow W_{t_0}$ produces the following commutative diagram:
\begin{equation}\label{diag3} 
\parbox{5in}{
\footnotesize
\begin{tikzpicture}[description/.style={fill=white,inner sep=2pt}]
\matrix (m) [matrix of math nodes, ampersand replacement=\&, row sep=2.5em, column sep=0.5em, text height=1.25ex, text depth=0.25ex]
{ 
\mathbb{G}_{i_0} = \mathbb{E}^1_{i_0} \psi_{t_0} \colon 0 \& {} \& X_{i_0} \& {} \& D^n_{i_0} \& {} \& W_{t_0} \& {} \& 0  \\
{} \& \mathbb{G} \colon 0 \& {} \& \varinjlim_I X_i \& {} \& Z \& {} \& W_{t_0} \& {} \& 0 \\
\mathbb{E}^1_{i_0} \colon 0 \& {} \& X_{i_0} \& {} \& C^n \& {} \& \varinjlim_T W_t \& {} \& 0 \\
{} \& \phi_{i_0} \mathbb{E}^1_{i_0} \colon 0 \& {} \& \varinjlim_I X_i \& {} \& B^n \& {} \& \varinjlim_T W_t \& {} \& 0 \\
};
\path[->]
(m-1-1) edge (m-1-3) (m-1-3) edge (m-1-5) (m-1-5) edge (m-1-7) (m-1-7) edge (m-1-9)
(m-3-1) edge (m-3-3) (m-3-3) edge (m-3-5) (m-3-5) edge (m-3-7) (m-3-7) edge (m-3-9)
(m-4-2) edge (m-4-4) (m-4-4) edge (m-4-6) (m-4-6) edge (m-4-8) (m-4-8) edge (m-4-10)
(m-3-5) edge (m-4-6)
(m-1-5) edge (m-3-5) edge (m-2-6)
;
\path[->]
(m-1-3) edge node[right] {\footnotesize$\phi_{i_0}$} (m-2-4) 
(m-3-3) edge node[left] {\footnotesize$\phi_{i_0}$} (m-4-4) 
(m-1-7) edge node[pos=0.25,right] {\footnotesize$\psi_{t_0}$} (m-3-7)
;
\path[-,font=\scriptsize]
(m-1-3) edge [double, thick, double distance=2pt] (m-3-3)
(m-2-4) edge [-,line width=6pt,draw=white] (m-4-4)
(m-2-4) edge [double, thick, double distance=2pt] (m-4-4)
(m-1-7) edge [double, thick, double distance=2pt] (m-2-8)
(m-3-7) edge [double, thick, double distance=2pt] (m-4-8)
;
\path[->]
(m-2-6) edge [-,line width=6pt,draw=white] (m-4-6)
(m-2-6) edge (m-4-6)
(m-2-2) edge [-,line width=6pt,draw=white] (m-2-4)
(m-2-2) edge (m-2-4) 
(m-2-4) edge [-,line width=6pt,draw=white] (m-2-6)
(m-2-4) edge (m-2-6) 
(m-2-6) edge [-,line width=6pt,draw=white] (m-2-8)
(m-2-6) edge (m-2-8) 
(m-2-8) edge [-,line width=6pt,draw=white] (m-2-10)
(m-2-8) edge (m-2-10)
;
\path[->]
(m-2-8) edge [-,line width=6pt,draw=white] (m-4-8)
(m-2-8) edge node[pos=0.25,right] {\footnotesize$\psi_{t_0}$} (m-4-8)
;
\end{tikzpicture}
}
\end{equation}
where $\mathbb{G} = (\phi_{i_0} \mathbb{E}^1_{i_0}) \psi_{t_0} \sim \phi_{i_0} (\mathbb{E}^1_{i_0} \psi_{t_0})$. 

For each $j \geq i_0$ we can similarly form a short exact sequence
\[
\mathbb{G}_j \colon 0 \to X_j \to D^n_j \to W_{t_0} \to 0.
\]
The family of sequences $\{ \mathbb{G}_j \mbox{ : } j \geq i_0 \}$ is a direct system over the cofinal set $I_{i_0}$: recall that for $j \geq i_0$ we have morphisms 
\[
\gamma^1_j \colon {\rm Ext}^1_{\mathcal{G}}(W_{t_0},X_j) \to {\rm Ext}^1_{\mathcal{G}}\left(W_{t_0},\varinjlim_{I_0} X_i\right)
\] 
which are compatible with respect to the morphisms 
\[
\hat{f}^1_{ij} \colon {\rm Ext}^1_{\mathcal{G}}(W_{t_0},X_i) \to {\rm Ext}^1_{\mathcal{G}}(W_{t_0},X_j),
\] 
that is, $\gamma^1_i = \gamma^1_j \circ \hat{f}^1_{ij}$ for every $i_0 \leq i \leq j$. Thus, we can consider the direct limit $\varinjlim_{I_0} \mathbb{G}_j$, which is again a short exact sequence since $\mathcal{G}$ is a Grothendieck category. By the universal property of colimits, we have the following diagram, which is an isomorphism of short exact sequences:
\[
\begin{tikzpicture}[description/.style={fill=white,inner sep=2pt}]
\matrix (m) [matrix of math nodes, row sep=2.5em, column sep=2.5em, text height=1.25ex, text depth=0.25ex]
{ 
\displaystyle\operatorname*{\varinjlim}_{I_0} \mathbb{G}_j \colon 0 & \displaystyle\operatorname*{\varinjlim}_{I_0} X_j & \displaystyle\operatorname*{\varinjlim}_{I_0} D^n_j & W_{t_0} & 0 \\
\mathbb{G} \colon 0 & \displaystyle\operatorname*{\varinjlim}_{I} X_j & Z & W_{t_0} & 0 \\
};
\path[->]
(m-1-1) edge (m-1-2) (m-1-2) edge (m-1-3) (m-1-3) edge node[above] {\footnotesize$p$} (m-1-4) (m-1-4) edge (m-1-5)
(m-2-1) edge (m-2-2) (m-2-2) edge (m-2-3) (m-2-3) edge (m-2-4) (m-2-4) edge (m-2-5)
(m-1-3) edge (m-2-3)
;
\path[-,font=\scriptsize]
(m-1-2) edge [double, thick, double distance=2pt] (m-2-2)
(m-1-4) edge [double, thick, double distance=2pt] (m-2-4)
;
\end{tikzpicture}
\]
Note that $\phi_{i_0} \mathbb{E}^1_{i_0} \sim \mathbb{E}' \psi \sim 0$, and so the sequence $\phi_{i_0} \mathbb{E}^1_{i_0}$ splits. By the homotopy lemma, the sequence $\mathbb{G}$ also splits. Then, the upper face of the diagram \eqref{diag3} represents a situation similar to the case $n = 1$. This implies that we can find an arrow $q \colon W_{t_0} \to \varinjlim_{I_0} D^n_j$ that is a right inverse of $p$ ($p \circ q = {\rm id}_{W_{t_0}}$) and factors through some $D^n_{j'}$ with $j' \geq i_0$, since $W_{t_0}$ is finitely presented. Thus, we have the following diagram between split short exact sequences:
\[
\begin{tikzpicture}[description/.style={fill=white,inner sep=2pt}]
\matrix (m) [matrix of math nodes, row sep=2.5em, column sep=2.5em, text height=1.25ex, text depth=0.25ex]
{ 
\mathbb{G}_{j'} \colon 0 & X_{j'} & D^n_{j'} & W_{t_0} & 0 \\
\displaystyle\operatorname*{\varinjlim}_{I_0} \mathbb{G}_j \colon 0 & \displaystyle\operatorname*{\varinjlim}_{I_0} X_j & \displaystyle\operatorname*{\varinjlim}_{I_0} D^n_j & W_{t_0} & 0 \\
};
\path[->]
(m-1-2)-- node[pos=0.5] {\footnotesize$\mbox{\bf po}$} (m-2-3)
(m-1-1) edge (m-1-2) (m-1-2) edge (m-1-3) (m-1-3) edge (m-1-4) (m-1-4) edge (m-1-5)
(m-2-1) edge (m-2-2) (m-2-2) edge (m-2-3) (m-2-3) edge (m-2-4) (m-2-4) edge (m-2-5)
(m-1-3) edge (m-2-3) 
(m-1-2) edge node[left] {\footnotesize$\phi_{j'}$} (m-2-2)
;
\path[-,font=\scriptsize]
(m-1-4) edge [double, thick, double distance=2pt] (m-2-4)
;
\end{tikzpicture}
\]
We now splice the sequences $\mathbb{G}_{j'}$ and 
\[
\mathbb{G}^{n-1}_{t_0} \colon 0 \to W_{t_0} \to D^{n-1}_{t_0} \to B^{n-2} \to \cdots \to B^1 \to F \to 0.
\]
We have a commutative diagram
\[
\footnotesize
\begin{tikzpicture}[description/.style={fill=white,inner sep=2pt}]
\matrix (m) [matrix of math nodes, row sep=3em, column sep=1.25em, text height=1.5ex, text depth=0.5ex]
{ 
\mathbb{G}_{j'} \colon 0 & X_{j'} & D^n_{j'} & W_{t_0} & 0 \\
{} & \mathbb{G}_{j'} \mathbb{G}^{n-1}_{t_0} \colon 0 & X_{j'} & D^n_{j'} & D^{n-1}_{t_0} & B^{n-2} & \cdots & B^1 & F & 0 \\
{} & {} & {} & \mathbb{G}^{n-1}_{t_0} \colon 0 & W_{t_0} & D^{n-1}_{t_0} & B^{n-2} & \cdots & B^1 & F & 0 \\
};
\path[->]
(m-1-1) edge (m-1-2) (m-1-2) edge (m-1-3) (m-1-3) edge (m-1-4) (m-1-4) edge (m-1-5)
(m-2-2) edge (m-2-3) (m-2-3) edge (m-2-4) (m-2-4) edge (m-2-5) (m-2-5) edge (m-2-6) (m-2-6) edge (m-2-7) (m-2-7) edge (m-2-8) (m-2-8) edge (m-2-9) (m-2-9) edge (m-2-10)
(m-3-4) edge (m-3-5) (m-3-5) edge (m-3-6) (m-3-6) edge (m-3-7) (m-3-7) edge (m-3-8) (m-3-8) edge (m-3-9) (m-3-9) edge (m-3-10) (m-3-10) edge (m-3-11)
;
\path[>->]
(m-1-4) edge (m-2-5)
;
\path[->>]
(m-2-4) edge (m-3-5)
;
\path[-,font=\scriptsize]
(m-1-2) edge [double, thick, double distance=2pt] (m-2-3)
(m-1-3) edge [double, thick, double distance=2pt] (m-2-4)
(m-2-5) edge [double, thick, double distance=2pt] (m-3-6)
(m-2-8) edge [double, thick, double distance=2pt] (m-3-9)
(m-2-9) edge [double, thick, double distance=2pt] (m-3-10)
;
\end{tikzpicture}
\]
along with a morphism $\overline{\mathbb{E}}_{i_0} \to \mathbb{G}_{j'} \mathbb{G}^{n-1}_{t_0}$:
\[
\footnotesize
\begin{tikzpicture}[description/.style={fill=white,inner sep=2pt}]
\matrix (m) [matrix of math nodes, row sep=3.5em, column sep=2em, text height=1.25ex, text depth=0.25ex]
{ 
\overline{\mathbb{E}}_{i_0} \colon 0 & X_{i_0} & D^n_{i_0} & D^{n-1}_{t_0} & B^{n-2} & \cdots & B^1 & F & 0 \\
\mathbb{G}_{j'} \mathbb{G}^{n-1}_{t_0} \colon 0 & X_{j'} & D^n_{j'} & D^{n-1}_{t_0} & B^{n-2} & \cdots & B^1 & F & 0 \\
};
\path[->]
(m-1-2)-- node[pos=0.5] {\footnotesize$\mbox{\bf po}$} (m-2-3)
(m-1-1) edge (m-1-2) (m-1-2) edge (m-1-3) (m-1-3) edge (m-1-4) (m-1-4) edge (m-1-5) (m-1-5) edge (m-1-6) (m-1-6) edge (m-1-7) (m-1-7) edge (m-1-8) (m-1-8) edge (m-1-9)
(m-2-1) edge (m-2-2) (m-2-2) edge (m-2-3) (m-2-3) edge (m-2-4) (m-2-4) edge (m-2-5) (m-2-5) edge (m-2-6) (m-2-6) edge (m-2-7) (m-2-7) edge (m-2-8) (m-2-8) edge (m-2-9)
(m-1-2) edge (m-2-2) (m-1-3) edge (m-2-3)
;
\path[-,font=\scriptsize]
(m-1-4) edge [double, thick, double distance=2pt] (m-2-4)
(m-1-5) edge [double, thick, double distance=2pt] (m-2-5)
(m-1-7) edge [double, thick, double distance=2pt] (m-2-7)
(m-1-8) edge [double, thick, double distance=2pt] (m-2-8)
;
\end{tikzpicture}
\]
By \cite[Lemma 4.1]{mitchell}, $f_{i_0,j'} \overline{E}_{i_0} \sim \mathbb{G}_{j'} \mathbb{G}^{n-1}_{t_0} \sim 0$. On the other hand, $f_{i_0,j'} \overline{\mathbb{E}}_{i_0} \sim f_{i_0,j'} \mathbb{E}_{i_0}$ since $\overline{\mathbb{E}}_{i_0} \sim \mathbb{E}_{i_0}$. Thus, we have $\overline{f}_{i_0,j'}([\mathbb{E}_{i_0}]) = [0]$. Now consider the morphisms $\lambda_j \colon {\rm Ext}^n_{\mathcal{G}}(F,X_j) \to \varinjlim_{I_0} {\rm Ext}^n_{\mathcal{G}}(F,X_j)$ in the diagram \eqref{diag1}. We have $\lambda_{i_0}([\mathbb{E}_{i_0}]) = \lambda_{j'} \circ \overline{f}_{i_0,j'}([\mathbb{E}_{i_0}]) = [0]$, and thus we can assert that $\mathbb{E}_{i_0} \sim 0$. This concludes the result.


\section{Some finiteness conditions for quasi-coherent sheaves}

In this second appendix we complement the study of finiteness conditions in the category $\mathfrak{Qcoh}(X)$ of quasi-coherent sheaves over a scheme $X$. We also give examples of schemes $X$ such that $\mathfrak{Qcoh}(X)$ is an $n$-coherent category.


\subsection*{Finitely presented quasi-coherent sheaves in terms of finitely generated and $\text{FP}_1$-injective quasi-coherent sheaves}

We show that Proposition \ref{FPn-in-terms-of-FPn-Inj} holds in the case $n = 1$ for the category of quasi-coherent sheaves over certain schemes. Indeed, we have already mentioned that this result is true in the category of modules over a ring (see \cite[Theorem 2.1.10]{Glaz}) and also in the category of complexes of modules. The latter follows by using the characterization of complexes in $\mathcal{FP}_0$, $\mathcal{FP}_1$ and $\mathcal{FP}_1\mbox{-}\textrm{Inj}$. For, suppose we are given a finitely generated complex $F$ (that is, bounded and with finitely generated module entries \cite[Proposition 2.1.4]{ZhaoPerez}) such that ${\rm Ext}^1_{\textrm{Ch}}(F,X) = 0$ for every complex $X \in \mathcal{FP}_1\mbox{-}\textrm{Inj}$. Since $\mathcal{FP}_1\mbox{-}\textrm{Inj}$ is formed by exact complexes with $\text{FP}_1$-injective cycles by \cite[Theorem 2.3.3]{ZhaoPerez}, the previous holds for every complex of the form $D^{m+1}(M)$ with $m \in \mathbb{Z}$ and $M$ an $\text{FP}_1$-injective module. That is, 
\[
0 = {\rm Ext}^1_{\textrm{Ch}}(F,D^{m+1}(M)) \cong {\rm Ext}^1_R(F_m, M),
\] 
using the natural isomorphism described in \cite[Lemma 3.1]{gillespie}. Thus, $F_m$ is a finitely generated $R$-module which is also left Ext-orthogonal to every $\text{FP}_1$-injective module. Hence, $F_m$ is a finitely presented $R$-module by \cite[Theorem 2.1.10]{Glaz}. In other words, we have that $F$ is a finitely presented complex. 

Under certain assumptions on a scheme $X$, we are also able to extend the equality $\mathcal{FP}_1 = \mathcal{FP}_{0} \cap {}^{\perp_1}(\mathcal{FP}_1\mbox{-}\textrm{Inj})$ to the category $\mathfrak{Qcoh}(X)$ of quasi-coherent sheaves over $X$. Specifically, we need $X$ to be a semi-separated scheme, that is, $X$ has an open affine covering $\{ U_i \}_{i \in I}$ such that $U_i \cap U_j$ is also an open affine for every $i, j \in I$. Now for each open affine $U \subseteq X$, consider the inclusion $\iota^U \colon U \hookrightarrow X$ and the induced direct image functor $\iota^U_\ast \colon \mathfrak{Qcoh}(U) \longrightarrow \mathfrak{Qcoh}(X)$ (the direct image functor preserves quasi-coherency since $X$ is semi-separated). By \cite[Corollary 5.5]{hartshorne}, we have a natural isomorphism 
\begin{align}\label{eqn:nat_iso_hom}
\textrm{Hom}_{\mathcal{O}_X(U)}(\mathscr{H}(U),E) & \cong {\rm Hom}_{\mathfrak{Qcoh}(X)}(\mathscr{H},\iota^U_\ast(E))
\end{align}
for every $\mathscr{H} \in \mathfrak{Qcoh}(X)$ and $E \in \mathcal{O}_X(U)\mbox{-}\textrm{Mod}$. Using \eqref{eqn:nat_iso_hom}, we can note that $E$ is an injective module over $\mathcal{O}_X(U)$ if, and only if, $\iota^U_\ast(E)$ is an injective quasi-coherent sheaf over $X$. Thus, we can obtain the following natural isomorphism for every $k \geq 0$:
\begin{align}\label{eqn:nat_iso_ext}
\textrm{Ext}^k_{\mathcal{O}_X(U)}(\mathscr{H}(U),E) & \cong {\rm Ext}^k_{\mathfrak{Qcoh}(X)}(\mathscr{H},\iota^U_\ast(E)).
\end{align}

Let us prove that the equality $\mathcal{FP}_1 = \mathcal{FP}_0 \cap {}^{\perp_1}(\mathcal{FP}_1\mbox{-}\textrm{Inj})$ holds true in $\mathfrak{Qcoh}(X)$. We shall need the following result, which is a consequence of \cite[Lemma 3.6 and Proposition 3.7]{AbsPureSheaves}.

\begin{lemma}
Let $X$ be a semi-separated scheme. Then, $\mathscr{F}$ is a finitely presented quasi-coherent sheaf if, and only if, $\mathscr{F}(U)$ is a finitely presented $\mathcal{O}_X(U)$-module, for every open affine $U \subseteq X$. 
\end{lemma}

\begin{proposition}
Let $X$ be a semi-separated scheme. Then, $\mathscr{F} \in \mathfrak{Qcoh}(X)$ is finitely presented if, and only if, $\mathscr{F}$ is finitely generated and ${\rm Ext}^1_{\mathfrak{Qcoh}(X)}(\mathscr{F},\mathscr{E}) = 0$ for every $\text{FP}_1$-injective quasi-coherent sheaf $\mathscr{E}$ over $X$.
\end{proposition}

\begin{proof}
The ``only if'' part is clear. Now suppose that $\mathscr{F}$ is a finitely generated quasi-coherent sheaf such that ${\rm Ext}^1_{\mathfrak{Qcoh}(X)}(\mathscr{F},\mathscr{E}) = 0$ for every $\mathscr{E} \in \mathfrak{Qcoh}(X)$ $\text{FP}_1$-injective. We show that $\mathscr{F}(U)$ is a finitely presented $\mathcal{O}_X(U)$-module for every open affine $U \subseteq X$. The result will follow by the previous lemma. 

First, notice that the direct image functor $\iota^U_\ast$ associated to $U$ preserves direct limits since $X$ is semi-separated. In particular, $\iota^U_\ast$ preserves direct unions, and so $\textrm{Hom}_{\mathfrak{Qcoh}(X)}(\mathscr{F},\iota^U_\ast(-))$ preserves direct unions since $\mathscr{F}$ is finitely generated. By \eqref{eqn:nat_iso_hom}, we have that $\textrm{Hom}_{\mathcal{O}_X(U)}(\mathscr{F}(U),-)$ preserves direct unions in $\mathcal{O}_X(U)\mbox{-}\textrm{Mod}$, that is, $\mathscr{F}(U)$ is finitely generated for every open affine $U \subseteq X$. 

Now consider an $\text{FP}_1$-injective $\mathcal{O}_X(U)$-module $E$. By the previous lemma and \eqref{eqn:nat_iso_ext}, we can note that $\iota^U_\ast(E)$ is an $\text{FP}_1$-injective quasi-coherent sheaf over $X$. By the assumption on $\mathscr{F}$, we have that ${\rm Ext}^1_{\mathfrak{Qcoh}(X)}(\mathscr{F},\iota^U_\ast(E)) = 0$. Using \eqref{eqn:nat_iso_ext} again, we have that $\textrm{Ext}^1_{\mathcal{O}_X(U)}(\mathscr{F}(U),E) = 0$ for every $\text{FP}_1$-injective $E$. That is, $\mathscr{F}(U) \in \mathcal{FP}_0 \cap {}^{\perp_1}(\mathcal{FP}_1\mbox{-}\textrm{Inj})$ in $\mathcal{O}_X(U)\mbox{-}\textrm{Mod}$ for every open affine $U \subseteq X$, and hence by \cite[Theorem 2.1.10]{Glaz} can conclude that $\mathscr{F}(U)$ is a finitely presented $\mathcal{O}_X(U)$-module. 
\end{proof}


\subsection*{$\bm{n}$-coherent categories}

We now study some conditions for certain schemes $X$ under which $\mathfrak{Qcoh}(X)$ is an $n$-coherent category. 

\begin{proposition}\label{prop:n-coherent_scheme}
Let $X$ be a semi-separated scheme with a semi-separated affine open cover 
\[
\{ U_1 \simeq {\rm Spec}(A_1), \dots, U_m \simeq {\rm Spec}(A_m) \}\footnote{Here, the notation $U_i \simeq {\rm Spec}(A_i)$ means that $U_i$ and ${\rm Spec}(A_i)$ are isomorphic as locally ringed spaces.} 
\]
such that each $A_i$ is a commutative $n$-coherent ring. Then, every quasi-coherent sheaf over $X$ of type $\text{FP}_n$ is of type $\text{FP}_\infty$. In particular, if $X$ is a \emph{coherent scheme} in the sense of \cite[Definition 6.8]{EstradaGillespie-scheme}, then $\mathfrak{Qcoh}(X)$ is a coherent category. 
\end{proposition}

\begin{proof}
Let $\mathscr{F} \in \mathfrak{Qcoh}(X)$ be of type $\text{FP}_n$. We show that $\textrm{Ext}^k_{\mathfrak{Qcoh}(X)}(\mathscr{F},-)$ preserves direct limits for every $k \geq 0$. Let $1 \leq i \leq m$ and consider the inclusion $\iota^{U_i} \colon U_i \hookrightarrow X$. By \eqref{eqn:nat_iso_ext}, Definition \ref{def-finitely-n-presented} and the fact that $\iota^{U_i}_\ast$ preserves direct limits (since $X$ is semi-separated), we have that the functor $\textrm{Ext}^k_{A_i}(\mathscr{F}(U_i),-)$ preserves direct limits for every $0 \leq k \leq n-1$, that is, $\mathscr{F}(U_i)$ is an $A_i$-module of type $\text{FP}_n$. Since $A_i$ is an $n$-coherent ring, $A_i\mbox{-}\textrm{Mod}$ is an $n$-coherent category, and so $\mathscr{F}(U_i)$ is an $A_i$-module of type $\text{FP}_\infty$, meaning that $\textrm{Ext}^k_{A_i}(\mathscr{F}(U_i),-)$ preserves direct limits for every $k \geq 0$. 

We now use the previous to show that $\mathscr{F}$ is of type $\text{FP}_\infty$. From $\{ U_1, \dots, U_m \}$, it is possible to construct a semi-separated affine basis $\mathcal{B} = \{ V_\alpha \}_{\alpha \in \Lambda}$ formed by those open affine subsets $V_\alpha \subseteq X$ contained in some $U_i$. We show that each $\mathscr{F}(V_\alpha)$ is an $\mathcal{O}_X(V_\alpha)$-module of type $\text{FP}_\infty$. Suppose $V_\alpha$ is contained in some $U_{i(\alpha)}$, and consider the inclusion $j \colon V_\alpha \hookrightarrow U_{i(\alpha)}$. Since $V_\alpha$ is affine, we have a natural isomorphism
\[
{\rm Ext}^k_{\mathcal{O}_X(V_\alpha)}(\mathscr{F}(V_\alpha),-) \cong {\rm Ext}^k_{\mathfrak{Qcoh}(U_{i(\alpha)})}(\mathscr{F}|_{U_{i(\alpha)}},j_\ast(-))
\]
where $j_\ast \colon \mathfrak{Qcoh}(V_\alpha) \longrightarrow \mathfrak{Qcoh}(U_{i(\alpha)})$ preserves direct limits. Since $\mathscr{F}(U_{i(\alpha)})$ is an $\mathcal{O}_X(U_{i(\alpha)})$-module of type $\text{FP}_\infty$ and $\mathcal{O}_X(U_{i(\alpha)})\mbox{-}\textrm{Mod}$ is equivalent to $\mathfrak{Qcoh}(U_{i(\alpha)})$, we have that $\mathscr{F}|_{U_{i(\alpha)}}$ is a quasi-coherent sheaf over $U_{i(\alpha)}$ of type $\text{FP}_\infty$, that is, ${\rm Ext}^k_{\mathfrak{Qcoh}(U_{i(\alpha)})}(\mathscr{F}|_{U_{i(\alpha)}},j_\ast(-))$ preserves direct limits for every $k \geq 0$. It follows that ${\rm Ext}^k_{\mathcal{O}_X(V_\alpha)}(\mathscr{F}(V_\alpha),-)$ preserves direct limits for every $k \geq 0$, and thus $\mathscr{F}(V_\alpha)$ is of type $\text{FP}_\infty$ for every $V_\alpha$ in the semi-separated affine basis $\mathcal{B}$. Applying \cite[Lemma 2.2 and Proposition 2.3]{Estrada-Gillespie}, we have that $\mathscr{F} \in \mathfrak{Qcoh}(X)$ is of type $\text{FP}_\infty$. 

For the last assertion, recall that a scheme $X$ with structure sheaf $\mathcal{O}_X$ is coherent if it is quasi-compact and $\mathcal{O}_X(U)$ is a commutative coherent ring for every open affine $U \subseteq X$. This implies that there exists an open affine finite cover $\{ U_1, \dots, U_m \}$ of $X$ such that $\mathcal{O}_X(U_i)$ is a coherent ring, since being \emph{locally} coherent as a scheme is a Zariski-local notion due to Christensen et al. \cite[Proposition 3.7]{CEI-Zariski}.
\end{proof}

\begin{remark}
The previous proof could suggest a notion of ``locally $n$-coherent schemes'' for $n \geq 2$ (that is, a scheme $(X,\mathcal{O}_X)$ such that $\mathcal{O}_X(U)$ is a commutative $n$-coherent ring for every open affine $U \subseteq X$), and restate the assumption in Proposition \ref{prop:n-coherent_scheme} in terms of such schemes. The problem with this is that we are not aware if the property of being $n$-coherent in Zariski-local in the class of commutative rings. In this context, one could try to show this as a consequence of \cite[Lemma 3.6]{CEI-Zariski}. Among other things, we would need to check that the property of being $n$-coherent is \emph{compatible with finite products}, meaning that for all commutative rings $A_1$ and $A_2$, the product ring $A_1 \times A_2$ is $n$-coherent if, and only if, $A_1$ and $A_2$ are $n$-coherent. This is not even known to be true for the case $A_1 = A_2 = A$. In other words, we have the open question: 
\begin{align}\label{ref:open}
\text{$A$ is a commutative $n$-coherent ring} & \Longrightarrow \text{$A^2$ is $n$-coherent?}
\end{align}
This is in turn related to the characterization of $n$-coherent rings in terms their ideals, namely, the following conditions are known to be equivalent (see Dobbs et al. \cite[Remark 3.5]{DKM_n-coherent}):
\begin{enumerate}
\item[(a)] A commutative ring $A$ is $n$-coherent if, and only if, it is \emph{weak} $n$-coherent (that is, each ideal of $A$ of type $\text{FP}_{n-1}$ is of type $\text{FP}_n$).\footnote{$n$-Coherent and weak $n$-coherent rings are called ``strong $n$-coherent'' and ``$n$-coherent'' in \cite{DKM_n-coherent}, respectively.}

\item[(b)] $A$ is weak $n$-coherent $\Longrightarrow$ $A^m$ is weak $n$-coherent for every $m \geq 1$. 
\end{enumerate}
Sufficient conditions under which the equivalence in (a) holds are given in \cite[Proposition 3.3]{DKMS}. Summarizing, we are not even aware if the product ring $A^2$ in \eqref{ref:open} is $n$-coherent in the weak sense. Hence, we have preferred to state the hypothesis in Proposition \ref{prop:n-coherent_scheme} in terms of having a scheme with a semi-separated affine open cover $\{ U_1 \simeq {\rm Spec}(A_1), \dots, U_m \simeq {\rm Spec}(A_m) \}$ and with each ring $A_i$ $n$-coherent. Next in Corollary \ref{coro:projective_line_n-coherent} and Example \ref{ex:n-coherent_scheme} we exhibit some of such schemes. 
\end{remark}

\begin{corollary}\label{coro:projective_line_n-coherent}
Let $A$ be a commutative ring and $X = \mathbb{P}^1(A)$ be the projective line over $A$. If the polynomial ring $A[x]$ is $n$-coherent, then $\mathfrak{Qcoh}(X)$ is an $n$-coherent category.
\end{corollary}

\begin{proof}
By \cite[Corollary 2.5]{Estrada-Gillespie}, we know that $\mathfrak{Qcoh}(X)$ is a locally type $\text{FP}_\infty$ category with semi-separated cover $\{ U_0, U_1 \}$, where 
\begin{align*}
U_0 & := D_+(x_0) = {\rm Spec}(A[x_0 / x_1]) & & \text{and} & U_1 & := D_+(x_1) = {\rm Spec}(A[x_1 / x_0]) & \text{(see \cite[Section 3.6]{GW_algebraic_geometry_I}).}
\end{align*} 
By hypothesis, we know that both $A[x_0 / x_1]$ and $A[x_1 / x_0]$ are $n$-coherent rings. So it follows by the previous proposition that the equality $\mathcal{FP}_n = \mathcal{FP}_\infty$ holds in $\mathfrak{Qcoh}(X)$. The result then follows. 
\end{proof}

\begin{example}\label{ex:n-coherent_scheme} 
The case $n = 0$ in Corollary \ref{coro:projective_line_n-coherent} yields a reformulation of Hilbert's basis theorem in terms of locally noetherian categories. Namely, if $A$ is a noetherian commutative ring, then so is $A[x]$, and hence $\mathfrak{Qcoh}(\mathbb{P}^1(A))$ is a locally noetherian category. 

For the case $n \geq 1$, there is no guarantee that $A[x]$ turns out to be an $n$-coherent ring if $A$ is $n$-coherent. For instance, if we set $n = 1$ we have the notion of \emph{stably coherent rings}: those coherent rings $A$ such that every polynomial ring $A[x_1, \dots, x_m]$ is also coherent for $m \geq 1$ (see for instance Glaz's book \cite[Section 7.3]{Glaz}). Not every coherent ring is stably coherent. In \cite[Section 7.3.13]{Glaz} the author constructs a commutative coherent ring of weak dimension 2 which is not stably coherent, while in \cite[Section Theorem 7.3.14]{Glaz} it is proved that every commutative coherent ring of global dimension 2 is stably coherent. 

We bring from the literature another example of a ring satisfying the condition of Corollary \ref{coro:projective_line_n-coherent}. For, recall that a ring $R$ is called an \emph{$(n,d)$-ring} (where $n$ and $d$ are nonnegative integers) if every $R$-module of type $\text{FP}_n$ has projective dimension at most $d$. Consider the commutative ring $A$ presented in Vasconcelos' \cite[Example 1.3 (b)]{Vasconcelos}. This is a noncoherent ring with weak dimension 1, which is also $(2,1)$-coherent. By Costa's \cite[Theorem 2.2]{costa}, $A$ is $2$-coherent. On the other hand, in \cite[Example 8.16]{Vasconcelos} it is proven that $A[x]$ is a $(2,2)$-ring, and again by \cite[Theorem 2.2]{costa} we have that $A[x]$ is $2$-coherent. 
\end{example}


\section{Some finiteness conditions for functor categories}

In this last appendix we explain in more detail some of the already mentioned facts and examples in the context of categories of additive functors. Some of the results below are well known, and restated and reproved within the terminology developed in the previous sections, but we also provide slightly more general statements for them.  

First, recall that in Example \ref{ex:FPn-injectives} (3) that an additive functor $G \colon \mathcal{C}^{\rm op} \longrightarrow \mathsf{Ab}$ is $\text{FP}_1$-injective if, and only if, $G$ maps pseudo-kernels in $\mathcal{C}$ into pseudo-cokernels in $\mathsf{Ab}$, provided that $\mathcal{C}$ has pseudo-kernels. This is a consequence of the following result due to Auslander \cite[part b) of Theorem 2.2]{auslander-coherent-functors}.

\begin{proposition}\label{prop:functor_category_1-coherent}
Let $\mathcal{C}$ be a skeletally small additive category. The following two conditions hold for any $n \geq 1$.
\begin{enumerate}
\item If $\mathcal{C}$ has kernels, then every object of type $\text{FP}_n$ has projective dimension at most $2$.

\item If $\mathcal{C}$ has pseudo-kernels, then every object of type $\text{FP}_n$ is of type $\text{FP}_\infty$.  
\end{enumerate}
\end{proposition}

\begin{proof}
Part (1) is proved in \cite[Chapter 1]{Dean}. A similar argument can be applied to show (2). Namely, suppose we are given an object $F \in \textrm{Fun}(\mathcal{C}^{\rm op},\mathsf{Ab})$ of type $\text{FP}_n$ with $n \geq 1$. So by Example \ref{Example-fg-projectives} there is an exact sequence in $\text{Fun}(\mathcal{C}^{\rm op},\mathsf{Ab})$:
\begin{align}\label{eqn:Yoneda}
(-,X_n) \xrightarrow{\alpha_n} (-,X_{n-1}) \to \cdots \to (-,X_1) \xrightarrow{\alpha_1} (-,X_0) \to F \to 0.
\end{align}
By the Yoneda Lemma, the natural transformations $\alpha_i$, with $1 \leq i \leq n$, can be represented by a morphisms $f_i \colon X_i \to X_{i-1}$ in $\mathcal{C}$ in the sense that $\alpha_i = \textrm{Hom}_{\mathcal{C}}(-,f_i)$. Let $f_{n+1} \colon X_{n+1} \to X_n$ be a pseudo-kernel of $f_n$. Then, we can note that the induced sequence
\[
(-,X_{n+1}) \xrightarrow{(-,f_{n+1})} (-,X_n) \xrightarrow{(-,f_n)} (-,X_{n-1})
\]
is exact. Glueing this sequence to \eqref{eqn:Yoneda} yields a finite $(n+1)$-presentation of $F$:
\[
(-,X_{n+1}) \to (-,X_n) \to \cdots \to (-,X_1) \to (-,X_0) \to F \to 0
\]
that is, the containment $\mathcal{FP}_n \subseteq \mathcal{FP}_{n+1}$ holds in the functor category $\textrm{Fun}(\mathcal{C}^{\rm op},\mathsf{Ab})$. The same argument repeated infinitely many times shows that $F$ is of type $\text{FP}_\infty$.
\end{proof}

The previous result along with Theorem \ref{them-locally n-coherent} (e) and the fact mentioned in Remark \ref{rem:locally_type_FPinfty} that $\textrm{Fun}(\mathcal{C}^{\rm op},\mathsf{Ab})$ has a generating set formed by finitely generated projective objects, imply the following result (see Auslander's \cite[Theorem 2.2]{auslander-coherent-functors} and Herzog's \cite[Dual of Proposition 2.1]{herzog-ziegler-spectrum}).

\begin{corollary}
Let $\mathcal{C}$ be a skeletally small additive category with kernels or pseudo-kernels. Then, $\textrm{Fun}(\mathcal{C}^{\rm op},\mathsf{Ab})$ is a $1$-coherent category. 
\end{corollary}

For the rest of this section we focus on characterizing $\text{FP}_2$-injective objects in $\textrm{Fun}(\mathcal{C}^{\rm op},\mathsf{Ab})$, following the spirit of Auslander's \cite[Section 4]{auslander-coherent-functors}. 

\begin{proposition}
Let $\mathcal{C}$ be a skeletally small additive category and $G \colon \mathcal{C}^{\rm op} \longrightarrow \mathsf{Ab}$ be an additive functor. The following conditions are equivalent.
\begin{itemize}
\item[(a)] $G$ is $\text{FP}_2$-injective.

\item[(b)] For every sequence 
\[
X_2 \xrightarrow{f_2} X_1 \xrightarrow{f_1} X_0
\] 
in $\mathcal{C}$ such that $f_2$ is a pseudo-kernel of $f_1$, the sequence 
\[
G X_0 \xrightarrow{G f_1} G X_1 \xrightarrow{G f_2} G X_2
\] 
is exact in $\mathsf{Ab}$.
\end{itemize}
\end{proposition}

\begin{proof}
Suppose first that $G$ is $\text{FP}_2$-injective. Let $f_1 \colon X_1 \to X_0$ be a morphism in $\mathcal{C}$ for which there exists a pseudo-kernel, let us say $f_2 \colon X_2 \to X_1$. Then, the sequence in $\textrm{Fun}(\mathcal{C}^{\rm op},\mathsf{Ab})$
\[
(-,X_2) \xrightarrow{(-,f_2)} (-,X_1) \xrightarrow{(-,f_1)} (-,X_0).
\]
is exact at $(-,X_1)$. If we let $F$ denote the cokernel object of $\textrm{Hom}_{\mathcal{C}}(-,f_1)$, we have that $F$ is of type $\text{FP}_2$. Since $G$ is $\text{FP}_2$-injective, that is $\textrm{Ext}^1_{\mathcal{\textrm{Fun}(\mathcal{C}^{\rm op},\mathsf{Ab})}}(F,G) = 0$, the resulting sequence
\begin{align}\label{eqn:G-hom}
((-,X_0), G) & \xrightarrow{((-,f_1),G)} ((-,X_1),G) \xrightarrow{((-,f_2),G)} ((-,X_2),G)
\end{align}
is exact at $((-,X_1),G)$. On the other hand, by the Yoneda Lemma we have that the previous sequence is naturally isomorphic to the sequence
\begin{align}\label{eqn:G-exact}
G X_0 & \xrightarrow{G f_1} G X_1 \xrightarrow{G f_2} G X_2,
\end{align}
which must be exact at $G X_1$. 

Assume now that condition (b) holds true. Let $F \in \textrm{Fun}(\mathcal{C}^{\rm op},\mathsf{Ab})$ be an object of type $\text{FP}_2$. Then, as in the proof of Proposition \ref{prop:functor_category_1-coherent}, there exists an exact sequence 
\[
(-,X_2) \xrightarrow{(-,f_2)} (-,X_1) \xrightarrow{(-,f_1)} (-,X_0) \to F \to 0.
\] 
Now consider the functor $G$ satisfying (b) and the corresponding induced sequence \eqref{eqn:G-hom}, which is isomorphic to \eqref{eqn:G-exact}. Since \eqref{eqn:G-exact} is exact at $G X_1$, \eqref{eqn:G-hom} is exact at $((-,X_1),G)$, and hence ${\rm Ext}^1_{\textrm{Fun}(\mathcal{C}^{\rm op},\mathsf{Ab})}(F,G) = 0$. Since $F$ is arbitrary, we can conclude that $G$ is $\text{FP}_2$-injective. 
\end{proof}

\begin{remark}
A similar characterization can be stated for $\text{FP}_n$-injective objects considering pseudo $n$-kernels, a concept motivated in \cite[Definition 2.2]{Jasso}. Let $f_1 \colon X_1 \to X_0$ be a morphism in an additive category $\mathcal{C}$. We say that a sequence 
\[
X_n \xrightarrow{f_n} X_{n-1} \xrightarrow{f_{n-2}} X_{n-2} \to \cdots \to X_2 \xrightarrow{f_2} X_1 
\]
is a \emph{pseudo $(n-1)$-kernel} of $f_1$ if for every $C \in \mathcal{C}$ the exact sequence of abelian groups 
\[
(C,X_n) \xrightarrow{(C,f_n)} (C,X_{n-1}) \xrightarrow{(C,f_{n-2})} (C,X_{n-2}) \to \cdots \to (C,X_2) \xrightarrow{(C,f_2)} (C,X_1) \xrightarrow{(C,f_1)} (C,X_0) 
\]
is exact.

In the case where $\mathcal{C}$ is a skeletally small additive category, and $G \colon \mathcal{C}^{\rm op} \longrightarrow \mathsf{Ab}$ is an additive functor, one can show that the following are equivalent for any $n \geq 2$.
\begin{itemize}
\item[(a)] $G$ is $\text{FP}_n$-injective.

\item[(b)] For every sequence 
\[
X_n \xrightarrow{f_n} X_{n-1} \to \cdots X_2 \xrightarrow{f_2} X_1 \xrightarrow{f_1} X_0
\] 
in $\mathcal{C}$ such that $(f_n,\dots,f_2)$ is a pseudo $n-1$-kernel of $f_1$, the sequence 
\[
G X_0 \xrightarrow{G f_1} G X_1 \xrightarrow{G f_2} G X_2 \to \cdots \to G X_{n-1} \xrightarrow{G f_n} G X_n
\] 
of abelian groups is exact at $G X_1$. 
\end{itemize}
\end{remark}


\section*{Acknowledgements}

Special thanks to Sinem Odaba\c si (Universidad Austral de Chile) whose ideas and explanations help us to prove the case $n > 1$ of Lemma \ref{lem:mono_condition}.


\section*{Funding}

The first author was partially funded by CONICYT/FONDECYT/REGULAR/1180888

The third author is supported by postdoctoral felloship from the Comisi\'on Acad\'emica de Posgrado - Universidad de la Rep\'ublica.


\bibliographystyle{plain}
\bibliography{biblio_n-coherent}

%
%
%
%
%
%
%
%
%
%
%
%
%
%
\end{document}